\documentclass[dvipsnames]{article}
\usepackage[utf8]{inputenc}
\usepackage{amsmath}
\usepackage{amsthm}
\usepackage{amssymb}
\usepackage{textcomp}
\usepackage{ifthen}
\usepackage{tikz}
\usepackage{xkeyval}
\usepackage{calc}
\usepackage{xcolor}
\usepackage{authblk}
\usepackage{amsfonts}
\usepackage[toc,page]{appendix}
\usepackage{graphicx}
\usepackage{mathrsfs}
\usepackage{hyperref}
\usepackage{todonotes}
\usepackage{verbatim}
\usepackage{dsfont}
\usepackage{ulem}
\usepackage[left=2.5cm,top=3cm,right=2.5cm,bottom=3cm,bindingoffset=0.5cm]{geometry}
\newcommand{\RR}{\mathbb{R}}
\newcommand{\NN}{\mathbb{N}}
\newcommand{\Z}{\mathbb{Z}}
\newcommand{\CC}{\mathbb{C}}
\newcommand{\norm}{\|\cdot\|}
\newcommand{\test}{\mathcal{E}}
\newcommand{\dist}{\mathcal{E}^\prime}

\newcommand{\vertiii}[1]{{\vert\kern-0.25ex\vert\kern-0.25ex\vert #1
    \vert\kern-0.25ex\vert\kern-0.25ex\vert}}

\newtheorem{prpstn}{Proposition}[section]
\newtheorem{lmm}{Lemma}[section]
\newtheorem{thrm}{Theorem}[section]
\newtheorem{dfntn}{Definition}[section]
\newtheorem{crllr}{Corollary}[section]
\newtheorem{rmrk}{Remark}[section]

\title{Stabilization of the linearized water tank system}
\author[a,e]{Jean-Michel Coron}
\author[b,e]{Amaury Hayat}
\author[c]{Shengquan Xiang}
\author[d]{Christophe Zhang}
\affil[a]{Sorbonne Universit\'{e}, Universit\'{e} de Paris, CNRS, Laboratoire Jacques-Louis Lions, 75005 Paris, France. Email address: jean-michel.coron@sorbonne-universite.fr.}
\affil[b]{CERMICS, \'{E}cole des Ponts ParisTech, 77455 Marne la Vall\'{e}e, France. Email address: amaury.hayat@enpc.fr.}
\affil[c]{B\^{a}timent des Math\'{e}matiques, EPFL, CH-1015 Lausanne, Switzerland.
E-mail address: shengquan.xiang@epfl.ch.}
\affil[d]{Chair in Applied Analysis (Alexander von Humboldt Professorship), Department Mathematik, Friedrich-Alexander-Universit\"{a}t Erlangen-N\"{u}rnberg, 91058 Erlangen, Germany. Email address: christophe.zhang@fau.de.}
\affil[e]{CAGE, INRIA, Paris, France.}

\date{\empty}
\begin{document}
\maketitle

\begin{abstract}
      In this article we study the so-called water tank system. In this system, the behavior of water contained in a 1-D tank is modelled by Saint-Venant equations, with a scalar distributed control. It is well-known that the linearized systems around uniform steady-states are not controllable, the uncontrollable part being of infinite dimension. Here we will focus on the linearized systems around non-uniform steady states, corresponding to a constant acceleration of the tank. We prove that these systems are controllable in Sobolev spaces, using the moments method and perturbative spectral estimates. Then, for steady states corresponding to small enough accelerations, we design an explicit Proportional Integral feedback law (obtained thanks to a well-chosen dynamic extension of the system) that stabilizes these systems exponentially with arbitrarily large decay rate. Our design relies on feedback equivalence/backstepping.
\end{abstract}

\smallskip
\noindent \textbf{Keywords.} Saint-Venant equations, controllability, feedback equivalence, backstepping, rapid stabilization.

\noindent \textbf{AMS Subject Classification.}
35L04,  
93D20,   
93D15.
\tableofcontents

\section{Introduction}
\subsection{Equations of the problem}\label{s1}
We consider the water tank model as introduced in \cite{DPR1999}, in one space dimension, under the hypothesis that the depth of the water is small compared to the length of the tank, and that the acceleration of the tank is small compared to the gravitational constant. The behavior of the water inside the tank, in the frame of reference of the tank, is then modelled by the Saint-Venant equation with no friction and no slope,
\begin{equation}\label{sv}
\left\{\begin{aligned}
 &\partial_{t}H+\partial_{x}(HV)=0,\\
 &\partial_{t}V+V\partial_{x}V+g\partial_{x}H=-U(t).
 \end{aligned}\right.
\end{equation}
where $H$ is the height of the water, $V$ its averaged horizontal velocity and $U$ is control input, given by the acceleration of the tank. Without loss of generality we can suppose that $g=1$.
The water is localized inside the water tank, which implies the following Dirichlet boundary conditions:
\begin{equation}\label{bc-sv0}
 V(t,0)=V(t,L)=0.
\end{equation}
Moreover, integrating the first equation of \eqref{sv} we obtainthe conservation of the mass of the water, we have that
\begin{gather}\label{cond-conser-mass}
    \int_0^L H(t, x) dx \textrm{ does not change with respect to time.}
\end{gather}
First derived in 1871 by Barr\'{e} de Saint-Venant \cite{ABSV1,ABSV3,ABSV2}, the Saint-Venant equations are among the most famous equations
in fluid dynamics and represent flows under shallow water approximation. Despite their apparent simplicity, they capture a large
number of physical behaviors, which made them a ground tool for practical application
in particular in the regulation of canals for agriculture management and in the regulation of navigable rivers.

The Saint-Venant equations are an example of quasilinear hyperbolic systems. The stabilization of such systems by proportional or output boundary control have been studied for decades, one can cite for instance the pioneering work of Li and Greenberg \cite{greenberg} for a system of two homogeneous equations considered in the framework of the $C^{1}$ norm. This was later generalized by \cite{Qin,Zhao,1994-Li-book,deHalleux,2010-Li-Rao-Wang-DCDS,2019-C1-H} to general nonhomogeneous systems and \cite{BastinCoron1D} in the framework of the $H^{2}$ norm.
The first result concerning the boundary stabilization of the Saint-Venant equations in themselves goes back to 1999
with \cite{Coron1999} where the stability of the homogeneous linearized Saint-Venant was shown, using proportional boundary conditions. This was extended in \cite{CAB2007} to the nonlinear homogeneous Saint-Venant equations.
Later, in 2008, using a semigroup approach and the method of the characteristics, the stabilization of the nonlinear homogeneous equations was achieved for sufficiently small friction and slope \cite{dos2008,Bastin2008}.
The same type of result was shown in \cite{CBA2008} using a Lyapunov approach while \cite{bastin2009} dealt with the inhomogeneous Saint-Venant equations in the particular case where the steady-states are uniform.
In 2017, the stabilization was achieved for arbitrary large friction but in the absence of slope \cite{BC2017},
and very recently for any section profil and any slope or friction \cite{HS,SVgeneral}. Other results exists using different boundary conditions for instance PI controls
\cite[Chapter 8]{BastinCoron1D}, \cite{XuSallet,SantosBastinCoronAndreaPI,BastinCoronPI,XuSallet2014,SV_trinh2018,SaintVenantPI}
or full-state feedbacks resulting of a backstepping approach \cite{BCKV} (see \cite{SVbilayer,SVexner} for its application on variant systems based on the Saint-Venant equations). However, the stabilization of the Saint-Venant equations by internal control,
has seldom been studied while being very interesting mathematically and corresponding to several physical situations, for instance a water tank subject to an acceleration.

\subsection{Main result}
The water tank problem is interesting in that it has been studied for a long time and is rich enough to have led to several interesting results. Among the control results in this setting, one can cite \cite{DPR1999,PR2002,Coron2002} where the authors show, among others, that the linearized homogeneous Saint-Venant equations with null velocity at the boundaries and subject to a scalar control force are not locally approximately controllable around their uniform steady-states. This proof also implies that they are not stabilizable either.

Now, consider the non-uniform steady-states corresponding to a (small) acceleration $U(t)=\gamma$ with $\gamma>0$ fixed: $H^{*}=H^{\gamma}$, $V^{*}=0$ with $H^{\gamma}(0)=1+\gamma L/2$ and
\begin{equation}\begin{aligned}
 H^{\gamma}(x)=&1-\gamma \left(x-\frac{L}{2}\right), \; \forall x\in [0, L],\\
 &\int_0^L H^\gamma (x)dx=L.
 \end{aligned}
\end{equation}
The linearized equations around this steady-state expressed with the variables $h=H-H^{\gamma}$ and $v=V-V^{\gamma}$ denoting the perturbations and the internal control $u(t)=-(U(t)-\gamma)$
\begin{equation}\label{sv-lin}
\partial_{t}
 \begin{pmatrix}
        h\\v
       \end{pmatrix}+
       \begin{pmatrix}
        0& H^{\gamma}\\1 & 0
       \end{pmatrix}\partial_{x}\begin{pmatrix}
        h\\v
       \end{pmatrix}
       +\begin{pmatrix}
        0& -\gamma\\0& 0
       \end{pmatrix}\begin{pmatrix}
        h\\v
       \end{pmatrix}= u(t)\begin{pmatrix}
        0\\1
       \end{pmatrix},
\end{equation}
with the boundary conditions:
\begin{equation}\label{bc-sv}
 v(t,0)=v(t,L)=0.
\end{equation}
Condition \eqref{cond-conser-mass} becomes
\begin{equation}\label{cond-cons-2-0}
    \frac{d}{dt}\int_0^L h(t, x)dx=0.
\end{equation}
We assume from now on that
\begin{equation}\label{cond-cons-init}
  \int_0^L h(0, x)dx=0,
\end{equation}
which implies from \eqref{cond-cons-2-0} that
\begin{equation}\label{cond-cons-2}
\int_0^L h(t, x)dx=0.
\end{equation}
The motivation behind this choice is physical: the total mass of water is conserved when moving the water tank. Therefore if the mass of water of the initial state is not equal to the mass of water of the steady-state, the convergence cannot occur. This is not a problem in practice as for any initial mass of water there is a corresponding steady-state of the system. Therefore if \eqref{cond-cons-init} is not satisfied it means that the target steady-state was not chosen well.

From now on we assume that $|\gamma|$ is small. In particular , for any $x\in[0,L)$ $H^{\gamma}(x)\neq 0$, the transport matrix is diagonalizable and the system is strictly hyperbolic.

 We now recall the definition of exponential stability:

\begin{dfntn}\label{def-1-new}
 For a  given feedback law $u$,  the
 system \eqref{sv-lin}--\eqref{bc-sv} is called  exponentially stable with decay rate $\mu$ if there exists a constant $C>0$ such that for
any $(h_{0},v_{0})\in H^{1}((0,L);\mathbb{R}^{2}))$ satisfying the compatibility conditions $v_{0}(0)=v_{0}(L)=0$ corresponding to \eqref{bc-sv}, the
system  has a unique solution $(h,v)\in C^{0}([0,+\infty); H^{1}((0,L);\mathbb{R}^{2}))$ and
\begin{equation}
    \lVert h(t,\cdot),v(t,\cdot)\rVert_{H^{1}((0,L);\mathbb{R}^{2})}
\leq Ce^{-\mu t}\lVert h_{0},v_{0}\rVert_{H^{1}((0,L);\mathbb{R}^{2})},\text{  }\forall\text{  }t\in[0,+\infty).
\end{equation}
Moreover, we say that the control $u$ stabilizes the system \eqref{sv-lin}--\eqref{bc-sv} with decay rate $\mu$, and the action of finding such a stabilizing control is called exponential stabilization.
\end{dfntn}

In this article we give a way of stabilizing system \eqref{sv-lin}--\eqref{bc-sv} exponentially for a small enough $\gamma$.

To state our main result let us introduce some notations. First, note that we defined stability for real-valued solutions to the water-tank equations, as they are the ones that make physical sense. However, as we rely on spectral properties of the system to build feedbacks, we will mainly work with general complex-valued solutions. This is the natural framework to express the spectral properties of the system, and we will then make sure that for real-valued initial data, our feedbacks lead to real-valued dynamics (see Corollary \ref{cor-real-valued}).  We know from \cite{Russell2} (see also Section \ref{syseigen}) that the family of eigenvectors associated to the problem \eqref{sv-lin}--\eqref{bc-sv} form a Riesz basis of $(L^{2}(0, L; \CC))^{2}$, let us note them $(h_{n}^\gamma,v_{n}^\gamma)_{n\in\mathbb{Z}}$.
We denote by $\mathcal{D}_\gamma$ the space of finite (complex) linear combinations of the $(h_{n}^\gamma,v_{n}^\gamma)_{n\in\Z}$. Then any sequence $(F_n)_{n \in \Z}\in \CC^\Z$ defines an element $F$ of $\mathcal{D}_\gamma^\prime$:
\begin{equation}\langle (h_{n}^\gamma,v_{n}^\gamma)^T, F \rangle_{\mathcal{D}_{\gamma}, \mathcal{D}_{\gamma}'} =\overline{F_n}.\end{equation}
This gives us a general framework to talk about linear feedback laws.

In the rest of this article, if there is no confusion we remove the indices in the preceding functional, and simply denote the duality bracket as $\langle , \rangle$. The actual domain of definition of our feedback laws, and their regularity, will be closely studied later on in Subsection \ref{section-feedback-reg}.
\begin{thrm}\label{th1}
For any $\mu>0$, there exists $\tilde{\gamma}>0$ such that, for any $\gamma\in(0,\tilde{\gamma})$, there exists a feedback law $u$ which stabilizes the system \eqref{sv-lin}--\eqref{bc-sv}
with decay rate $\mu$. More precisely,
there exists $\nu \neq 0$ such that the feedback law $u$

\begin{equation}\label{F0-new}
    u(t):=\left\langle  \begin{pmatrix}h\\v\end{pmatrix}(t, \cdot) , F_{1}^{\gamma}\right\rangle +u_{2}(t),
\end{equation}
where
\begin{equation}\label{F0-new2}
    u_{2}'(t) = \frac{\nu L}{L^{\gamma}}\langle  (h_{0}^{\gamma} \ v_{0}^{\gamma})^{T},F_{1}^{\gamma}\rangle\left(u_{2}(t)+\left\langle  \begin{pmatrix}h\\v\end{pmatrix}(t, \cdot) , F_{1}^{\gamma}\right\rangle\right)
\end{equation}
and where $L_{\gamma}:=\frac{2}{\gamma}\left(1-\sqrt{1-\gamma\frac{L}{2}}\right)$ and
 $F_{1}^{\gamma}\in \mathcal{D}_\gamma^\prime$ is given by
\begin{equation}\label{F}
\begin{aligned}
    &\langle (h_{n},v_{n})^{T},F_{1}^{\gamma}\rangle=-\frac{\tanh(4\mu L)}{H^{\gamma}(0)}\frac{( h_{n})^{2}(0)}{\int_{0}^{L}\frac{L}{L_{\gamma}\sqrt{1-\gamma \left(x-\frac{L}{2}\right)}}\exp\left(-\int_{0}^{x}\frac{3\gamma}{4(1-\gamma (x-\frac{L}{2}))}ds\right)v_{n}(x)dx},\text{  } \forall\text{  }n\in\mathbb{Z}^{*},\\
    &\langle (h_{0}, v_{0})^{T}, F_{1}^{\gamma}\rangle=-2\frac{\tanh(4\mu L)}{H^{\gamma}(0)}\frac{(h_{0})^{2}(0)}{\nu},
    \end{aligned}
\end{equation}
stabilizes the system \eqref{sv-lin}--\eqref{bc-sv} exponentially in $(H^1)^2$ norm with decay rate $\mu$, for initial conditions in $(H^1)^2$ satisfying the boundary conditions \eqref{bc-sv}.
\end{thrm}

\begin{rmrk}
Note that the control $u$ is actually given by a feedback law having the form of a proportional integral control. As we will see later on, this comes from the fact that we actually design a linear feedback law for a dynamic extension of the system.
\end{rmrk}
\begin{rmrk}
This result works for any small $\gamma>0$, therefore one could wonder whether it could be extended to $\gamma=0$. However,  when $\gamma=0$ the system \eqref{sv-lin}--\eqref{bc-sv}
has an uncontrollable part of infinite dimension, and our strategy of proof does not apply anymore (see Section \ref{sec-mom} for details on this uncontrollable subspace). Therefore this result is sharp in this sense.
\end{rmrk}

\subsection{Transforming the system}
\label{transfosubsection}

To obtain Theorem \ref{th1}, we will proceed by first transforming \eqref{sv-lin}--\eqref{bc-sv}--\eqref{cond-cons-2} through several variable changes, then prove a stabilization result for the system thus obtained.

Let us first use the change of variables
\begin{equation}\label{diffeo0}
\begin{pmatrix}
  \xi_1\\
        \xi_2
       \end{pmatrix}=\begin{pmatrix}
        \sqrt{\frac{1}{H^\gamma}}&1\\
        -\sqrt{\frac{1}{H^\gamma}}&1
       \end{pmatrix}
       \begin{pmatrix}
        h\\v
       \end{pmatrix},
\end{equation}
we get the system in Riemann coordinates:
\begin{equation}
 \label{character-form}
\begin{pmatrix}
        \xi_{1}\\ \xi_{2}
       \end{pmatrix}_t
       \!+\!\begin{pmatrix}
        \lambda_1& 0\\0& -\lambda_2
       \end{pmatrix}
       \begin{pmatrix}
        \xi_1\\\xi_2
       \end{pmatrix}_x
       \!+\!\delta_{0}(x)\begin{pmatrix}
        1&\frac{1}{3}\\ -\frac{1}{3}& -1
       \end{pmatrix}
       \begin{pmatrix}
       \xi_1\\ \xi_2
       \end{pmatrix}=u(t)\begin{pmatrix}1\\1\end{pmatrix},
\end{equation}
where $\lambda_{1}=\lambda_{2}=\sqrt{\vphantom{1}H^{\gamma}}$, $\delta_{0}(x)=-\frac{3}{4}\frac{\gamma}{\vphantom{\sqrt{1}}\sqrt{1-\gamma \left(x-\frac{L}{2}\right)}}$, and with the boundary conditions:
\begin{equation}\label{sec-2-bound-0}
\left\{\begin{split}
 \xi_{1}(t,0)=-\xi_{2}(t,0),\\
 \xi_{2}(t,L)=-\xi_{1}(t,L).
\end{split}\right.
\end{equation}
The condition \eqref{cond-cons-2} becomes:
\begin{equation}\label{sec-2-conserv-1}
    \int_0^L \sqrt{H^{\gamma}(x)} \left(\xi_1-\xi_2\right)(x) dx=0.
\end{equation}

In the literature of hyperbolic systems, the zero-order terms in the left-hand side of \eqref{character-form} are known as source terms. In this work, we will rather refer to them as ``coupling terms", and keep the denomination ``source term" for the control. Indeed, in what follows we will be studying the spectral properties of the operator with coupling terms included.
We would like to simplify the matrix in front of the transport term. To this aim, let us introduce a change of variables in space: $y=\frac{2}{\gamma\vphantom{\sqrt{1}}}\left(\sqrt{1+\gamma \frac{L}{2}}-\sqrt{1-\gamma \left(x-\frac{L}{2}\right)}\right)$ and define
\begin{equation}\label{L-gamma}
L_{\gamma}=\frac{2}{\gamma\vphantom{\sqrt{1}}}\left(\sqrt{1+\gamma \frac{L}{2}}-\sqrt{1-\gamma \frac{L}{2}}\right),
\text{   (}\gamma\text{  }\text{is supposed sufficiently small).}
\end{equation}
Note that
\begin{equation}
L_{\gamma} = L +O(\gamma^{2}).
\end{equation}
By a slight abuse of notation we used again $\xi$, now defined on $y \in [0,L_{\gamma}]$, to denote the solutions to this last system, so that these equations become:
\begin{equation}
\label{sys0}
\begin{aligned}
\partial_{t}\begin{pmatrix}\xi_{1}\\\xi_{2}\end{pmatrix}+\begin{pmatrix}1&0\\0 &-1\end{pmatrix}\partial_{y}\begin{pmatrix}\xi_{1}\\\xi_{2}\end{pmatrix}+\delta_{1}(y) \begin{pmatrix}1&\frac{1}{3}\\-\frac{1}{3}&-1\end{pmatrix}\begin{pmatrix}\xi_{1}\\\xi_{2}\end{pmatrix}=u(t)\begin{pmatrix}1\\1\end{pmatrix}
\end{aligned},
 \end{equation}
where $\delta_{1}(y)=-\frac{3}{4}\frac{\gamma}{(\sqrt{1+\frac{\gamma L}{2}}-\gamma \vphantom{1}y/2)}$, and with the boundary conditions:
\begin{equation}
\label{bound1}
\begin{aligned}
\xi_{1}(t,0)=-\xi_{2}(t,0),\\
\xi_{2}(t,L_{\gamma})=-\xi_{1}(t,L_{\gamma}).
 \end{aligned}
\end{equation}
The conservation law \eqref{sec-2-conserv-1} becomes,
\begin{equation}
    \int_0^{L_{\gamma}} \left(\sqrt{1+\frac{\gamma L}{2}}- \frac{\gamma}{2}y\right)^2 (\xi_1(y)- \xi_2(y))dy=0.
\end{equation}

 This could be expressed in a more compact form using the following notations:
\begin{equation}\label{lambdaJ}
\begin{aligned}
\Lambda=\begin{pmatrix}1&0\\0 &-1\end{pmatrix},
J_{0}= \begin{pmatrix}1&\frac{1}{3}\\-\frac{1}{3}&-1\end{pmatrix}.
\end{aligned}
\end{equation}
We also define
\begin{equation}\label{defJ}
\begin{aligned}
J=\begin{pmatrix}0&\frac{1}{3}\\-\frac{1}{3}&0\end{pmatrix},
\end{aligned}
\end{equation}
which will be used later on. Looking at \eqref{sys0} and \eqref{lambdaJ}, the transport matrix
$\Lambda$ has now a simple form, as expected, but the length of the domain depends now on $\gamma$. We arrange this by using a scaling simultaneously on time and space and we define
\begin{equation}\label{scal}
\begin{split}
&w(t, z):= \xi(L_{\gamma} t/L, y(z)),\text{ with }y(z)= \frac{L_{\gamma}}{L}z,\\
&u_{\gamma}(t) = \frac{L_{\gamma}}{L}u(\frac{L_{\gamma}}{L}t).
\end{split}
\end{equation}
For convenience we renote $x:=z\in [0, L]$ so that $x$ still denotes the space variable, then $w(t,x)$ and $u(t) $satisfy
\begin{equation}
\begin{aligned}
&\partial_{t}w(t, x)+\Lambda\partial_{x}w(t, x)+\delta J_0 w(t, x)=u_{\gamma}(t)\begin{pmatrix}1\\1\end{pmatrix},
\label{w-con}\\
& w_{1}(t,0)=-w_{2}(t,0),\\
& w_{2}(t,L)=-w_{1}(t,L),
\end{aligned}
\end{equation}
with
\begin{equation}\label{def_delta}
\begin{split}
\delta(x)&=(L_{\gamma}/L)\delta_{1}(L_{\gamma}x/L)\\
&=   -\frac{3L_{\gamma}}{4L}\frac{\gamma}{\sqrt{1+\frac{\gamma L}{2}}-\frac{\gamma L_{\gamma}x}{2L}},
\end{split}\text{ for all }x\in[0,L],
\end{equation}
so that $\delta$ is smooth on $[0, L]$ and has the following asymptotic expression with respect to $\gamma$, \begin{equation}\label{es-delta-O2}
\delta(x)= -\frac{3}{4}\gamma \Big(1+ \frac{1}{2}\gamma \left(\frac{L}{2}+x\right)+ O(\gamma^2)\Big).
\end{equation}
And the condition of mass conservation becomes
\begin{equation}\label{missing-direc-w-con}
    \textcolor{black}{\int_0^L \left(\sqrt{1+\frac{\gamma L}{2}}- \frac{\gamma}{2}\frac{L_{\gamma}}{L}x\right)^2 (w_1(x)- w_2(x))\frac{L}{L_{\gamma}}}dx= 0,
\end{equation}
which, from now on, will be called the ``missing direction''  as this cannot be changed, whatever the control, and restricts necessarily the admissible perturbation or the reachable states.

And, finally, we use a diagonal change of coordinates
\begin{equation}\label{diffeo}
\zeta(t,x) :=\exp\left(\int_{0}^{x} \delta(y) dy\right)w(t,x),
\end{equation}
with
\begin{equation}
   \exp\left(\int_{0}^{x} \delta(y) dy\right)=\textcolor{black}{ \left(\sqrt{1+\frac{\gamma L}{2}}- \frac{\gamma}{2} \frac{L_{\gamma}}{L} x\right)^{3/2}=
   1+\left(\frac{3L}{8}-\frac{3x}{4}\right)\gamma +O(\gamma^{2}).}
\end{equation}
This last operation is used to remove the diagonal coefficients of the \textcolor{black}{coupling term} (see \cite[Chapter 9]{KrsticBook}, \cite{HuOlive2019} for more examples on the interest of this change of coordinates). The system then becomes
\begin{equation}
\partial_{t}\zeta+\Lambda\partial_{x}\zeta+\delta J \zeta=u_{\gamma}(t)\exp\left(\int_{0}^{x}\delta(y) dy\right)\begin{pmatrix}1\\1\end{pmatrix},
\label{sys01}
\end{equation}
with boundary conditions
\begin{equation}\label{cond-0}
\left\{\begin{aligned}
\zeta_{1}(t,0)&=-\zeta_{2}(t,0),\\
\zeta_{2}(t,L)&=-\zeta_{1}(t,L).
 \end{aligned}\right.
\end{equation}
{\color{black}The \textcolor{black}{condition of mass conservation becomes}
\begin{equation}\label{missing-direc-sys1}
    \int_0^L \textcolor{black}{\left(\sqrt{1+\frac{\gamma L}{2}}- \frac{\gamma}{2}\frac{L_{\gamma}}{L}x\right)^{1/2} (\zeta_1(x)- \zeta_2(x))\frac{L}{L_{\gamma}}}dx= 0.
\end{equation}}

This will be our system in the following, together with the boundary conditions \eqref{cond-0}.
\subsection{Spaces and notations}
In this subsection, we \textcolor{black}{define} several notations which will be used throughout the article. \textcolor{black}{Some of them will be introduced later on in the article but are gathered here as a glossary for the reader's convenience.
To simplify the computations and the statements we denote
\begin{equation}
\label{defspaces}
\begin{aligned}
&(L^{2})^{2}=L^{2}((0,L);\CC^{2}),\text{  }\\
&(H^{s})^{2}=H^{s}((0,L);\CC^{2}),\text{ for any }s\geq0.
\end{aligned}
\end{equation}
Similarly for any $s\in\mathbb{N}$ we denote $C^{s}=C^{s}([0,L]; \mathbb{C}^{2})$ and we note $C^{s}_{pw}$ the space of piecewise $C^{s}$ functions, i.e. functions $f$ such that there exists a subdivision $\{\sigma_i\}, \sigma_i\in [0,L], i\in \{1, \ldots, n-1\}$ for some $n\geq 1$, such that
\begin{equation}
    f_{|[\sigma_i, \sigma_{i+1}]} \in C^1([\sigma_i, \sigma_{i+1}]), \quad \forall i\in \{1, \ldots, n-1\}.
\end{equation}
For any family we denote for simplicity $(a_{n})_{n\in\mathbb{Z}}=(a_{n})$, the index set being specified when the family is not considered over all of $\mathbb{Z}$.
}
The scalar product corresponding to the $L^2$ norm is defined by
\begin{equation}
    \langle f, g\rangle:=\frac{1}{2 L} \int_0^{L} f_1(x)\overline{g_1(x)}+ f_2(x)\overline{g_2(x)} dx.
\end{equation}

We now present the following families of functions, whose existence will be justified later on:
\begin{itemize}
\item $(f_{n})_{n\in\mathbb{Z}}$ denote the eigenfunctions of the operator given by \eqref{A}--\eqref{domaineA} and associated to the original system \eqref{sys01}--\eqref{cond-0} and forming an orthonormal basis.
\item $(\widetilde{f}_{n},\widetilde{\phi}_{n})_{n\in\mathbb{Z}}$ denote the eigenfunctions forming a Riesz basis and the associated biorthonormal family of the operator given by \eqref{Atilde} and associated to the target system \eqref{target}.
        \item $(\psi_{n},\chi_{n})$ denote the eigenfunctions forming a Riesz basis and the associated biorthonormal family of the operator associated to the system \eqref{w-con}.
    \item $(\widetilde{\psi}_{n},\widetilde{\chi}_{n})$ denote the eigenfunctions forming a Riesz basis and the associated biorthonormal family of the operator associated to the system \eqref{target-new}.
  
\end{itemize}

Let us now note $\mathcal{E}$ the space of finite linear combinations of the $(f_n)_{n \in \Z}$. Then any sequence $(F_n)_{n \in \Z}$ defines an element $F$ of $\mathcal{E}^\prime$ by
\begin{equation}
\langle f_n, F \rangle_{\mathcal{E}, \mathcal{E}'} =\overline{F_n}.
\end{equation}
Here $\test$ and $\dist$ are linked to the spaces $\mathcal{D}_\gamma$ and $\mathcal{D}_\gamma^\prime$  by the changes of variables performed in the previous section.

Finally we define the spaces
\begin{equation}
    X^s:=\{f\in (L^2)^2, \quad (\tau^\mathcal{I})^{-1}(\Lambda \partial_x f + \delta(x) J f) \in (H^{s-1})^2\}, \quad s\geq 1,
\end{equation}
where $\tau^{I}$ is an isomorphism of $H^{s}$ defined by \eqref{tau-I},
and we endow them with the norms:
\begin{equation}
    \|f\|_{X^s}:= \|(\tau^\mathcal{I})^{-1}(\Lambda \partial_x f + \delta(x) J f) \|_{(H^{s-1})^2} + \|f\|_{L^2}, \quad s\geq 1.
\end{equation}

\section{General presentation of the method of proof}
\subsection{Backstepping and system equivalence}
This result will be shown by combining, on one hand, pole-shifting (see \cite[Section 10.1]{CoronBook}) and the notions of system equivalence introduced by Pavel Brunovsky in \cite{Brunovsky}, and on the other hand, ideas developed in the context of the backstepping method for PDEs.

Generalizations of pole-shifting theorems to infinite-dimensional systems have been investigated in \cite{Sun, Rebarber}. In \cite{Russell2}, pole-shifting results for general $2\times2$ hyperbolic systems with bounded feedback laws are studied using a notion of canonical form. In all these works, as in finite dimension, the assumption that the system is controllable is crucial.

On the other hand, the backstepping method was not particularly designed for controllable systems.
Backstepping originally referred to a way of designing stabilizing feedback laws for systems consisting in a stabilizable finite dimensional system with an added chain of integrators (see \cite{Kodit,Tsinias,Isi} for instance). Later, this method has been modified and adapted to general triangular systems, using this particular structure to design iterative changes of variables to perform feedback linearization. It was then modified to obtain a consistent feedback law when applied to spatial discretizations of parabolic equations (see \cite{KBHeat} and \cite{KBParab}). Remarkably, the new change of variables was the discretization of a Volterra transformation of the second kind. This was the starting point of PDE backstepping. The key idea is to look for a Volterra transformation of the second kind (which has the advantage of being invertible) mapping the original system to a target system for which the stability is easy to prove.
 These transformations were extensively used in the last decades, for instance for the heat equation \cite{BoskovicKrsticLiu,KBParab,KBHeat}, for first order hyperbolic linear then quasilinear systems \cite{VCKB,BCKV}, and for many particular cases (see \cite{CerpaCoron13, XiangKdVNull,XiangKdVFinite} for the KdV equations, \cite{MBHK} for coupled PDE-ODE systems, \cite{Coron-Xiang-2018} for Burgers equations, or \cite{KrsticBook} for an overview), the goal of each new study being to show that such a transformation exists. However, considering only a special type of invertible transformations necessarily restricts the cases where this method can be applied. Moreover, Volterra transformations of the second kind are usually used to move a complexity in the dynamics to the boundaries, to be dealt with using an appropriate control. Therefore it could be ill-adapted to an internal control stabilization problem, where the boundary conditions are fixed and cannot be changed, although some results exist by applying a second invertible transform (see \cite{KrsticHaraTsub} or \cite{Woittennek}).

 Several works have broadened the scope of the method by considering general kernel operators, namely Fredholm transformations. This requires more work as a Fredholm transform is not always invertible, but has been successful in many cases: see \cite{CoronLu1} for the Korteweg-de Vries equation and \cite{CoronLu2} for a Kuramoto-Sivashinsky equation,   \cite{Schrodinger} for a Schr\"{o}dinger equation,   \cite{CHO1} for integro-differential hyperbolic systems, and in \cite{CHO2} for general hyperbolic balance laws.

Broadly speaking, the spirit of PDE backstepping is that the stabilization problem becomes that of the existence of an isomorphism between the system under consideration and a second exponentially stable system, depending on the expression of the feedback law. This is actually closely related to the concept of ``$F$-equivalence'' for linear controllable finite-dimensional systems in \cite{Brunovsky}. And although the first installments of PDE backstepping with Volterra transformations do not seem to have anything to do with the controllability of the system, extensions of the method to more general transformations sometimes rely on a controllability assumption to build an invertible Fredholm transformation (see \cite{CHO1} and the references therein).

More recently, in \cite{Schrodinger} rapid stabilization of the linearized bilinear Schr\"{o}dinger equation was obtained. The authors used such an extension of the PDE backstepping method with a Fredholm type transformation, relying mainly on spectral properties of the Schr\"{o}dinger equation and on a controllability assumption. This was also adapted successfully, despite very different spectral properties, to the linear transport equation in \cite{ZhangRapidStab}.

As we will show in the next section, the outline of this method revolves around the notion of $F$-equivalence (with pole-shifting in mind as the choice of target system will illustrate), as shown by \eqref{FiniteDOpEq}, which can be easily adapted to an infinite-dimensional setting. The search for invertible transformations then draws heavily from the techniques of PDE backstepping, in that we look the transformation in the form of a kernel operator.

\subsection{A finite-dimensional example}\label{FiniteD}

 We use a finite-dimensional example to illustrate the variant of the backstepping method we use in this article. We refer to \cite{coron2015stabilization,Schrodinger} for alternate presentations.

Consider the finite-dimensional control system
\begin{equation}\label{FiniteDControlSys}
\dot{x}(t)=Ax(t)+ Bu(t), \quad x\in \RR^n, A \in \RR^{n\times n}, B\in \RR^{n\times1}.
\end{equation}
Suppose that \eqref{FiniteDControlSys} is controllable. Then it is known (see for example \cite[Section 10.1]{CoronBook}) that for every unitary polynomial $P \in \RR_n[X]$ of degree $n$ there exists a feedback $K\in \RR^{1\times n}$ such that $P$ is the characteristic polynomial of $A+BK$.

This pole-shifting property for controllable systems can be formulated in another way, by trying to invertibly transform system \eqref{FiniteDControlSys} into another system with shifted poles, namely
\begin{equation}\label{FiniteDTarget}
\dot{x}=(A-\lambda I)x,
\end{equation}
which is exponentially stable for a large enough $\lambda$.

More generally, in the spirit of the classification of linear controllable systems of \cite{Brunovsky}, and in particular the notion of $F$-equivalence, we can try to transform system \eqref{FiniteDControlSys} with an invertible matrix $T\in \RR^{n\times n}$:
\begin{equation}\label{FiniteDTargetGeneral}
y=Tx \;  \textrm{  satisfying }\; \dot{y}=\tilde{A}y,
\end{equation}
which is exponentially stable if $\tilde{A}$ is well chosen,  $e.g.$ $\tilde{A}= A-\lambda I$.

Suppose that $x(t)$ is a solution of system \eqref{FiniteDControlSys} with $u(t)=Kx(t)$. The invertible matrix $T$ would map \eqref{FiniteDControlSys} into
$$\dot{(Tx)} = T\dot{x}=T(A+BK)x .$$
In order for $Tx$ to be a solution of \eqref{FiniteDTargetGeneral}, it is necessary and sufficient that
\begin{equation}\label{backstepping-eq} T(A+BK)= \tilde{A}T.\end{equation}

To find such a $T$, one method is to rely on the control canonical form, introduced in \cite{Brunovsky} (see also \cite[section 10.1]{CoronBook}). As $(A,B)$ is controllable, it can always be put in canonical form (see \cite[Lemma 10.2]{CoronBook}). Thus, we now suppose without loss of generality that $(A,B)$ is in canonical form, that is:
\begin{equation}\label{canonical}
A=\begin{pmatrix} 0 & 1 & 0 & \cdots  & 0 \\
                  \vdots & \ddots & 1 & \ddots & \vdots \\
                  \vdots & \ddots & \ddots & 1 &0 \\
                  0 & \cdots & \cdots & 0&1 \\
                  a_1 & \cdots & \cdots &\cdots & a_n
                    \end{pmatrix},
                    \quad
                    B=\begin{pmatrix} 0 \\
                    \vdots \\
                    0 \\ 1 \end{pmatrix}.
\end{equation}

Now, suppose that $(\tilde{A},B)$ is also controllable, so that there exists an invertible matrix $T$ such that
\begin{equation}\label{tildecanonical}
    \begin{aligned}
        T^{-1} \tilde{A} T&= c(\tilde{A}), \quad TB=\begin{pmatrix} 0 \\
                    \vdots \\
                    0 \\ 1 \end{pmatrix}=B
    \end{aligned}
\end{equation}
where $c(\widetilde{A})$ is the companion matrix of $\widetilde{A}$. Now, because $(A,B)$ is in the canonical form \eqref{canonical}, there exists a unique $K$ such that
\begin{equation}\label{companion-A-tilde}
    A+BK=c(\tilde{A}),
\end{equation}
which yields
\begin{equation*}
   T(A+BK)=\tilde{A}T.
\end{equation*}
Recall now the second equation in \eqref{tildecanonical}:
\begin{equation}
    TB=B.
\end{equation}
Injecting the above equation into \eqref{backstepping-eq}, we get the following equations:
\begin{equation}\label{FiniteDOpEq}
\left\{\begin{aligned}
TA+BK&=\tilde{A}T, \\
TB=& B,
\end{aligned}\right.
\end{equation}
for which we just proved the existence of a solution $(T, K)$.

Let us now prove that the solution to \eqref{FiniteDOpEq} is unique. First note that from \eqref{companion-A-tilde} and \eqref{backstepping-eq}, there exists a unique $K\in \RR^{1\times n}$ such that there exists solutions $(T,K)$ to \eqref{FiniteDOpEq}. Now let $T_1, T_2 \in \RR^{n\times n }$ be invertible solutions to \eqref{FiniteDOpEq}. Then from \eqref{tildecanonical} we get
\begin{equation}
\begin{aligned}
    T_1T_2^{-1} \widetilde{A}  & = \widetilde{A} T_1T_2^{-1}, \\
    T_1T_2^{-1} B & = B.
    \end{aligned}
\end{equation}
Then, applying $\widetilde{A}$ to the second equation above, we get
\begin{equation}
    T_1T_2^{-1} \widetilde{A}^j B  = \widetilde{A}^j B, \quad j=0, \cdots, n-1.
\end{equation}
Using the Kalman rank condition on the controllable pair $(\tilde{A}, B)$, this implies that $T_1T_2^{-1}=I_n$, and thus, the uniqueness of the solution $(T,K)$ to \eqref{FiniteDOpEq}.

\begin{thrm}
If $(A,B)$ and $(\tilde{A}, B)$ are controllable, then there exists a unique pair $(T,K)$ satisfying equations \eqref{FiniteDOpEq}.
\end{thrm}

\begin{rmrk}
Another way of understanding these equations is to notice that \eqref{backstepping-eq} is not ``well-posed''. Indeed, if there exists a solution $(T,K)$ to the single equation \eqref{backstepping-eq}, then the $(aT,K), a \neq 0$ are all solutions. This suggests that a second condition could be added to resolve this degree of freedom. Moreover, \eqref{backstepping-eq} is not linear, but bilinear in $(T,K)$. An obvious option is then to consider the additional equation
\begin{equation*}
    TB=B,
\end{equation*}
so that the $TBK$ term becomes $BK$, and the problem becomes linear in $(T,K)$, as can be seen in \eqref{FiniteDOpEq}.

As we have seen above, it turns out that with this additional equation the problem becomes ``well-posed''. In fact, this additional equation appears when one considers canonical forms, which makes it all the more natural.
\end{rmrk}

The controllability of \eqref{FiniteDControlSys} and \eqref{FiniteDTargetGeneral}  is crucial here, as it allows us to use the control canonical form. However, another proof can be found in \cite{Schrodinger}, which is more adaptable to the context of PDEs: the idea is to suppose that $A$ and $\tilde{A}$ are diagonalizable. Then the controllability of $(\tilde{A}, B)$ allows to build a basis for the space state using the eigenvectors of $A$, in which $T$ can then be constructed. The $TB=B$ condition along with the controllability of the first system help define the coefficients of the feedback $K$, and finally the controllability of the second system ensures the invertibility of $T$ with $K$ thus defined.

This $F$-equivalence formulation of pole-shifting, which links controllability to stabilization, can be used in infinite dimension. In our case, the controllability of \eqref{sys01}--\eqref{cond-0} will have the same importance: it will also allow us to build a basis for the state space,  which will be crucial to build invertible transformations depending on $F$, using ideas developed in the context of PDE backstepping.

\section{Preparing for backstepping}

As announced above, we will use a variant of PDE backstepping that relies on a controllability assumption. As such, it cannot be applied directly to the linearized water-tank system \eqref{sys01}--\eqref{cond-0}--\eqref{missing-direc-sys1}, as it is not controllable. Instead it will be applied to a modified version of this system, with a modified control operator, which corresponds to a specific dynamic extension of this system. This explains the proportional integral form of the feedback, which is simply a linear feedback for the extended system.

\subsection{Dealing with mass conservation}
\subsubsection{A modified control operator}
Let us now consider again our system \eqref{sys01}:
\begin{equation*}
\partial_{t}\zeta+\Lambda\partial_{x}\zeta+\delta J \zeta=u_{\gamma}(t)\exp\left(\int_{0}^{x}\delta(y) dy\right)\begin{pmatrix}1\\1\end{pmatrix},
\end{equation*}
with boundary conditions \eqref{cond-0}.
In order to identify the control, in the following we define
\begin{equation}
    \mathcal{I}:=\exp\left(\int_{0}^{x}\delta(y) dy\right)\begin{pmatrix}1 \\ 1 \end{pmatrix}.
\end{equation}
As we will see later on, this control term has a drawback:
from \eqref{f-0-1-2-0}, we get
\begin{equation}
    \label{missing-moment}
    \langle \mathcal{I}, f_{0}\rangle=0,
\end{equation}
which means that the control cannot act on this direction and therefore the system is not fully controllable. Physically this comes from the fact that the control does not add or spill any water, thus the mass is conserved.\\

To overcome this difficulty we introduce the following virtual control profile
\begin{equation}\label{def-I-nu}
    \mathcal{I}_{\nu}:=  \mathcal{I}+ \nu f_0, \quad \nu >0,
\end{equation}
together with the following virtual system
\begin{equation}
\label{sys1virt}
\left\{\begin{aligned}
& \partial_{t}\mathbf{Z}+\Lambda\partial_{x}\mathbf{Z}+\delta(x) J \mathbf{Z}
=\langle \mathbf{Z}(t,\cdot), F\rangle\mathcal{I}_{\nu}, \\
& \mathbf{Z}_1(t, 0)=-\mathbf{Z}_2 (t, 0) , \quad \forall t\geq 0, \\
& \mathbf{Z}_1(t, L)=-\mathbf{Z}_2 (t, L), \quad \forall t\geq 0,
\end{aligned}\right.
\end{equation}
where $F$ is a linear feedback to be determined.   This corresponds to the following controlled system
\begin{equation}
\label{sys1virt-con}
\left\{\begin{aligned}
& \partial_{t}\mathbf{Z}+\Lambda\partial_{x}\mathbf{Z}+\delta(x) J \mathbf{Z}
= u_{\gamma}^{\nu}(t) \mathcal{I}_{\nu}, \\
& \mathbf{Z}_1(t, 0)=-\mathbf{Z}_2 (t, 0) , \quad \forall t\geq 0, \\
& \mathbf{Z}_1(t, L)=-\mathbf{Z}_2 (t, L), \quad \forall t\geq 0,
\end{aligned}\right.
\end{equation}

\subsubsection{A system of eigenvectors for the open-loop system (i.e. without feedback laws)}
\label{syseigen}

Considering \eqref{sys01}, let us define the following operator, $\mathcal{A}: \mathcal{D}(\mathcal{A})\rightarrow (L^2)^2$,
\begin{equation}\label{A}
    \begin{array}{c}
        \mathcal{A}:=\Lambda \partial_x + \delta(x) J,\\
    \end{array}
\end{equation}
defined on the domain
\begin{equation}\label{domaineA}
  D(\mathcal{A}):=\left\{(f_{1}, f_{2}) \in (H^1)^2 , \quad f_{1}(0)+ f_{2}(\textcolor{black}{0})=0, \quad f_{1}(L)+f_{2}(L)=0\right\}.
\end{equation}
According to  the explicit formulation of $\delta$ given in \eqref{def_delta}, for real valued $\gamma>0$ the operator $\mathcal{A}$ has real coefficients.  In our stabilization problem we always treat the cases that $\gamma>0$,  except that  for  technical reasons  in order to derive some asymptotic information of eigenfunctions in Section \ref{s3} we will extend the definition domain of $\gamma>0$ to $\gamma\in \mathbb{C}$ thus complex resolvent tools can be used.   After having obtained the required spectral information we further restrict the $\gamma$ in $\mathbb{R}^{+}\cap \mathbb{C}= \mathbb{R}^{+}$ to ensure that $\delta$ has real value to make the closed-loop system of physical sense.
Its adjoint operator,  $\mathcal{A}^*: \mathcal{D}(\mathcal{A}^*)\rightarrow (L^2)^2$, is clearly defined by
\begin{equation}\label{A-star}
    \begin{array}{c}
        \mathcal{A}^\ast:=-\Lambda \partial_x - \overline{ \delta}(x) J,\\
    D(\mathcal{A}^\ast):=
    \left\{(f_{1}, f_{2}) \in (H^1)^2 , \quad f_{1}(0)+ f_{2}(0)=0, \quad f_{1}(L)+f_{2}(L)=0\right\}.
    \end{array}
\end{equation}
In the preceding equation by denoting $\overline{\delta}(x)$ we  want to include the situation when $\gamma$ be complex valued, as it will be used in Section \ref{s3}.  However, we shall always keep in mind that in our stabilization problem $\gamma>0$, for which we have,
\begin{equation}
    \label{AntiHermitian}
    \mathcal{A}^\ast=-\mathcal{A},  \textrm{ if $\gamma\in \mathbb{R}$}.
\end{equation}
We know from \cite{Russell2} that $\mathcal{A}$ has a family of eigenfunctions, which we note $(f_n(\gamma))= (f_n)$, that form a Riesz basis of $(L^2)^2$. From \eqref{AntiHermitian} we know that the $(f_n)$ actually form an orthonormal basis, and the corresponding eigenvalues $\mu_n$ are all imaginary. Moreover, they satisfy the following asymptotic behaviour, given the boundary conditions we have set:
\begin{equation}\label{mu_n asymp}\mu_n= \frac{i \pi n}L + \mathcal{O}\left(\frac1n\right), \quad \forall n\in \Z.\end{equation}
Moreover, given the definition of $\mathcal{A}$, we can easily derive a few additional properties (see Appendix \ref{appendixA} for the proof):
\begin{prpstn}\label{eigenf-prop}
The $(f_n, \mu_n)$ satisfy the following:
\begin{itemize}
\item[(i)]\begin{equation}
\label{imaginary-ev}
\mu_{-n}=\overline{\mu_n}=-\mu_n, \quad \forall n \in \Z.
\end{equation}
In particular, $\mu_0=0$.
\item[(ii)] \begin{equation}
\label{fn-f-n}
    f_{-n}=\overline{f_n}=(-f_{n,2}(\cdot), -f_{n,1}(\cdot)), \quad \forall n \in \Z.
\end{equation}

In particular, $f_{n,1}(0), f_{n,1}(L) \in \RR$, and
\begin{equation}\label{f-0-1-2-0}
    f_{0, 1}(x)+ f_{0, 2}(x)=0, \forall x\in [0, L].
\end{equation}
\end{itemize}
\end{prpstn}
As pointed out earlier, we work with complex valued functions because it is the natural framework to express and use the spectral properties of $\mathcal{A}$. However, as we are working with a physical system, we will have to bear in mind that our final stabilization result concerns real-valued initial conditions. Consequently, the feedback law we build should be real-valued on real-valued functions, so that the dynamics of the closed-loop system make physical sense.

It turns out that it is quite straightforward to characterize real-valued functions and feedbacks, given the previous proposition:
\begin{crllr}\label{cor-real-valued}
Let $f\in (L^2)^2$. Then $f$ is real-valued (a.e.) if and only if
\begin{equation}
    \label{real-valued-func}
    \langle f, f_n \rangle=\overline{\langle f, f_{-n} \rangle}, \quad \forall n \in \Z.
\end{equation}
Let $F \in \dist$. Then $F$ is real-valued on real-valued functions if and only if
\begin{equation}
    \label{real-valued-form}
    \langle f_n, F \rangle=\overline{\langle f_{-n},  F\rangle}, \quad \forall n \in \Z.
\end{equation}
\end{crllr}

Throughout this article $s$ denotes an element of $[0,+\infty)$. We have:
\begin{equation}
    \label{D(A^s)}
    D(\mathcal{A}^s):=\{\alpha \in (L^2)^2, \quad \sum\limits_{n\in\mathbb{Z}} (1+|\mu_n|^{2s}) |\langle\alpha, f_n \rangle |^2 < \infty \}\subset (H^s)^2, \quad s\in  [0, +\infty).
\end{equation}
We endow  $D(\mathcal{A}^s)$ with the norm
\begin{equation}
    \|\alpha\|^2_{D(\mathcal{A}^s)}=\sum\limits_{n\in\mathbb{Z}} (1+|\mu_n|^{2s}) |\langle\alpha, f_n \rangle |^2, \quad  \forall \alpha \in D(\mathcal{A}^s).
\end{equation}
In fact, for $s\in \NN$, $\mathcal{A}^s f_n= \mu_n^s f_n$, and a simple Sobolev embedding estimate yields the existence of  $c$ and $C$ independent of $n$ such that
\begin{equation}
    c(1+|\mu_n|^{2s})\leq \|f_n\|_{L^2}^2+  \|\partial_x^s f_n\|_{L^2}^2\leq C  (1+|\mu_n|^{2s}),\; \forall n\in \mathbb{Z},
\end{equation}
therefore
\begin{equation}
    c\sum\limits_{n\in\mathbb{Z}} (1+|\mu_n|^{2s}) |\langle\alpha, f_n \rangle |^2\leq  \|\partial_x^s \alpha\|_{L^2}^2+ \|\alpha\|_{L^2}^2 \leq  C\sum\limits_{n\in\mathbb{Z}} (1+|\mu_n|^{2s}) |\langle\alpha, f_n \rangle |^2, \; \forall \alpha \in D(\mathcal{A}^s),
\end{equation}
while the other cases can be treated by interpolation arguments.

\subsection{Our target system}\label{section-target}

In past applications of the backstepping method, the most frequently used target system is simply the uncontrolled system corresponding to the damped operator $\mathcal{A}-\mu Id$, for some $\mu>0$. This choice can be easily understood: by adding a damping large enough the solution is likely to be decaying with a decay rate large enough. Moreover, note that if $(\mathcal{A},B)$ is controllable and $\tilde{\mathcal{A}}=\mathcal{A}-\mu Id$, then $(\tilde{\mathcal{A}}, B)$ is controllable as it can be seen by performing the change of state-variables $y=e^{\mu t}x$ (or by using the Kalman's rank condition \cite[Theorem 1.6]{CoronBook} in finite dimension).
As it appears, in our case it is more practical to consider a target system where the dissipation occurs instead at the boundary:
\begin{equation}
\label{target}
\left\{\begin{aligned}
 &\partial_{t}z+\Lambda \partial_x z + \delta(x)J z= 0,\\
&z_{1}(t,0)=-e^{-2\mu L} z_{2}(t,0), \\
&z_{2}(t,L)=-z_{1}(t,L),
 \end{aligned}\right.
\end{equation}
which corresponds to the auxiliary control system
\begin{equation}
\label{target-control}
\left\{\begin{aligned}
 &\partial_{t}z+\Lambda \partial_x z + \delta(x)J z= v(t)\mathcal{I}_{\nu},\\
&z_{1}(t,0)=-e^{-2\mu L} z_{2}(t,0), \\
&z_{2}(t,L)=-z_{1}(t,L),
 \end{aligned}\right.
\end{equation}
An important step will then be to prove that for some choice of $\nu>0$, both \eqref{sys1virt-con} and \eqref{target-control} are controllable in $D(\mathcal{A})$ and $D(\widetilde{\mathcal{A}})$  respectively  (see \eqref{domaineA} and \eqref{Atilde}).
\subsubsection{A system of eigenvectors for the target system}

As this target system has boundary conditions that are different from the original system,
let us define a new operator:
\begin{equation}\label{Atilde}
    \begin{array}{c}
        \widetilde{\mathcal{A}}:=\Lambda \partial_x + \delta(x) J,\\
        D(\widetilde{\mathcal{A}}):=\left\{(f_{1}, f_{2}) \in (H^1)^2 , \quad f_{1}(0)+e^{-2\mu L} f_{2}(0)=0,\text{  } f_{1}(L)+f_{2}(L)=0\right\}.
    \end{array}
\end{equation}
The eigenvectors $(\widetilde{f}_n, \widetilde{\mu}_{n})$ of $\widetilde{\mathcal{A}}$ form a Riesz basis of $\left(L^2 \right)^2$, and, using again the results in \cite{Russell2}, we have the following asymptotic development for $\widetilde{\mu}_{n}$:
\begin{equation}\label{mu_nAsympDissip}
\widetilde{\mu}_{n}=\mu + \frac{i \pi n}L + O\left(\frac1n\right).
\end{equation}
Moreover, the $(\widetilde{f}_n)$ admit a biorthogonal family which we note $(\widetilde{\phi}_n)$. Note that it is a well known fact that the $(\widetilde{\phi}_n, \overline{\widetilde{\mu}_{n}})$ are the eigenvectors of the adjoint operator:
\begin{equation}\label{AtildeStar}
    \begin{array}{c}
        \widetilde{\mathcal{A}}^\ast:=-\Lambda \partial_x - \delta(x) J,\\
        D(\widetilde{\mathcal{A}}^\ast):=\left\{(f_{1}, f_{2}) \in (H^1)^2 , \quad f_{1}(0)+e^{2\mu L} f_{2}(0)=0, f_{1}(L)+f_{2}(L)=0\right\}.
    \end{array}
\end{equation}
Again, this allows us to define the following spaces:
\begin{equation}
    D(\widetilde{\mathcal{A}}^s):=\{\alpha \in (L^2)^2, \quad \sum (1+|\widetilde{\mu}_n|^{2s}) |\langle\alpha, \widetilde{\phi}_n \rangle |^2 < \infty \}, \quad s\geq 0,
\end{equation}
endowed with the norm
\begin{equation}
    \|\alpha\|^2_{D(\widetilde{\mathcal{A}}^s)}=\sum (1+|\widetilde{\mu}_n|^{2s}) |\langle\alpha, \widetilde{\phi}_n \rangle |^2, \quad  \forall \alpha \in   D(\widetilde{\mathcal{A}}^s),
\end{equation}
and
\begin{equation}
    D((\widetilde{\mathcal{A}}^\ast)^s):=\{\alpha \in (L^2)^2, \quad \sum_{n\in \Z} (1+|\widetilde{\mu}_n|^{2s}) |\langle\alpha, \widetilde{f}_n \rangle |^2 < \infty \}, \quad s\geq 0,
\end{equation}
endowed with the norm
\begin{equation}
    \|\alpha\|^2_{ D((\widetilde{\mathcal{A}}^\ast)^s)}=\sum_{n\in \Z} (1+|\widetilde{\mu}_n|^{2s}) |\langle\alpha, \widetilde{f}_n \rangle |^2, \quad  \forall \alpha \in  D((\widetilde{\mathcal{A}}^\ast)^s).
\end{equation}

Now notice that when $\gamma=0$, from the expression of $\delta_{1}$ and \eqref{def_delta}, $\delta(x)=0$ and the operator becomes:
\begin{equation}\label{Atilde-order0}
    \begin{array}{c}
        \widetilde{\mathcal{A}}^{(0)}:=\Lambda \partial_x ,\\
        D(\widetilde{\mathcal{A}}^{(0)}):=\left\{(f_{1}, f_{2}) \in (H^1)^2 , \quad f_{1}(0)+e^{-2\mu L} f_{2}(0)=0, f_{1}(L)+f_{2}(L)=0\right\},
    \end{array}
\end{equation}
for which the eigenvectors and eigenvalues are
\begin{equation}\label{38}
    \widetilde{f}_n^{(0)}=\begin{pmatrix} e^{\left(\mu+\frac{i \pi n}L\right)x} \\ -e^{\left(\mu+\frac{i \pi n}L\right)(2L-x)} \end{pmatrix}, \quad \widetilde{\mu}_n^{(0)}=\mu + \frac{i\pi n}L, \quad \forall n \in \Z,
\end{equation}
and the corresponding biorthogonal family is given by:
\begin{equation}\label{39}
    \widetilde{\phi}_n^{(0)}=\begin{pmatrix} e^{\left(-\mu+\frac{i \pi n}L\right)x} \\ -e^{\left(-\mu+\frac{i \pi n}L\right)(2L-x)} \end{pmatrix},  \quad \forall n \in \Z.
\end{equation}
Direct calculation implies that these eigenfunctions are not normal in $(L^2)^2$ space. Thanks to \cite{Russell2}, we know that $ \{\widetilde{f}_n^{(0)}/ \lVert  \widetilde{f}_n^{(0)} \lVert_{L^2}\}$ form a Riesz basis of $(L^2)^2$ space. Moreover, since
\begin{equation}\label{eqref-f-L2}
     \lVert  \widetilde{f}_n^{(0)} \lVert_{(L^2)^2}^{2}= \frac{e^{4\mu L}-1}{2\mu},
\end{equation}
we know that
\begin{gather}\label{eqref-f-n-riesz}
    \{\widetilde{f}_n^{(0)}\} \textrm{ form a Riesz basis of $(L^2)^2$ space.}
\end{gather}

Now note that in this case, as $\delta=0$, $\mathcal{I}=(1, 1)$. By integration by parts,
\begin{equation}\label{40}
    \begin{aligned}
      \overline{\widetilde{\mu}_n^{(0)}} \left\langle (1,1) ,  \widetilde{\phi}_n^{(0)} \right\rangle &= \left\langle (1,1) ,  \Lambda \partial_x \widetilde{\phi}_n^{(0)} \right\rangle \\
      &= \overline{\left(\widetilde{\phi}_n^{(0)}\right)_1(L)}-\overline{\left(\widetilde{\phi}_n^{(0)}\right)_1(0)} - \overline{\left(\widetilde{\phi}_n^{(0)}\right)_2(L)} + \overline{\left(\widetilde{\phi}_n^{(0)}\right)_2(0)} \\
      &=  2(-1)^n e^{-\mu L} - 1 - e^{-2\mu L},
    \end{aligned}
\end{equation}
which is clearly bounded for $\mu\neq 0$. This, together with the moment theory,  suggest that  $(\widetilde{\mathcal{A}}^{(0)}, (1,1)) $ might be controllable, the detailed proof of which will be given in Section \ref{sec-con-diss}.

Moreover, similarly to \eqref{w-con}, we can also consider the following system with coupling term $\delta(x)J_0$ and control $v(t)(1, 1)$,
\begin{equation}
\label{target-new}
\left\{\begin{aligned}
 &\partial_{t}\tilde z+\Lambda \partial_x \tilde z + \delta(x)J_0 \tilde z=v(t)\begin{pmatrix}1\\1\end{pmatrix},\\
&\tilde z_{1}(t,0)=-e^{-2\mu L} \tilde z_{2}(t,0), \\
&\tilde z_{2}(t,L)=-\tilde z_{1}(t,L),
 \end{aligned}\right.
\end{equation}
which, thanks to similar calculations,  is also controllable for $\gamma=0, \mu\neq 0$.

These calculations at order $0$ also allow us to state the following properties for the eigenfunctions of $\widetilde{\mathcal{A}}$, and their biorthogonal family, analogous to Proposition \ref{eigenf-prop} and its proof (see Appendix \ref{appendixA}):
\begin{prpstn}\label{eigenf-prop-tilde}
The $(\widetilde{f}_n, \widetilde{\phi}_n, \widetilde{\mu}_n)$ satisfy the following:
\begin{itemize}
\item[(i)]\begin{equation}
\label{tilde-ev-conjugate}
\widetilde{\mu}_{-n}=\overline{\widetilde{\mu}_n}, \quad \forall n \in \Z.
\end{equation}
\item[(ii)] \begin{equation}
\label{fn-f-n-tilde}
\begin{aligned}
    \widetilde{f}_{-n}&=\overline{\widetilde{f}_n},  \\
    \widetilde{\phi}_{-n}&=\overline{\widetilde{\phi}_n}, \quad \forall n \in \Z.
    \end{aligned}
\end{equation}
\end{itemize}
\end{prpstn}

\subsubsection{Exponential stability of the target system}\label{subsubsection-expstab-target}
In this section we show the following proposition  concerning the stability of system \eqref{sys1virt}.
\begin{prpstn}\label{expstab-target}
For any $\lambda_{0}\in(0,\mu)$, there exists $\gamma_{s}(\lambda_{0})>0$ such that for any $\gamma\in(0,\gamma_{s})$ and $\lambda\in[0,\lambda_{0}]$, the target system \eqref{target} is exponentially stable with decay rate $\lambda$ (for the $H^{p}$ norm, for any $p\in\mathbb{N}$). Moreover $\gamma_{s}$ can be chosen continuous and decreasing with respect to  $\lambda_{0}$.
\end{prpstn}
\begin{proof}
First,
note that the system is well-posed in $H^{p}$ for initial condition satisfying the compatibility conditions at the boundaries \cite{BastinCoron1D} (see also  from
\cite{LiBook} in $C^{p}$). More precisely, let $T>0$, there exists a constant $C(T)>0$ such that for any $ z_{0}\in H^{p}$ the system \eqref{target} with initial condition $ z_{0}$ has a unique solution $ z\in C^{0}([0,T],H^{p})$ and
\begin{equation}\label{estimate-well-posed}
    \lVert z(t,\cdot)\rVert_{H^{p}}\leq C(T)\lVert z_{0}\rVert_{H^{p}}.
\end{equation}

We now define the following Lyapunov function candidate $V$ 
\begin{equation}\label{defV}
V( Z)=\sum\limits_{n=0}^{p}\lVert \Theta(x) (\widetilde{\mathcal{A}}^{n} Z) \rVert_{L^{2}(0,L)}^{2},\quad \forall\text{  } Z\in H^{p}(0,L),
\end{equation}
where $\Theta=\text{diag}(\sqrt{\theta_{1}},\sqrt{\theta_{2}})$ with $\theta_{1}$ and $\theta_{2}$ two positive $C^{1}$ functions to be selected later on. This is a common form of so-called basic quadratic Lyapunov function \cite{BastinCoron2011,BastinCoron1D}. Obviously $V$ is equivalent to the square of the $H^{p}$ norm in the sense that there exists positive constants $C_{1}$ and $C_{2}$ such that for any $ Z\in H^{p}(0,L; \mathbb{C}^2)$,
\begin{equation}\label{equivnorm}
C_{1}\lVert  Z\rVert_{H^{p}(0,L)}^{2}\leq V( Z)\leq C_{2}\lVert  Z\rVert_{H^{p}(0,L)}^{2}.
\end{equation}
Now observe that for a solution $ z$ to  system \eqref{target}, one has
\begin{equation}
V(z)=\sum\limits_{n=0}^{p}\langle\Theta(x)\partial_{t}^{n} z(t,\cdot),\Theta(x)\partial_{t}^{n} z(t,\cdot)\rangle.
\end{equation}
Let $n\in\{0,...,p\}$.  From \eqref{target} we know that  $\partial_{t}^{n}z$ is also a solution to \eqref{target}.
Therefore, differentiating $V( z)$ along time, one has
\begin{equation}
\dot V( z)=-2\mathfrak{Re}\left(\sum\limits_{n=0}^{p}\langle \Lambda\partial_{x}(\partial_{t}^{n} z)(t,\cdot),\Theta^{2}\partial_{t}^{n} z(t,\cdot)\rangle+\langle \delta J\partial_{t}^{n} z(t,\cdot),\Theta^{2}\partial_{t}^{n} z(t,\cdot)\rangle\right).
\end{equation}
Thus, integrating by parts and using the boundary conditions of \eqref{target},
\begin{equation}
\begin{aligned}
\dot V( z)=&-2\lambda V\\
&-\sum\limits_{n=0}^{p}[(\theta_{1}(L)-\theta_{2}(L))(\partial_{t}^{n}z_{1})^{2}(t,L)+(\theta_{2}(0)-\theta_{1}(0)e^{-4\mu L})(\partial_{t}^{n}z_{2})^{2}(t,0)]\\
&-\sum\limits_{n=0}^{p}\mathfrak{Re}\left(\langle(-\Lambda\left(\Theta^{2}\right)'-2\lambda \Theta^{2}+\Theta^{2}\delta J+(\Theta^{2}\delta J)^{\top})\partial_{t}^{n} z,\partial_{t}^{n} z\rangle \right).
\end{aligned}
\end{equation}
Our goal is now to choose $\Theta$ such that the two last sums are nonnegative. Recognizing a 
quadratic form in the integrals, it suffices to ensure that
\begin{equation}\label{condg10}
\begin{aligned}
&\theta_{1}(L)\geq \theta_{2}(L),\\
&\theta_{2}(0)-\theta_{1}(0)e^{-4\mu L}\geq 0\\
&(-\Lambda\left(\Theta^{2}\right)'-2\lambda \Theta^{2}+\Theta^{2}\delta J+(\Theta^{2}\delta J)^{\top}))\text{ is definite semi-positive}.
\end{aligned}
\end{equation}
Denoting $\Xi_{1}=\theta_{1}\exp(2\lambda (x-L))$ and $\Xi_{2}=\theta_{2}\exp(-2\lambda (x-L))$, \eqref{condg10} is equivalent to
\begin{equation}
\begin{aligned}
&\Xi_{1}(L)\geq \Xi_{2}(L),\\
&\Xi_{2}(0)-\Xi_{1}(0)e^{-4(\mu-\lambda) L}\geq 0,\\
&-\Xi_{1}'\Xi_{2}'\geq \left(\frac{\delta}{3}\right)^{2}(\Xi_{1}\exp(-2\lambda (x-L))-\Xi_{2}\exp(2\lambda (x-L)))^{2}.
\end{aligned}
\end{equation}
Following \cite[Proposition 1]{BastinCoron2011}, the existence of $\Xi_{1}$, $\Xi_{2}$ positive and of class $C^{1}$ satisfying these conditions is equivalent to the existence of $\eta$ positive and of class $C^{1}$ satisfying
\begin{equation}
\begin{aligned}
&\eta(L)\leq 1,\\
&\eta(0)= e^{-2(\mu-\lambda) L},\\
&\eta' = \left|\frac{\delta}{3}\right|\left|\exp(-2\lambda (x-L))-\eta^{2}\exp(2\lambda (x-L))\right|,
\end{aligned}
\end{equation}
which, using the first condition in the third one, is in fact equivalent to
\begin{equation}\label{eta_cond}
\begin{aligned}
&\eta(L)\leq 1,\\
&\eta(0)= e^{-2(\mu-\lambda) L},\\
&\eta' = \left|\frac{\delta}{3}\right|\left(\exp(-2\lambda (x-L))-\eta^{2}\exp(2\lambda (x-L))\right).
\end{aligned}
\end{equation}
We will now show the existence of such $\eta$ by exhibiting a super-solution to the two last equations of \eqref{eta_cond} satisfying also the first condition. Let us introduce $\xi$ being the $C^{1}$ solution of
\begin{equation}\label{def_xi}
\begin{aligned}
&\xi'=\lVert\frac{\delta}{3}\rVert_{L^{\infty}(0,L)}\left(e^{2\lambda(L-x)}\right),\\
&\xi(0)=e^{-2(\mu-\lambda) L}.
\end{aligned}
\end{equation}
This system can be easily solved and
\begin{equation}
    \xi(x)=e^{-2(\mu-\lambda) L}+\frac{\lVert\delta\rVert_{L^{\infty}(0,L)}}{6\lambda}\left(e^{2\lambda L}-e^{2\lambda(L-x)}\right).
\end{equation} 
Let us now set
\begin{equation}\label{cond_gamma}
\gamma_{s}= \min \left(\frac{7}{16 L},6\lambda\frac{1-e^{-2(\mu-\lambda) L}}{\left(e^{2\lambda L}-e^{2\lambda(L-x)}\right)}\right),
\end{equation}
and assume that $\gamma\in(0,\gamma_{s})$, then one has from the definition of $\delta$ given in \eqref{def_delta},
\begin{equation}
\xi(L)\leq 1.
\end{equation}
Thus, looking at \eqref{def_xi} and the two last equations of \eqref{eta_cond} and by comparison (see for instance \cite{Hartman}), $\eta$ exists on $[0,L]$ and in addition $\eta\leq\xi $ on $[0,L]$.
Thus, choosing such $\eta$, we have
\begin{equation}\label{V-diff-ineq}
    \dot V( z)\leq -2\lambda V,
\end{equation}
thus, using \eqref{equivnorm},
\begin{equation}\label{target-exp-stab-Hp}
\lVert  z(t,\cdot)\rVert_{H^{p}}\leq \sqrt{\frac{C_{2}}{C_{1}}}\lVert  z_{0}\rVert_{H^{p}}e^{-\lambda t},\text{ on }[0,T].
\end{equation}
But, as $T>0$ was chosen arbitrary and $C_{1}$, $C_{2}$ and $\lambda$ are independent of $T$ this is also true on $[0,+\infty)$. This ends the proof of Proposition \ref{expstab-target}
\end{proof}
\begin{rmrk}
As it can be seen in condition \eqref{cond_gamma},
when $\gamma$ is small enough we can actually achieve a decay rate as close as we want to $\mu$, as was expected looking at the eigenvalues of $\widetilde{\mathcal{A}}$ given by \eqref{mu_nAsympDissip}. Indeed, $\widetilde{\mathcal{A}}$ is obtained by adding a small perturbation (i.e. $\gamma$ small) to a simple operator (corresponding to the case $\gamma=0$). As the real part of the eigenvalues of the unperturbed operator is $\mu$, what we showed with \eqref{mu_nAsympDissip} is that the eigenvalues of the perturbed  system can have a real part as close as we want to $\mu$, provided that $\gamma$ is small enough, which is also characterized by Theorem \ref{tho-rus} and Lemma \ref{lem-es-mu-n-mu}.  This, combined with the fact that eigenfunctions form a Riesz basis (or the Spectral Mapping Theorem), implies an exponential decay rate as close as we want to $\mu$, provided again that $\gamma$ is small enough.  From this point of view, Lyapunov approach also provides a rough estimate on the perturbation of eigenvalues.
\end{rmrk}

Finally, following \cite{BastinCoron2011} again, if the maximal solution of the two last equations of \eqref{eta_cond} does not exist on $[0,L]$ or does not satisfies the first condition of \eqref{eta_cond}, then there does not exist any Lyapunov function with a decay rate larger of equal to $\lambda$ of the form
\begin{equation}\label{defV2}
V( Z)=\sum\limits_{n=0}^{p}\langle\widetilde{\mathcal{A}}^{n} Z,Q(x) (\widetilde{\mathcal{A}}^{n} Z)\rangle,\text{  }\forall\text{  } Z\in H^{p}(0,L),
\end{equation}
where $Q\in C^{1}([0,L],M_{2}^{+}(\mathbb{R}))$, with $M_{2}^{+}(\mathbb{R})$ the space of positive definite matrix on $\mathbb{R}^{2}$. Note that the form \eqref{defV2} includes all the Lyapunov functions of the form \eqref{defV}. This implies that we cannot get a uniform bound on $\gamma$ and in particular that for any $\gamma>0$ there exists $\mu_{\gamma}>0$ such that there does not
exists any Lyapunov function of the form \eqref{defV2} with a decay rate larger of equal to $\mu_{\gamma}>0$. Indeed let $\gamma>0$, and assume by contradiction for any $\mu>0$ there exists a Lyapunov function of the form \eqref{defV2} with decay rate larger or equal to $\mu$, then there exists a function $\eta$ on $[0,L]$ satisfying \eqref{eta_cond}. Besides,
\begin{equation}
    \eta'\geq \inf\limits_{[0,L]}\left|\frac{\delta}{3}\right|\left(e^{\mu(L-x)}-e^{-\mu(L-x)}\right),
\end{equation}
hence, integrating and using that $\eta(0)\geq0$,
\begin{equation}\label{contrad}
    \eta(L)\geq \frac{1}{\mu}\inf\limits_{[0,L]}\left|\frac{\delta}{3}\right|(\text{ch}(\mu L)-1).
\end{equation}
Using the first condition of \eqref{eta_cond},
\begin{equation}
    \frac{1}{\mu}\inf\limits_{[0,L]}\left|\frac{\delta}{3}\right|(\text{ch}(\mu L)-1)\leq 1.
\end{equation}
As $\inf_{[0,L]}|\delta|>0$, there exists $\mu_{1}>0$ such that there is contradiction.

\subsection{Outline of the strategy}\label{outline-of-proof}

We have introduced a virtual control operator in order to obtain a controllable system, by adding a component to the control profile $\mathcal{I}$ along the null eigenfunction of the operator $\mathcal{A}$. As it turns out, this can also be understood as a dynamical extension of the linearized water tank system \eqref{sv-lin}--\eqref{bc-sv}--\eqref{cond-cons-2}.

Indeed, thanks to the fact that $\Lambda \partial_x f_0+ \delta(s)J f_0=0$, projecting the system \eqref{sys1virt} on the $(f_{n})_{n\in\mathbb{Z}}$ and denoting
\begin{equation}
    \mathbf{Z}=\zeta_0 f_0 + \zeta, \text{   with  }\langle\zeta, f_0\rangle =0,
\end{equation}
one has
\begin{equation}
\dot{\zeta}_0=\nu\left(\langle \zeta, F \rangle + \zeta_0\langle  f_{0}, F\rangle\right),
\end{equation}
and
\begin{equation}\label{control-form}
\left\{\begin{aligned}
&\partial_{t}\zeta+\Lambda\partial_{x}\zeta+\delta J \zeta
=\left(\langle \zeta(t,\cdot), F \rangle+\zeta_0\langle f_{0}, F\rangle\right)\mathcal{I}, \\
&\zeta_{1}(t,0)=-\zeta_{2}(t,0),\\
&\zeta_{2}(t,L)=-\zeta_{1}(t,L).
\end{aligned}\right.
\end{equation}

We will show later (see Remark \ref{rmkcontrol}) that the condition \eqref{missing-direc-sys1} corresponding to the missing direction is equivalent to $\langle\zeta,f_{0} \rangle =0$. Thus $\zeta$ is exactly the solution of \eqref{sys01}, \eqref{cond-0}, \eqref{missing-direc-sys1} with control $u_{\gamma}(t):=\left(\langle \zeta(t,\cdot), F\rangle+\zeta_0\langle f_{0}, F\rangle\right)$. This control is indeed an implementable feedback law as it depends only on the state $\zeta$ and on the additional variable $\zeta_0$, which is a dynamical extension of the original system (see, for example, \cite{CoronPraly} and  \cite[Section 11.1]{CoronBook}). Physically speaking, we can think of it as ``virtual mass'' that is added or removed from the real physical system described by $\zeta$.

Then the key result to prove Theorem \ref{th1} is the following:
\begin{prpstn}\label{Prop-feedback-for-virtual-system}
For any $\mu>0$, there exists $\gamma_{0}>0$ such that, for any $\gamma\in(0,\gamma_{0})$, there exists $\nu \neq 0$ such that $F\in \dist$ given by
\begin{equation}
    \begin{aligned}
    \langle f_0, F \rangle&= -2\tanh(\mu L) \frac{(f_{0,1}(0))^{2}}{\nu},\\
    \langle f_n, F \rangle&=-2\tanh(\mu L) \frac{(f_{n,1}(0))^{2}}{\langle \mathcal{I}, f_n \rangle}, \quad \forall n \in \Z^\ast,
    \end{aligned}
\end{equation}
stabilizes \eqref{sys1virt} exponentially  in $D(\mathcal{A})$, with decay rate $3\mu/4$.
\end{prpstn}

\begin{rmrk}
We have introduced a virtual control operator in order to deal with completely controllable systems, which is a crucial point of our proof.

However, note that the control operator
\begin{equation}
    B:u\in \CC    \mapsto u \mathcal{I}_\nu
\end{equation}
is not a bounded operator in the state space $D(\mathcal{A})$, as $\mathcal{I}_\nu\notin D(\mathcal{A})$. Otherwise, rapid stabilizability of the virtual system \eqref{sys1virt} would be a direct consequence of controllability thanks to \cite[Corollary 1, p.13]{DATKO1971346} (see also \cite[item (i), Theorem 3.3, p.227]{Zabczyk}).

Rather, Proposition \ref{Prop-feedback-for-virtual-system} extends the implication ``controllability $\implies$ stabilizability'' to a case where the control operator is unbounded, as is done in \cite{Schrodinger, Rebarber, Urquiza, ZhangRapidStab}.
\end{rmrk}
Indeed, to design a feedback ensuring the exponential stability of \eqref{sys01}, \eqref{cond-0}, \eqref{missing-direc-sys1}, it suffices to design a linear feedback ensuring the exponential stability of the virtual system \eqref{sys1virt}. Then, for any initial condition of system \eqref{sys01}, \eqref{cond-0},
\begin{equation}\label{integrator-starts-at-0}
  \zeta^0  \in D(\mathcal{A}), \quad \langle \zeta^0, f_0 \rangle =0,
\end{equation}
one can implement that feedback on system \eqref{sys1virt} with initial condition
\begin{equation*}
    \zeta_0(0)=0, \quad \zeta(0)=\zeta^0.
\end{equation*}
This will yield, in particular, exponential stabilization of the $\zeta$ part.

Then, as the transformations \eqref{diffeo0} and \eqref{diffeo} and the scaling introduced in \eqref{sys0} and \eqref{scal} define a diffeomorphism, this implies the exponential stability of the initial system \eqref{sv-lin}--\eqref{cond-cons-2}, albeit with a different decay rate (because of the time scaling in \eqref{scal}). Finally, the inverse transformations of \eqref{diffeo0} and \eqref{diffeo} allow us to recover \eqref{F} in the original variables from the definition of $F$ given by \eqref{Fcoeff} (this is given in more detail in Appendix \ref{appendixC}), which completes the proof of Theorem \ref{th1}.

The following sections are thus devoted to the proof of Proposition \ref{Prop-feedback-for-virtual-system}, by applying the backstepping method described in Section \ref{FiniteD} to virtual system \eqref{sys1virt} and target system \eqref{target}.

Section \ref{s3} handles the controllability of both systems, proving that there exists $\nu \neq 0$ such that System  \eqref{sys1virt} (Lemma \ref{lem-3-5}), and System \eqref{target-control} (Lemma \ref{til-lem-fond}) are both controllable.

Section \ref{s4} builds candidates for a suitable backstepping transformation for our problem, and gives a sufficient condition to find such a backstepping transformation.

Section \ref{s5} then builds a suitable backstepping transformation, along with the associated feedback law, and checks that this feedback law is indeed exponentially stabilizing in some sense.

\section{Controllability}\label{s3}
The goal of this section is to achieve some controllability results for the original system and the target system. Thanks to the moment theory, those results can be obtained by several estimates, such as Lemma \ref{lem-fond}, Lemma \ref{lem-3-5}, Lemma \ref{tild-lem-fond},  \eqref{psi-sca-pro}, \eqref{Iboundary}, \eqref{tild-es-bd-con-III}, which are exactly some  of the key points needed to solve our stabilization problem.

When $\gamma=0$ the target system \eqref{target-control} becomes quite simple (see \eqref{Atilde-order0}) and obviously controllable. Hence, it is rather easy to prove its controllability with $\gamma>0$ small. On the other hand, as already mentioned, the initial system \eqref{sys01}--\eqref{cond-0} is not controllable when $\gamma=0$. 
Therefore, in this section we mainly focus  on  the controllability of system \eqref{sys01}--\eqref{cond-0} which will be the object of Section \ref{sec-con-or} and Corollary \ref{cor-con-or}. Then, as the operator $\mathcal{A}$ given by \eqref{A} and the operator $\widetilde{\mathcal{A}}$ given by \eqref{Atilde} share many common properties, almost all the calculations and estimates in Section \ref{sec-def-eig}--\ref{sec-2.2} also hold for \textcolor{black}{$\widetilde{\mathcal{A}}$}, which will lead to the controllability of System \eqref{target-control}.This will be the goal of  Section \ref{sec-con-diss} and Theorem \ref{tmm-diss-con}.

\subsection{Controllability of System \eqref{sys01}--\eqref{cond-0}}\label{sec-con-or}
As system \eqref{sys01}-- \eqref{cond-0}
  is obtained from \eqref{w-con} 
by an isomorphism, it suffices to prove the controllability of System \eqref{w-con}, while for ease of notation we simply replace the control by $u(t)$. In the following  we mainly focus on System \eqref{w-con}.

\subsubsection{Asymptotic calculation: perturbed operators and perturbed eigenfunctions}\label{sec-def-eig}
Let us define the operator
\begin{equation}
     T:= \Lambda \partial_x,
 \end{equation}
\begin{equation}\label{T-gamma}
  T_{\gamma}:=  \mathcal{S}_0=\Lambda \partial_x+ \delta(x)J_0,
\end{equation}
\begin{equation}
    D(\mathcal{S}_0):= \left\{(w_1, w_2)\in (H^{1})^2: w_1(0)= -w_2(0), w_1(L)= -w_2(L)\right\},
\end{equation}
associated to the system \eqref{w-con}. 
We want to find an asymptotic expression of eigenvalues and eigenfunctions of this operator.
The adjoint operator $\mathcal{S}_0^\ast$ is given by
\begin{equation}
    \mathcal{S}_0^\ast :=-\Lambda \partial_x+ \bar \delta(x) J^*_0,
\end{equation}
\begin{equation}
    D(\mathcal{S}^*_0)= D(\mathcal{S}_0).
\end{equation}
Hence, $\mathcal{S}_0$ is neither self-adjoint nor anti-adjoint. In fact, it is  not even a normal operator:
\begin{equation}
    \mathcal{S}_0\mathcal{S}_0^*- \mathcal{S}_0^*\mathcal{S}_0= 2 \delta_{x}(x).
\end{equation}

Note that
this operator has the same eigenvalues $ \mu_{n}(\gamma)$ as
the operator $\mathcal{A}$, thanks to the change of variables \eqref{diffeo}.

Concerning the relation between $\mathcal{S}_0$ and $\mathcal{S}_0^*$, if $(\psi_{n}(\gamma), \mu_{n}(\gamma))$ are eigenfunctions and eigenvalues of $\mathcal{S}_0$, then  $(\chi_{n}(\gamma), \overline{\mu_n}(\gamma))$ are eigenfunctions and eigenvalues of $\mathcal{S}_0^*$. Moreover, these two group of eigenfunctions are bi-orthogonal in the sense that,
\begin{equation}\label{ortho}
    \big\langle \psi_{n}(\gamma), \chi_m(\gamma)\big\rangle= 0, \textrm{ if } n\neq m.
\end{equation}
Direct calculation implies that the eigenvalues of $T$   are  simple and isolated,
\begin{equation}
\mu_{n}^{(0)}= i\pi n/L.
\end{equation}
The related normalized eigenfunctions of $T$ and $T^*$ are given by
\begin{equation}
    \psi_{n}^{(0)}=(e^{i\pi nx/L}, -e^{-i \pi nx/L})\; \textrm{ and } \;\chi_{n}^{(0)}=(e^{-i\pi nx/L}, -e^{i \pi nx/L}).
\end{equation}

Similarly, the eigenvalues of $T_{\gamma}$ are denoted by  $\mu_{n}(\gamma)$, which is also isolated provided $\gamma$ small enough.  In fact, due to the fact that an eigenfunction multiplied by a scalar number is still an eigenfunction, it is convenient to consider \textit{perturbed eigenfunctions}, $\psi_{n}(\gamma)$, according to Kato's book \cite[page 92]{Kato-book}:
\begin{equation}
\langle \psi_{n}(\gamma), \chi_{n}^{(0)}\rangle=1,\text{  }(\text{resp. } \langle \psi_{n}^{(0)}, \chi_{n}(\gamma)\rangle=1).
\end{equation}
This normalization formula is standard and is convenient to perform symbol calculation.  For this reason we cannot  assume $\langle \psi_{n}(\gamma), \chi_{n}(\gamma)\rangle=1$ at the same time, see \eqref{ortho}.

\subsubsection{$L^2$-normalized eigenfunctions and Riesz basis}\label{sec-per-Riesz}
We are also interested in the $L^2$--normalized eigenfunctions:
\begin{equation}
\hat \psi_{n}(\gamma):=    \psi_{n}(\gamma)/ \lVert \psi_{n}(\gamma)\lVert_{(L^2)^2},  \hat \chi_{n}(\gamma):=    \chi_{n}(\gamma)/ \lVert \chi_{n}(\gamma)\lVert_{(L^2)^2}.
\end{equation}
The following theorem by Russell tells us that those eigenfunctions form a Riesz basis when $\gamma$ is sufficiently small.
\begin{thrm}[Russell \cite{Russell2}]\label{tho-rus} There exists $r_R>0$ such that, for any $\gamma\in (-r_R, r_R)$ the following properties hold.
\begin{itemize}
    \item[(1)] $T_{\gamma}$ has simple isolated eigenvalues $\mu_{n}(\gamma)$.  Moreover,
     \begin{gather}\label{mu_n asymp-lam}
     \mu_{n}(\gamma)= \frac{i \pi n}{L} + \gamma \mathcal{O}\left(\frac1n\right), \quad \forall n\in \Z^*.
     \end{gather}
        \item[(2)] Both $\{\hat\psi_{n}(\gamma)\}_n$ and $\{\hat\chi_{n}(\gamma)\}_n$  form a Riesz basis 
of $(L^{2}(0,L))^2$.
 \item[(3)] The   functions $\{e^{\mu_n(\cdot- 2L)}\}_n$ form a Riesz basis for $L^2(0, 2L)$, while its dual bi-orthogonal basis $\{p_n\}_n$ also form a Riesz basis.
\end{itemize}
\end{thrm}

\subsubsection{The moments method and the controllability of  System \eqref{w-con}}\label{sec-mom}
The moments method consists in decomposing the state and the control term in a Riesz basis of eigenfunctions. This yields an infinity of independent ODEs, and leads to a moments problem. The main issues, as indicated above, are to check that  the eigenfunctions form a Riesz basis, and that the projection of the control term on each direction is away from 0 (hence observable).  We also refer to the book \cite{moment-book} for a good introduction of this method.

This suggests us to study
\begin{equation}\label{moment-bas}
     a_n:= \big\langle  \psi_{n}(\gamma), (1,1)\big \rangle  \textrm{ and }b_n:= \big\langle \chi_{n}(\gamma), (1,1)\big\rangle,\text{ }\forall n\in\mathbb{Z},
\end{equation}
Indeed, at least formally,  we are able to decompose $(1, 1)$ by
\begin{equation}\label{decomp0}
   (1,1)= \sum_{n\in\mathbb{Z}} d_n \psi_{n}(\gamma),
\end{equation}
thus, using \eqref{ortho}
\begin{equation}\label{dn-bn}
    d_n \big\langle \psi_{n}(\gamma), \chi_{n}(\gamma)\big \rangle=   \big\langle (1, 1), \chi_{n}(\gamma) \big\rangle= \bar b_n.
\end{equation}

To \textit{observe} the  ``direction" generated by  $\psi_n(\gamma)$, the value  $b_n$ should be nonzero.  We further get the controllability provided that the values of $\{\psi_n(\gamma)\}_n$ satisfying some inequalities, see \cite{moment-book}  for more details of this method, and see \eqref{u-mom-pres} for the construction of the control in our case. \\

\noindent\textbf{Not controllable for $\gamma$=0.}

For the case $\gamma=0$, thanks to the explicit expression,     $\chi_n(0)= \chi_{n}^{(0)}=(e^{-i\pi nx{\color{black}/L}}, -e^{i \pi nx{\color{black}/L}})$, we have,
\begin{equation}
    \big\langle \chi_n(0), (1, 1)\big\rangle= -\frac{2iL}{\pi n} (1-\cos{\pi n}).
\end{equation}
The preceding inequality implies that there are infinitely many direction that can not be observed, more precisely the space generated by $\{\psi_{2n}(0)\}_n$. Consequently,  with another argument than the one given in \cite{Coron2002} we prove that the system is not controllable with the uncontrollable part being of infinite dimension. \\

\noindent\textbf{Controllable for $\gamma>0$ sufficiently small.}

For that moment, let us assume that the following key lemma holds (it will be proved in Sections \ref{sec-2.2}--\ref{sec-2.3}):
\begin{lmm}\label{lem-fond}
There exists $\gamma_0,$ $c,$ $C>0$ such that for any $\gamma\in (0, \gamma_0)$ we have
\begin{itemize}
    \item[(i)] $(\psi_{n}(\gamma))_{n\in\mathbb{Z}} \textrm{ is a Riesz basis  of $L^{2}$}$;
    \item[(ii)] $\lvert\langle \psi_{n}(\gamma), \chi_{n}(\gamma) \rangle\lvert\in (1/2, 2)$;
    \item[(iii)] $| \mu_{n}(\gamma)-\mu_{n}^{(0)}|< \frac{1}{4L}$;
    \item[(iv)]  For $n>0$, $b_n$ is away from zero in the following sense
    \begin{gather}
b_{0}=0,
       \\
         \gamma  \frac{c}{n} <|
       b_{n}
       |<\frac{C}{n}, \forall n \in \mathbb{Z}^*.
    \end{gather}
\end{itemize}
\end{lmm}
Thanks to Lemma \ref{lem-fond} and the classical moment theory, we can conclude that the system is not yet controllable but there is only one dimension missing corresponding to the moment $b_{0}$. From \eqref{decomp0}, the missing direction corresponds therefore to $\text{Span}\{\psi_{0}\}$. In fact we will show later on that this missing direction corresponds exactly to the condition \eqref{missing-direc-w-con} which is the condition of mass conservation in the original system \eqref{cond-cons-2} (see Remark \ref{rmkcontrol}).
This also means that any state that keeps a constant mass is reachable, more precisely
\begin{thrm}\label{tho-con-w}
 For $T\in [2L,+\infty)$, for $ \gamma\in (0, \gamma_0)$,  system \eqref{w-con} is $D(\mathcal{H}^s_{(0)})$ controllable with controls having $D_p(\mathcal{H}_{(0)}^{s-1}(0, 2L))$ regularity, where
 \begin{gather}
    D(\mathcal{H}^s_{(0)})=\{f:= \sum_{n\in \mathbb{Z}^*}f_n \psi_n \in H^{s} | \sum_{n\in \mathbb{Z}^*} (1+ n^{2s}) f_n^2< +\infty\}, \\
  D_p(\mathcal{H}^s_{(0)}(0, 2L))=\{f:= \sum_{n\in \mathbb{Z}^*}f_n \bar{p}_n \in H^{s} | \sum_{n\in \mathbb{Z}^*} (1+ n^{2s}) f_n^2< +\infty\}, \\
  \textrm{$\{p_n\}_n$ is the dual basis of $\{e^{\mu_n(s-2L)}\}_n$ presented in Theorem \ref{tho-rus}},\\
  \textrm{the control vanishes on $(2L, T)$.}
 \end{gather}
\end{thrm}

\begin{proof}[Proof of Theorem \ref{tho-con-w}]
Indeed, for ease of notations, we simply denote System \eqref{w-con} by
\begin{equation}
    w_t(t)+ A w(t)= u(t)(1, 1),\; t\in (0, 2L).
\end{equation}
Moreover, it suffices to prove the controllability for the case $T=2L$, since in  the other cases we can simply take $u|_{(2L, T)}=0$.
Because the above equation is linear, it further suffices to  assume that the initial state as $w(0, \cdot)=0$ and the final state as $w(2L, \cdot)= \sum_n k_n \psi_n\in D(\mathcal{H}^s_{(0)})$.

Thus it allows us to decompose the state in the Riesz basis $\{\psi_n\}_n$
\begin{equation}
    w(t, \cdot)= \sum_{n\in \mathbb{Z}^*} w_n(t)\psi_n,
\end{equation}
which satisfies
\begin{equation}
    \sum_n\big((w_n)_t(t)+ \mu_n w_n(t)- d_n u(t)\big)\psi_n=0.
\end{equation}
Therefore
\begin{equation}
    (w_n)_t(t)+ \mu_n w_n(t)- d_n u(t)=0, \;\forall n\in \mathbb{Z}^*.
\end{equation}
Hence Duhamel's formula yields
\begin{equation}
    k_n= w_n(2L)=  d_n\int_0^{2L} e^{\mu_n(s-2L)} u(s) ds.
\end{equation}
Thanks to Theorem \ref{tho-rus}, more precisely the duality of $\{e^{\mu_n(\cdot-2L)}\}_n$ and $\{p_n\}_n$, the control $u(s)|_{[0, 2L]}$ can be given by
\begin{equation}\label{u-mom-pres}
    u|_{[0, 2L]}(s)= \sum_{n\in \mathbb{Z}^*} \frac{k_n}{d_n} \bar{p}_n= \sum_{n\in \mathbb{Z}^*} \frac{k_n \langle \psi_{n}(\gamma), \chi_{n}(\gamma) \rangle}{\bar{b}_n} \bar{p}_n\in D(\mathcal{H}^{s-1}_{(0)}).
\end{equation}
\end{proof}

\subsubsection{Asymptotic calculation: holomorphic extension}
 We will now use Kato's method \cite{Kato-book} of asymptotic calculation with the help of complex analysis to obtain an explicit formulation, and remainder estimates, for the eigenvalues, and an insight on the eigenfunctions' asymptotic behavior.  For ease of the presentation and notations, in the following we only focus on the operator $  \mathcal{S}_0= T_{\gamma}$ while its adjoint $\mathcal{S}_0^*=T_{\gamma}^*$ can be treated the same way.

 Let us consider eigenfunctions on the space $(C^{0}([0, 1]))^2$. From now on, $\lVert \cdot \lVert_{\infty}$ denotes the $L^{\infty}(0, L)^2$ norm, $i.e.$ for $u=(u_1, u_2)$,
 \begin{equation}
     ||u||_{\infty}=  ||u||_{(L^{\infty}(0, L))^2}:= \max\{||u_1||_{L^{\infty}(0, L)}, ||u_1||_{L^{\infty}(0, L)}\}.
 \end{equation}
 In the preceding formulas, $\gamma$ was assumed to be a (sufficiently small) real number. Now we extend those formulas to $\gamma \in \mathbb{C}$ with $|\gamma|$ small: at least formally this extension is true.  In fact, this complex extension enables us to use holomorphic techniques concerning asymptotic calculation. Once we will get estimates that we require, we will apply them with a real $\gamma$, as several properties are better in this case, for example $\mathcal{A}$ is skew-adjoint provided that $\gamma$ is real. Moreover, the operators $\{T_{\gamma}\}$ are of \textit{type (A)} (see \cite[Chapter 7, Section 2]{Kato-book}), hence the extended formulas are holomorphic for $|\gamma|$ small.
\begin{rmrk}
We choose the $(L^{\infty})^2$ norm for asymptotic information on boundary points, because it will be useful in the following demonstration. Getting an estimation on the $(L^2)^2$ norm, though, would be much simpler.
\end{rmrk}
Let us define,
\begin{gather}
    A(\gamma):= T_{\gamma}-T= \delta(x) J_0,\\
    A^{(1)}(\gamma)= T_{\gamma}-T- \gamma  T^{(1)}, \text{ }\text{ } T^{(1)}= -\frac{3}{4} J_0,
\end{gather}
where $T_{\gamma}$ is still given by \eqref{T-gamma}. Using \eqref{def_delta}--\eqref{es-delta-O2} we immediately get the existence of  $d\in (0, 1/(8L^2))$ such that, for any  $|\gamma|< d$,
\begin{gather}
    |\delta(x)|< |\gamma|, \; |\delta(x)+\frac{3}{4} \gamma|< L |\gamma|^2,
\end{gather}
thus the linear operators $A(\gamma)$ and $A^{(1)}(\gamma)$ on $(L^{\infty})^2$ space verifies
\begin{gather}\label{A-gamma-es}
    \lVert A(\gamma)\lVert< 2|\gamma|, \;  \lVert A^{(1)}(\gamma)\lVert< 2L|\gamma|^2.
\end{gather}
\begin{thrm}
The resolvent $R(\xi):= (T-\xi Id)^{-1}$ is defined on $D_0:= \mathbb{C}\setminus \bigcup\limits_{n\in\mathbb{Z}} \mu_n^{(0)}$. Besides, $T$ has compact resolvent, i.e. $R(\xi)$ is compact for all $\xi \in D_0$.
\end{thrm}
The proof is straightforward: let us suppose that  $\xi \in D_0$, and that  $v=R(\xi)u$. If we write down $u$ and $v$ by
\begin{equation}
   u =\sum\limits_{n\in\mathbb{Z}} a_n \psi_{n}^{(0)} \textrm{ and } v=\sum\limits_{n\in\mathbb{Z}} b_n \psi_{n}^{(0)},
\end{equation}
 then by comparing the coefficient of $\psi_n$ we get
\begin{equation}\label{R-exp}
    b_n= \frac{a_n}{\mu_{n}^{(0)}-\xi}.
\end{equation}
By slightly changing the notations and  basically following the same calculations as in Kato \cite{Kato-book}, we get the following result.
\begin{thrm}\label{thm-t-g}
 For any $n\in \mathbb{Z}$ there exists $d_n>0$ (convergence radii) such that
 \begin{itemize}
     \item[(i)] (Kato, \cite[page 377, Theorem 2.4]{Kato-book}) $T_{\gamma}$ has compact resolvent;
     \item[(ii)]  (Kato, \cite[page 382, Example 2.14]{Kato-book}) $\mu_{n}(\gamma)$ and $\psi_{n}(\gamma)$ are holomorphic in $B_{d_n}= \{z\in \mathbb{C}: |z|\leq d_n\}$;
 \end{itemize}
\end{thrm}
In fact $d_n\geq d$: let us define
\begin{gather}
    \mathbb{D}:= \bigcup_{n} B(i \pi n/L; 1/L), \textrm{ with } B(a; r):= \{x\in \mathbb{C}; |x-a|<r\}, \\
    \Gamma_n:= \{x\in \mathbb{C}; |x- i \pi n/L|= 1/L\},
\end{gather}
and the perturbation of the resolvent
\begin{equation}
    R(\xi, \gamma):= (T_{\gamma}- \xi Id)^{-1}.
\end{equation}
Simple symbol calculation leads to the second Neumann series of $R(\xi, \gamma)$:
\begin{equation}
    R(\xi, \gamma)= R(\xi)\big(1+ A(\gamma) R(\xi)\big)^{-1}.
\end{equation}
In order to justify the above formal calculations upon holomorphic coefficient $\gamma$, it suffices to have $\lVert A(\gamma) R(\xi)\lVert< 1$.

For $\xi \in \mathbb{C}\setminus \mathbb{D}$, we estimate the norm of $R(\xi)$. Notice that for   $u= (u_1, u_2)= \sum\limits_{n\in\mathbb{Z}} a_n \psi_{n}^{(0)}\in (\mathcal{L}_{(0)})^{2}$, we have
\begin{equation}
    ||u||_{(L^2)^2}^2= \frac{1}{2L}\int_0^L |u_1|^2(x)+|u_2|^2(x) dx\leq ||u||_{(L^{\infty})^2}^2= ||u||_{\infty}^2.
\end{equation}
Moreover,  we have the following
\begin{equation}
   \sum\limits_{n\in\mathbb{Z}}\frac{1}{\lvert \mu_n^{(0)}-\xi\rvert^2}\leq 4L^2.
\end{equation}
Indeed, observing the periodicity of $\mathbb{C}\setminus \mathbb{D},$ we may assume that Im$\; \xi\in [0, \pi n/L)$ and further get
\begin{align}
 \sum\limits_{n\in\mathbb{Z}}\frac{1}{\lvert \mu_n^{(0)}-\xi\rvert^2}&= \sum\limits_{n\in\mathbb{Z}\setminus\{0, 1\}}\frac{1}{\lvert \mu_n^{(0)}-\xi\rvert^2}+ \frac{1}{\lvert\xi\rvert^2}+ \frac{1}{\lvert \mu_1^{(0)}-\xi\rvert^2},  \\
 &\leq 2\sum\limits_{n\in\mathbb{N}^*}\frac{L^2}{(\pi n)^2}+ 2 L^2, \\
 &=7L^2/3.
\end{align}
Therefore,
thanks to the explicit expression of $\psi_n^{(0)}$, we have
\begin{equation}
\lVert R(\xi) u\lVert_{\infty} \leq \sum\limits_{n\in\mathbb{Z}} \left|\frac{a_n}{\mu_n^{(0)}- \xi}\right|\leq \left(\sum\limits_{n\in\mathbb{Z}} |a_n|^2\right)^{\frac{1}{2}} \left(\sum\limits_{n\in\mathbb{Z}} \frac{1}{\lvert \mu_n^{(0)}-\xi\rvert^2}\right)^{\frac{1}{2}}< 2L \lVert u\lVert_{(L^2)^2} \leq  2L \lVert u\lVert_{\infty},
\end{equation}
which gives $\lVert R(\xi) \lVert < 2L$.

On the other hand, for $|\gamma|< d$, we have
\begin{equation}\label{est-A-gamma}
    \lVert A(\gamma)\lVert< 2|\gamma|< 1/(4L^2).
\end{equation}
Therefore
\begin{equation}
\lVert A(\gamma) R(\xi)\lVert< 1\text{ for }\xi \in \mathbb{C}\setminus \mathbb{D}\text{ and }|\gamma|< d.
\end{equation}
Hence $R(\xi, \gamma)$ is holomorphic.  From now on, we will always let $\xi$ and $\gamma$ be in this domain to guarantee convergences of calculations. In the following subsection our aim will be to derive asymptotic estimates on $\mu_{n}(\gamma)$ and $\psi_{n}(\gamma)$, first by direct estimation, then using majoring series.
\subsubsection{Asymptotic calculation: direct estimation}
In order to estimate $\mu_{n}(\gamma)$ and $\psi_{n}(\gamma)$, we need to decompose $R(\xi)$.
More precisely, for each $\mu_{n}^{(0)}$, we have the following Laurent series:
\begin{equation}
    R(\xi)= -\big(\xi-\mu_{n}^{(0)}\big)^{-1} P_n+ \sum_{m=0}^{+\infty} \big(\xi-\mu_{n}^{(0)}\big)^m S_n^{m+1},
\end{equation}
with
\begin{equation}
    P_n:= -\frac{1}{2\pi i} \int_{\Gamma_n} R(\xi) d \xi,
\end{equation}
\begin{equation}
     S_n:= \frac{1}{2\pi i} \int_{\Gamma_n} \big(\xi-\mu_{n}^{(0)}\big)^{-1} R(\xi) d \xi.
\end{equation}
Moreover, thanks to the explicit formula \eqref{R-exp}, we get
\begin{equation}
    P_n u = \big\langle u, \psi_{n}^{(0)}\big\rangle \psi_{n}^{(0)},
\end{equation}
\begin{equation}\label{S-gap}
    S_n u= \sum_{k\neq n} \frac{\big\langle u, \psi_k^{(0)} \big\rangle}{\mu_k^{(0)}-\mu_{n}^{(0)}} \psi_k^{(0)}.
\end{equation}
The following lemma gives the expansion of eigenvalues.
\begin{lmm}\label{lem-es-mu-n-mu}  According to \cite[Chapter 2 Section 3.1]{Kato-book}, one has
\begin{equation}
   \big\lvert \mu_{n}(\gamma)- \mu_{n}^{(0)}-\sum_{p=1}^m\gamma^p \mu_n^{(p)} \big\lvert\leq \frac{\lvert \gamma\lvert^{m+1}}{d^m (d-\lvert \gamma\lvert)}, \;\forall |\gamma|< d.
\end{equation}
\end{lmm}
Therefore we conclude the estimates on the eigenvalues,
\begin{equation}\label{est-mu-n-mu}
    \big \lvert \mu_{n}(\gamma)- \mu_{n}^{(0)} \big\lvert\leq \frac{3|\gamma|}{2d} \leq \frac{1}{2L}, \;\forall |\gamma|< \min \{\frac{d}{3L}, \frac{d}{3}\}.
\end{equation}

Now we turn to the eigenfunctions.  The explicit asymptotic formulation of $\psi_{n}(\gamma)$ reads
\begin{align}
\psi_{n}(\gamma)= \psi_{n}^{(0)}-\gamma S_nT^{(1)}\psi_{n}^{(0)}+...
=\psi_{n}^{(0)}+\gamma  \psi^{(1)}_n+ \gamma^2  \psi^{(2)}_n...
\end{align}
Thanks to \eqref{S-gap}, $\psi^{(k)}_n$ can be explicitly characterized and calculated, $e.g.$
the first order term is given by
\begin{equation}\label{varphi-1-exp}
    \psi^{(1)}_n= -S_nT^{(1)}\psi_{n}^{(0)}= \frac{3}{4} S_n J_0 \psi_{n}^{(0)}=\frac{3L}{4} \sum_{k\neq n} \frac{\big\langle J_0\psi_{n}^{(0)}, \psi_k^{(0)}  \big\rangle}{i \pi (k-n)} \psi_k^{(0)}.
\end{equation}
The remainder of the zeroth order reads as
\begin{equation}\label{first-rem}
    \psi_{n}(\gamma)- \psi_{n}^{(0)}= - S_n \Big( 1+ \big(A(\gamma)-\mu_{n}(\gamma)+\mu_{n}^{(0)}\big)S_n\Big)^{-1}A(\gamma) \psi_{n}^{(0)}.
\end{equation}
Let us define
$s_n:=\lVert S_n\lVert$. Suppose that  $u= \sum_k a_k\psi_k^{(0)}\in (L^{\infty})^2$, then
\begin{equation}\label{es-sn-norm}
\lVert S_n u\lVert_{\infty} \leqslant \sum_{k\neq n} \frac{|a_k|}{\big|\mu_k^{(0)}-\mu_n^{(0)}\big|}\leqslant \frac{L}{\pi}\Big(\sum_k |a_k|^2\Big)^{1/2}\Big(\sum_{k\neq n}\frac{1}{(n-k)^2}\Big)^{1/2}\leqslant L\lVert u\lVert_{L^2}\leqslant L\lVert u\lVert_{\infty}.
\end{equation}
On the other hand, choosing $u=\psi_{n+1}^{(0)}$, we get
\begin{equation}
    S_n \psi_{n+1}^{(0)}= \frac{L\psi_{n+1}^{(0)}}{i \pi}.
\end{equation}
Therefore
\begin{equation}\label{est-s-n}
    L/\pi \leq s_n\leq L, \; \forall n\in \mathbb{Z}.
\end{equation}
Combining \eqref{first-rem}, \eqref{est-s-n}, \eqref{est-A-gamma} and \eqref{est-mu-n-mu}, we get
\begin{equation}\label{varphi-0}
    \lVert  \psi_{n}(\gamma)- \psi_{n}^{(0)} \lVert_{\infty}\leq 4L|\gamma|, \;\forall |\gamma|< d^{(1)}:= \min \left\{\frac{d}{3L}, \frac{d}{3}\right\}.
\end{equation}

Let us continue to estimate the remainder of the first order.
\begin{align}
   &\;\;\;  \;\;\; \psi_{n}(\gamma)- \psi_{n}^{(0)}- \gamma \psi_{n}^{(1)}\notag \\
      &= - S_n \Big( 1+ \big(A(\gamma)-\mu_{n}(\gamma)+\mu_n^{(0)}\big)S_n\Big)^{-1}A(\gamma) \psi_{n}^{(0)}- \gamma \frac{3}{4} S_n J_0 \psi_{n}^{(0)}\notag \\
       &=  S_n \Big( 1+ \big(A(\gamma)-\mu_{n}(\gamma)+\mu_n^{(0)}\big)S_n\Big)^{-1}\left(\frac{3}{4}\gamma J_0- A^{(1)}(\gamma)\right) \psi_{n}^{(0)}- \gamma \frac{3}{4} S_n J_0 \psi_{n}^{(0)}\notag \\
      &= \gamma \frac{3}{4} \Big(S_n \Big( 1+ \big(A(\gamma)-\mu_{n}(\gamma)+\mu_n^{(0)}\big)S_n\Big)^{-1}- S_n \Big)J_0\psi_{n}^{(0)} \notag\\
      &\;\;\;\;\;\;\;\;\;\;\;\;\;\;\; -S_n \Big( 1+ \big(A(\gamma)-\mu_{n}(\gamma)+\mu_n^{(0)}\big)S_n\Big)^{-1}A^{(1)}(\gamma) \psi_{n}^{(0)}\notag\\
      &= -\gamma \frac{3}{4} S_n \left( \big(A(\gamma)-\mu_{n}(\gamma)+\mu_n^{(0)} \big) S_n\right) \Big( 1+ \big(A(\gamma)-\mu_{n}(\gamma)+\mu_n^{(0)}\big)S_n\Big)^{-1} J_0 \psi_{n}^{(0)}    \notag\\
       &\;\;\;\;\;\;\;\;\;\;\;\;\;\;\; -S_n \Big( 1+ \big(A(\gamma)-\mu_{n}(\gamma)+\mu_n^{(0)}\big)S_n\Big)^{-1}A^{(1)}(\gamma) \psi_{n}^{(0)}.\label{rem-seond}
\end{align}
Combining \eqref{A-gamma-es}, \eqref{first-rem}, \eqref{est-s-n},  \eqref{est-A-gamma} and \eqref{est-mu-n-mu},  we find that the preceding formula is of order $\gamma^2$,  thus get the existence of $d^{(2)}>0$ and $C^{(2)}$ which are independent of $n\in \mathbb{Z}$ such that
\begin{equation}\label{es-vn-final0}
      \lVert  \psi_{n}(\gamma)- \psi_{n}^{(0)}- \gamma \psi_{n}^{(1)} \lVert_{\infty} \leq C^{(2)}|\gamma|^2 , \;\forall |\gamma|< d^{(2)}.
\end{equation}

\begin{rmrk}
We observe that, by using this direct method, the best estimates that we can get are Lemma \ref{lem-es-mu-n-mu} and
\begin{equation}
    \lVert \psi_{n}(\gamma)- \psi_{n}^{\textcolor{black}{(0)}}- \gamma \psi_{n}^{(1)}- \gamma^2 \psi_{n}^{(2)}-...- \gamma^k \psi_{n}^{(k)} \lVert_{\infty} \leqslant C^{(k)} \lvert\gamma\lvert^{k+1}, \;\forall |\gamma|< d^{(k)}.
\end{equation}
Clearly there is no asymptotic behavior with respect to $n$.
\end{rmrk}

\subsubsection{Asymptotic calculation: majorizing series for better estimation}\label{sec-2.2}
The so called majorizing series provides a more  systematic way of  estimating remainder terms. This method is heavily used for  high order remainder terms estimates, since it is rather difficult to perform direct calculation as \eqref{rem-seond} then.  One can see \cite[Chapter 2, Section 3.2; or page 382, Example 2.14]{Kato-book} for \textcolor{black}{a} perturbation of \textcolor{black}{the} Laplace operator. \textcolor{black}{To give a comprehensive view of the method,} we start \textcolor{black}{by} considering a reduced case: the eigenvalues and eigenfunctions of
$T+ \gamma T^{(1)}$.
As suggested in \cite{Kato-book}, we define the following
\begin{equation}
  p_n:= \lVert \frac{3}{4} J_0 P_n \lVert,\;\; q_n:= \lVert \frac{3}{4} J_0 S_n \lVert,\;\; r_n:= [(p_n s_n)^{1/2}+ q_n^{1/2}]^{-2}.
\end{equation}

\begin{lmm}
There exists positive constants $c$ and $C$ independent of $n$ such that
\begin{equation}\label{spq}
   c < p_n, q_n, r_n, \lVert P_n \lVert< C, \forall n\in \mathbb{Z}.
\end{equation}
\end{lmm}
\begin{proof}
Suppose that  $u= \sum_k a_k\psi_k^{\textcolor{black}{(0)}}\in (L^{\infty})^2$, then
\begin{equation}
    \lVert P_n u\lVert_{\infty} =\lvert a_n\lvert \lVert \psi_{n}^{\textcolor{black}{(0)}} \lVert_{\infty} \leqslant  \lVert u \lVert_{L^2}\leqslant \lVert u\lVert_{\infty},
\end{equation}
which completes the right hand side inequality of \eqref{spq}.\\
On the other hand, for the left hand side of \eqref{spq},  by choosing $u$ as $\psi_{n}^{\textcolor{black}{(0)}}$ or $\psi_{n+1}^{\textcolor{black}{(0)}}$, successively we have
\begin{equation}
\begin{aligned}
&    S_n \psi_{n+1}^{\textcolor{black}{(0)}}= \frac{L \psi_{n+1}^{\textcolor{black}{(0)}}}{i \pi}, \text{  }\;\;  J_0S_n \psi_{n+1}^{\textcolor{black}{(0)}}=J_0 \frac{L \psi_{n+1}^{\textcolor{black}{(0)}}}{i \pi},\\
&    P_n \psi_{n}^{\textcolor{black}{(0)}}= \psi_{n}^{\textcolor{black}{(0)}},\text{  }\;\; J_0 P_n \psi_{n}^{\textcolor{black}{(0)}}=J_0 \psi_{n}^{\textcolor{black}{(0)}}.
    \end{aligned}
\end{equation}
\end{proof}

We obtain \textcolor{black}{the} asymptotic expression of $\psi_{n}(\gamma)$ \textcolor{black}{similarly as} in \cite[page 382 Example 2.14]{Kato-book} (which is based on majorizing series and \cite[Chapter 2 Section 3 Examples 3.2-3.7]{Kato-book}):
\begin{lmm}
For any $|\gamma|< \min \{r_n\}$ we have
\begin{equation}\label{vn-es-co}
    \lVert \psi_{n}(\gamma)-\psi_{n}^{\textcolor{black}{(0)}}- \gamma \psi^{(1)}_n \lVert_{\infty} \leqslant |\gamma|^2 \frac{s_n}{(p_nq_ns_n)^{1/2}} \left((p_ns_n)^{1/2}+ q_n^{1/2}\right)^2\left(p_ns_n+q_n+(p_nq_ns_n)^{1/2}\right).
\end{equation}
\end{lmm}
By inserting \eqref{spq} into \eqref{vn-es-co}, we get
\begin{equation}\label{es-vn-final}
     \lVert \psi_{n}(\gamma)-\psi_{n}^{\textcolor{black}{(0)}}- \gamma \psi^{(1)}_n \lVert_{\infty} \leqslant C |\gamma|^2, \forall |\gamma|< \min \{r_n\}, \forall n\in \mathbb{Z}.
\end{equation}
However, as we have seen that $p_n, q_n, s_n$ are bounded by constant\textcolor{black}{s} instead of $c/n$,
\textcolor{black}{when considering not only $T+ \gamma T^{(1)}$ but the whole $T_{\gamma}$,}
the best asymptotic estimate that we can get from majorizing series is (still)
\begin{equation}\label{Kato-es}
    \lVert \psi_{n}(\gamma)- \psi_{n}^{\textcolor{black}{(0)}}- \gamma \psi_{n}^{(1)}- \gamma^2 \psi_{n}^{(2)}-...- \gamma^k \psi_{n}^{(k)} \lVert_{\infty} \leqslant C^{(k)} \gamma^{k+1},
\end{equation}
with $C^{(k)}$ independent of $n$.
\begin{rmrk}[Comparison with Laplacian]
The asymptotic calculation for the Laplacian operator contains some factor  $1/n$ which comes from \eqref{S-gap}. More precisely, it comes from the localization of the eigenvalues and the fact that for the Laplacian operator case $\mu_{n}\sim n^{2}$. Indeed, for the same reason, we can expect some asymptotic behavior of eigenvalues' gap for fractional Laplacian $\Delta^s$ when $s>1$.
\end{rmrk}

\subsubsection{Controllability from  boundary conditions of eigenfunctions}\label{sec-2.3}
After having obtained  the asymptotic information of the perturbed eigenfunctions in the preceding sections, from now on we restrict ourselves to the case where $\gamma$ is in  $\mathbb{R}$.

With the two last subsections, we have an estimate on $ \lVert \psi_{n}(\gamma)- \psi_{n}^{\textcolor{black}{(0)}}- \gamma \psi_{n}^{(1)}- \gamma^2 \psi_{n}^{(2)}-...- \gamma^k \psi_{n}^{(k)} \lVert_{\infty}$ that it not totally satisfactory and that does not depend on $n$. However, notice that, from \eqref{dn-bn}, what \textcolor{black}{we need} to conclude the study of the controllability  is indeed the value of $\langle \psi_n, (1, 1)\rangle$ and $\langle \chi_n, (1, 1)\rangle$.  Therefore, the values of type:
\begin{equation}\label{v-n-asy-sp}
  \big \langle \psi_{n}(\gamma)- \psi_{n}^{\textcolor{black}{(0)}}- \gamma \psi_{n}^{(1)}- \gamma^2 \psi_{n}^{(2)}-...- \gamma^k \psi_{n}^{(k)}, (1,1) \big \rangle.
\end{equation}
Thus it sounds  natural  to investigate \eqref{v-n-asy-sp} directly.

Let us calculate  the boundary value of $\psi_{n}^{(1)}$, more precisely,  a combination of boundary values:
\begin{equation}
l^{(1)}_n:=
\psi_{n, 1}^{(1)}(1)- \psi_{n, 1}^{(1)}(0)-\psi_{n, 2}^{(1)}(1)+\psi_{n,2}^{(1)}(0).
\end{equation}
We observe that $l^{(0)}_{n}:=\psi_{n, 1}^{(0)}(1)- \psi_{n, 1}^{(0)}(0)-\psi_{n, 2}^{(0)}(1)+\psi_{n,2}^{(0)}(0) = 2 \big ((-1)^n-1 \big )$.  Direct calculation shows that
\begin{equation}
l^{(1)}_n= \frac{3L}{2 i\pi}  \sum_{k\neq n} \frac{\langle J_0\psi_{n}^{\textcolor{black}{(0)}}, \psi_k^{\textcolor{black}{(0)}}  \rangle}{ (k-n)}   \big ((-1)^k-1 \big ).
 \end{equation}
On the other hand we know that
\begin{equation}
 \big \langle J_0\psi_{n}^{\textcolor{black}{(0)}}, \psi_k^{\textcolor{black}{(0)}}  \big \rangle= \frac{(-1)^{n+k}-1}{i \pi} \Big(\frac{1}{n-k}+ \frac{1}{3}\cdot \frac{1}{n+k} \Big), \textrm{ if }|n|\neq |k|,
\end{equation}
\begin{equation}
 \big \langle J_0\psi_{n}^{\textcolor{black}{(0)}}, \psi_k^{\textcolor{black}{(0)}}  \big  \rangle= 0, \textrm{ if }|n|= |k|.
\end{equation}
Notice that if $n$ is odd we will have that $ \big ((-1)^k-1 \big ) \big ((-1)^{n+k}-1 \big )=0$, therefore $l^{(1)}_n=0.$  We \textcolor{black}{thus} only need to consider the case when $n=2m$ is even. Then
\begin{align}
 l^{(1)}_n=& -\frac{3L}{2\pi^2}\sum_{k\neq 2m}  \big ((-1)^k-1 \big )^2  \frac{1}{k-2m} \Big(\frac{1}{2m-k}+ \frac{1}{3}\cdot \frac{1}{2m+k} \Big) \\
 =& -\frac{6L}{\pi^2}\sum_{k\in \mathbb{Z}}  \frac{1}{2k+1-2m} \Big(\frac{1}{2m-2k-1}+ \frac{1}{3}\cdot \frac{1}{2m+2k+1} \Big)\\
 =& \frac{6L}{\pi^2}\sum_{k\in \mathbb{Z}}  \frac{1}{2k+1-2m} \Big(\frac{1}{2k+1-2m}- \frac{1}{3}\cdot \frac{1}{2m+2k+1} \Big).
\end{align}
We know that
\begin{equation}
\sum_{k\in \mathbb{Z}}  \left(\frac{1}{2k+1-2m}\right)^2=\sum_{k\in \mathbb{Z}}  \left(\frac{1}{2k+1}\right)^2= \frac{\pi^2}{4} =: \mathcal{C}_0>1.
\end{equation}

If $m=0$, then $l^{(1)}_n>0$. If $m\neq 0$, without loss of generality,  by assuming that $m>0$ we get
\begin{align}
&\Big| \frac{1}{3}\sum_k \frac{1}{2k+1-2m}\cdot \frac{1}{2m+2k+1}\Big|\\=&\Big|\frac{1}{3} \Big(\sum_{k\geqslant m}+ \sum_{k\leqslant -m-1}+ \sum_{k=-m}^{m-1} \Big)    \frac{1}{2k+1-2m}\cdot \frac{1}{2m+2k+1}\Big|\\
\leqslant& \Big|\frac{1}{3} \sum_{k\geqslant m}   \frac{1}{2k+1-2m}\cdot \frac{1}{2m+2k+1}\Big|+ \Big|\frac{1}{3}  \sum_{k\leqslant -m-1}  \frac{1}{2k+1-2m}\cdot \frac{1}{2m+2k+1}\Big| \\
& \;\;\;\;\;+ \Big|\frac{1}{3}  \sum_{k=-m}^{m-1}    \frac{1}{2k+1-2m}\cdot \frac{1}{2m+2k+1}\Big|\\
\leqslant &\frac{1}{6}\mathcal{C}_0+ \frac{1}{6}\mathcal{C}_0+  \frac{2m}{3(4m-1)}\\
< & \mathcal{C}_0-\frac{1}{3}.
\end{align}
Therefore,
\begin{equation}\label{ln-es}
  \big| l^{(1)}_n\big| > \frac{2L}{ \pi^2} \textrm{ if $n$ is even, $\big| l^{(1)}_n\big|=0$ if $n$ is odd.}
\end{equation}

\begin{rmrk}\label{rmk-diag-es}
The way that we get this uniform bound relies on the matrix $J_0$, more precisely, the diagonal matrix $\Lambda$. Indeed, this is the key point to get the controllability with control as $(1, 1)u(t)$: in \eqref{T-gamma} if we replace $\delta(x)J_0$ by $\delta(x)\Lambda$, then the operator becomes easier as the coupling term disappear; thanks to the expansion of $\delta(x)$, \eqref{es-delta-O2}, if we further replace $\delta(x)\Lambda$ by $-\frac{3}{4}\gamma \Lambda$, then the operator becomes
\begin{equation}
  T_{\gamma, s}:=\Lambda \left( \partial_x-\frac{3}{4}\gamma\right),
\end{equation}
\begin{equation}
    D( T_{\gamma, s}):= \left\{(w_1, w_2)\in (H^1)^2: w_1(0)= -w_2(0), w_1(L)= -w_2(L)\right\},
\end{equation}
for which the controllability is rather easy to obtain. Maybe with the help of some perturbation arguments we can also prove the controllability of the operator $\Lambda \partial_x+ \Lambda\delta(x)$ from this observation, of course we need to deal with some loss of derivative issues as the normal fixed point argument can not be applied here. However, in our case we need to consider $J_0= \Lambda+ J$, where $J$ is of the same order  as $\Lambda$, hence cannot be ignored.
\end{rmrk}
We  observe from \eqref{spq} and \eqref{varphi-1-exp}  that $\lVert \psi_{n}^{(1)} \lVert_{\infty}$ are $O(1)$ rather than $O((1/n)^{\alpha})$.
However, in \eqref{v-n-asy-sp}, we still have a $\gamma^{k}$ in front of $\lVert \psi_{n}^{(k)} \lVert_{\infty}$, which, when $\gamma$ is small, gives good estimates.

In fact, thanks to the above calculation, \eqref{T-gamma}, \eqref{es-vn-final}, we get the following:
\begin{align}
\delta(x)=& = -\frac{3}{4}\gamma + R_0(\gamma),\label{delta-es}\\
\label{varphi-1-es}
\psi_{n, 1}(\gamma)=&e^{i \pi n x/L}+ \gamma \psi_{n, 1}^{(1)}+ R_{n, 1}(\gamma),\\
\psi_{n, 2}(\gamma)=&-e^{-i \pi n x/L}+ \gamma  \psi_{n, 2}^{(1)}+ R_{ n, 2}(\gamma).\label{varphi-2-es}
\end{align}
Therefore, the existence of $r^{(2)}>0$ and $C_2>0$ such that
\begin{equation}\label{varphi-es}
    \lVert\psi_{n}^{(1)}\lVert_{\infty}\leq L,
\end{equation}
\begin{equation}\label{R-es}
 \lVert R_0(\gamma) \lVert_{\infty},\;\; \lVert R_{n, 1}(\gamma) \lVert_{\infty},\;\; \lVert R_{n, 2}(\gamma) \lVert_{\infty}\leqslant C_2|\gamma|^2.
\end{equation}
  Moreover, there exists $r^{(3)}\in (0, r^{(2)})$ such that, if $|\gamma|< r^{(3)}$, then we have
 \begin{equation}
     4/5<\lVert \psi_{n, 1}(\gamma)\lVert_{\infty},  \lVert \psi_{n, 2}(\gamma)\lVert_{\infty},
 \end{equation}
  \begin{equation}\label{es-L2}
     4/5<\lVert \psi_{n, 1}(\gamma)\lVert_{(L^{2})^2},  \lVert \psi_{n, 2}(\gamma)\lVert_{(L^{2})^2}<6/5.
 \end{equation}
 The same estimates hold for $\chi_{n}(\gamma)$, hence
 \begin{equation}\label{es-sca-pro}
    \big   \langle \psi_{n}(\gamma), \chi_{n}(\gamma) \big \rangle= 1+  \big \langle \psi_{n}(\gamma), \chi_{n}(\gamma)-\chi_{n}^{\textcolor{black}{(0)}} \big \rangle \in (3/4, 5/4).
 \end{equation}

 Furthermore, for $n\in \mathbb{Z}^*$ we observe from the definition of $\psi_{n}(\gamma)$ and \eqref{delta-es}--\eqref{R-es} that
 \begin{align}
     &\mu_{n}(\gamma) \int_0^L \psi_{n, 1}(\gamma)(x)+ \psi_{n, 2}(\gamma)(x) dx\\
     =& 2 \langle \mu_{n}(\gamma) \psi_{n}(\gamma), (1,1)\rangle\\
     =& 2\langle \Lambda \partial_x \psi_{n}(\gamma)+ \delta(x)J_0\psi_{n}(\gamma), (1,1)\rangle\\
     =& \big(\psi_{n, 1}(\gamma)(1)-\psi_{n, 1}(\gamma)(0)-\psi_{n, 2}(\gamma)(1)+\psi_{n, 2}(\gamma)(0)\big)+ \frac{2}{3}\int_0^L \delta(x)(\psi_{n, 1}-\psi_{n, 2})(x) dx \notag \\
     =&\Big(2((-1)^n-1)+ \gamma l^{(1)}_n+ \mathcal{O}(\gamma^2) \Big) -\frac{1}{2} \gamma \left(e^{i\pi nx/L}+ e^{-i\pi nx/L}  \right)+ \mathcal{O}(\gamma^2) \notag\\
     =&2((-1)^n-1)+ \gamma l^{(1)}_n+ \mathcal{O}(\gamma^2).\label{bound-es-for}
 \end{align}
 {\color{black}We remark here that $O(\gamma^2)$ means uniformly bounded by $C\gamma^2$  with some $C>0$ independent of   $n\in \mathbb{Z}^*$ and $\gamma$ small. The same convention is used for other similar notations, for example $\mathcal{O}(\gamma)$.}
By inserting \eqref{ln-es}  into \eqref{bound-es-for}, we get
\begin{equation}
 \mu_{n}(\gamma) \int_0^1 \psi_{n, 1}(\gamma)+ \psi_{n, 2}(\gamma)=\left\{
\begin{aligned}
0+ \gamma l^{(1)}_n+ R_n(\gamma),  \textrm{ when $n$ is even}, \\
-4+ 0+ R_n(\gamma), \textrm{ when $n$ is odd},
\end{aligned}
\right.
\end{equation}
with $|R_{n}(\gamma)|\leqslant C_3|\gamma|^2$.
Therefore,  there exists $r^{(4)}\in (0, r^{(3)})$ such that, if $|\gamma|< r^{(4)}$, then we have
\begin{equation}
    |\gamma| \frac{L}{ \pi^2} < |\mu_{n}(\gamma) \langle \psi_{n}(\gamma), (1,1)\rangle|< 5,\quad {\color{black} \forall n \in \mathbb{Z}^*}.
\end{equation}
We are able to perform the same calculation for $\chi_{n}(\gamma)$ which are eigenfunctions of $\mathcal{S}_0^*$, thus
\begin{equation}\label{psi-sca-pro}
    |\gamma| \frac{L}{ \pi^2} < |\bar \mu_{n}(\gamma) \langle \chi_{n}(\gamma), (1,1)\rangle|< 5, \quad {\color{black} \forall n \in \mathbb{Z}^*}.
\end{equation}
{\color{black}
\begin{rmrk}[$\chi_0$ and the ``missing direction'']
\label{rmkcontrol}
As for the case when $n=0$, thanks to the diffeomorphism \eqref{diffeo} and the equation \eqref{f-0-1-2-0}, we know that \begin{equation}
\begin{aligned}
    \psi_{0, 1}(x)+ \psi_{0, 2}(x)&=0, \quad\forall x \in [0, L],\\
    \partial_x \psi_{0,1} + \frac{2}{3}\delta \psi_{0,1}&=0,
\end{aligned} \end{equation}
thus $\psi_0$ can be computed explicitly,
\begin{equation}
    \psi_{0, 1}= -\psi_{0, 2}=  \textcolor{black}{\left(\sqrt{1+\frac{\gamma L}{2}}- \frac{\gamma}{2} \frac{L_{\gamma}}{L} x\right)^{-1}}=  \textcolor{black}{\left(\sqrt{1+\frac{\gamma L}{2}}- \frac{\sqrt{1+\frac{\gamma L}{2}}-\sqrt{1-\frac{\gamma L}{2}}}{L} x\right)^{-1}},
\end{equation}
which of course satisfies
\begin{equation}
\langle \psi_{0}(\gamma), (1,1)\rangle=0.
\end{equation}
The same reason and similar calculations for $\chi_0$ lead to
\begin{gather}
    \chi_{0, 1}(x)+ \chi_{0, 2}(x)=0,\\
    \partial_x \chi_{0, 1}(x)- \frac{4}{3}\delta(x) \chi_{0, 1}(x)= 0,
\end{gather}
thus
\begin{gather}
    \chi_{0, 1}(x)= \textcolor{black}{\left(\sqrt{1+\frac{\gamma L}{2}}- \frac{\sqrt{1+\frac{\gamma L}{2}}-\sqrt{1-\frac{\gamma L}{2}}}{L} x\right)^{2}}, \\
    \langle \chi_{0}(\gamma), (1,1)\rangle=0.\label{proj-chi-0}
\end{gather}
We observe from \eqref{proj-chi-0} that we can not cover direction $\psi_0$ with our control, due to the moment theory. {\color{black}More precisely, the quantity
\begin{equation}\label{uncontr-direc}
    \langle w, \chi_0\rangle= \int_0^L \textcolor{black}{\left(\sqrt{1+\frac{\gamma L}{2}}- \frac{\sqrt{1+\frac{\gamma L}{2}}-\sqrt{1-\frac{\gamma L}{2}}}{L} x\right)^{2}} (w_1(x)- w_2(x)) dx,
\end{equation}
which is precisely the term appearing in \eqref{missing-direc-w-con},
remains constant in time for any control $u(t)$. \textcolor{black}{Note that the term $\frac{L_{\gamma}}{L}$ in factor of the term in \eqref{missing-direc-w-con} is a constant and therefore does not change which quantity is conserved quantity.} Hence the control profile $(1,1)^T$ is compatible with the conservation of mass \eqref{missing-direc-w-con}. The equivalence between \eqref{missing-direc-w-con} and \eqref{uncontr-direc} is not a coincidence: from \eqref{missing-direc-w-con} we know that there is at least one direction for which we can not control; on the other hand, \eqref{proj-chi-0} gives us the only uncontrollable direction of the system. Algebraically, these two directions have to be the same.}
\end{rmrk}
}

Thanks to Theorem \ref{tho-rus} and  \eqref{es-L2}, for $|\gamma|< r_R$, we know that
\begin{equation}\label{var-rie-bas}
\textrm{  $(\psi_{n}(\gamma))_{n\in\mathbb{Z}^{*}}$
form a Riesz basis \textcolor{black}{of $\mathcal{L}^{2}_{(0)}$},}
\end{equation}
where
\begin{equation}
\mathcal{L}^{2}_{(0)}:=\{f\in L^{2}|\langle f, \chi_{0} \rangle=0 \}.
\end{equation}

Combining \eqref{moment-bas}, \eqref{dn-bn}, \eqref{mu_n asymp}, \eqref{es-sca-pro}, \eqref{psi-sca-pro}, and \eqref{var-rie-bas}, \textcolor{black}{this} complete the proof of Lemma \ref{lem-fond}.

\subsubsection{Controllability of \textcolor{black}{transformed system} \eqref{sys01}--\eqref{cond-0}}

Let us look now at the operator $\mathcal{A}$ given by \eqref{A}--\eqref{domaineA} and
associated to system \eqref{sys01}--\eqref{cond-0}.

As stated previously, the eigenvalues of the operators \eqref{T-gamma} and of \eqref{A} are the same; and the eigenfunctions of \eqref{A} can be generated by applying the diffeomorphism \eqref{diffeo} on the eigenfunctions of \eqref{T-gamma}. Indeed, with $(f_{n})_{\textcolor{black}{n\in\mathbb{Z}^{*}}}$ the family of eigenfunctions of $\mathcal{A}$ that form a Riesz basis of $(L^{2}_{(0)})^{2}$, and depends of course on $\gamma$, \textcolor{black}{and from $(ii)$ in Lemma \ref{lem-fond}, there exists uniformly bounded positive constants $(a_{n})_{n\in\mathbb{Z}^{*}}$ such that}
\begin{equation}\label{def-an}
\textcolor{black}{f_{n}(\gamma):= a_{n}}\exp \Big( \int_0^x \delta(y)dy\Big)   \psi_{n}(\gamma).
\end{equation}

\begin{rmrk}\label{rmk-fn-not-orth}
The constant $a_n$ is to be added to insure  that the generated eigenfunctions $f_{n}(\gamma)$  verify the normalized condition from perturbation, $\langle f_{n}(\gamma), \phi_{n}^{\textcolor{black}{(0)}}\rangle=1$, so that similar perturbation calculations can be performed. Moreover, they are not orthonormal basis as $L^2$ normalized condition is not satisfied.
\end{rmrk}
As \textcolor{black}{the change of coordinates} \eqref{diffeo} \textcolor{black}{used to remove the diagonal coefficients of the \textcolor{black}{coupling term}} is an isomorphism in $H^s$, Theorem \ref{tho-con-w} directly leads to the controllability of \textcolor{black}{this system and thus of \eqref{sys01}--\eqref{cond-0}}:
\begin{crllr}\label{cor-con-or}
 If $T\geqslant 2L, \gamma\in (0, \gamma_0)$, then system \eqref{sys01}\textcolor{black}{--\eqref{cond-0}} is $D(\mathcal{A}^s_{(0)})$ controllable with $D_p(\mathcal{H}_{(0)}^{s-1}(0, 2L))$  controls, \textcolor{black}{where
 \begin{equation}
    D(\mathcal{A}^s_{(0)})=\{f\in D(\mathcal{A}^s)| \langle f,f_{0}\rangle=0\}.
 \end{equation}
 }
\end{crllr}
The preceding corollary is just a direct consequence of Theorem \ref{tho-con-w} and isomorphism by change of variables.  Moreover,  actually,  by using the same technique we are able to prove the following estimate (compare to \eqref{psi-sca-pro})
\textcolor{black}{\begin{equation}\label{Igrowth0}
m_{0}|\mu_n^{-1}|\leq |\langle \mathcal{I},a_{n}^{-1} f_n \rangle | \leq M_{0}|\mu_n^{-1}|, \quad \forall n \in \Z^{*},
\end{equation}
for some constants $m_{0},M_{0}>0$, which implies, from the uniform boundedness of $(a_{n})$,}
\begin{equation}\label{Igrowth}
m|\mu_n^{-1}|\leq |\langle \mathcal{I}, f_n \rangle | \leq M|\mu_n^{-1}|, \quad \forall n \in \textcolor{black}{\Z^{*}},
\end{equation}
for some constants $m,M>0$.
We can be even more precise, using the fact that
$\mathcal{I}\in (H^1)^2$:
let us write, for $n\in \Z^\ast$,
\begin{equation}\label{es-cal-macal-I}
    \begin{aligned}
    -\mu_n \langle \mathcal{I}, f_n \rangle &= \langle \mathcal{I}, \mathcal{A} f_n \rangle \\
    &= \int_0^L \mathcal{I}_1 \partial_x f_{n,1} - \mathcal{I}_2 \partial_x f_{n,2} -  \langle \delta(x) J \mathcal{I}, f_n \rangle \\
    &= \left[\mathcal{I}_{1} f_{n,1}\right]_0^L -\left[\mathcal{I}_2 f_{n,2}\right]_0^L - \int_0^L \partial_x\mathcal{I}_{1}  f_{n,1} - \partial_x\mathcal{I}_2  f_{n,2} - \langle \delta(x) J \mathcal{I}, f_n \rangle \\
    &= 2\left(e^{\int_0^L \delta} f_{n,1}(L)-f_{n,1}(0)\right) - \langle \Lambda \partial_x \mathcal{I} + \delta(x) J \mathcal{I}, f_n \rangle\\
    &={\color{black}  2\left(e^{\int_0^L \delta} f_{n,1}(L)-f_{n,1}(0)\right) - \langle \delta(x)J_0 \mathcal{I}, f_n \rangle.}
    \end{aligned}
\end{equation}
Now, as $\Lambda \partial_x \mathcal{I} + \delta(x) J \mathcal{I} \in (L^2)^2$, this gives us an asymptotic expansion for the coefficients of $\mathcal{I}$:
\begin{equation}
    \label{I-asympt}
    \langle \mathcal{I}, f_n \rangle=\frac{2\left(f_{n,1}(0)-e^{\int_0^L \delta} f_{n,1}(L)\right)}{\mu_n} + \frac{1}{\mu_n}\langle \delta(x)J_0 \mathcal{I}, f_n \rangle, \quad \forall n \in \Z^\ast.
\end{equation}
{\color{black}Thanks to the expansion of $f_n$ and $\delta(x)$, we know that\textcolor{black}{, for any $n\in\mathbb{Z}^{*}$}
\begin{equation}
    f_{n,1}(0)-e^{\int_0^L \delta} f_{n,1}(L)= \left(1-(-1)^n \right)+ \gamma \left( \frac{3L}{4}(-1^n)+ f^{(1)}_{n, 1}(0)-f^{(1)}_{n, 1}(L) \right)+ \mathcal{O}(\gamma^2),
\end{equation}
which, combined with the fact that
\begin{equation}
    \langle \delta(x)J_0 \mathcal{I}, f_n \rangle= -\frac{3}{4}\gamma \langle J_0 \mathcal{I}, f_n \rangle+ O(\gamma^2)= -\frac{3}{4}\gamma \langle J_0 (1, 1), f_n^{(0)} \rangle+ O(\gamma^2)= \mathcal{O}(\gamma^2),
\end{equation}
leads to
\begin{equation}
     -\mu_n \langle \mathcal{I}, f_n \rangle= \left(1-(-1)^n \right)+ \gamma \left( \frac{3L}{4}(-1^n)+ f^{(1)}_{n, 1}(0)-f^{(1)}_{n, 1}(L) \right)+ \mathcal{O}(\gamma^2).
\end{equation}
Similar calculations as in Subsection \ref{sec-2.3}, lead to the existence of $\widetilde{\gamma}_0\in (0, \gamma_0), m, M> 0$ (where $\gamma_0$ is defined in Lemma \ref{lem-fond}) such that
\begin{equation}\label{Iboundary}
m\leq |\mu_n \langle \mathcal{I}, f_n \rangle|\leq M, \quad \forall n \in \Z^*,\quad  \forall \gamma\in (0, \widetilde{\gamma}_0).
\end{equation}
}
From $\|f_n\|_{(L^2)^2}= 1+ O(\gamma)$ and Remark \ref{rmk-fn-not-orth}, we know that \textcolor{black}{$(f_n)$ is not an} orthonormal basis.  From now on, we replace  $f_n$ by \textcolor{black}{the} orthonormal basis $(f_n/\|f_n\|_{(L^2)^2})$ which still satisfies the preceding inequality, and which is still denoted by $(f_n)$ to simplify the notations.

\textcolor{black}{We now study the system with our virtual control $\mathcal{I}_{\nu}$:}
thanks to the fact that
\begin{equation}
    \langle f_0, f_n \rangle= \delta_{0, n},
\end{equation}
we obtain the following lemma that leads to the controllability of the system with \textcolor{black}{control $\mathcal{I}_{\nu}$}.
\begin{lmm}\label{lem-3-5}
Let $0< \lvert\gamma\lvert< \widetilde{\gamma}_0$ .  Let $\nu\neq 0$. There exist $m, M> 0$ such that
\begin{gather}
m\leq | \langle \mathcal{I}_{\nu}, f_0 \rangle|= |\nu|\leq M,\\
m\leq |\mu_n \langle \mathcal{I}_{\nu}, f_n \rangle|\leq M, \quad \forall n \in \Z^*.
\end{gather}
\end{lmm}
And thus, as all the moments are uniformly bounded away from 0, we recover the controllability of the system in $D(\mathcal{A})$ (which is a Sobolev space with certain boundary conditions). This explains the definition of $\mathcal{I}_{\nu}$ given by \eqref{def-I-nu}.

{\color{black}\begin{crllr}
 For $T\in [2L,+\infty)$, for $ \gamma\in (0, \gamma_0)$,  system \eqref{sys1virt-con} is $D(\mathcal{H}^s)$ controllable with controls having $D_p(\mathcal{H}^{s-1}(0, 2L))$ regularity, where
 \begin{gather}
  D_p(\mathcal{H}^s(0, 2L))=\{f:= \sum_{n\in \mathbb{Z}}f_n \bar{p}_n \in H^{s} | \sum_{n\in \mathbb{Z}} (1+ n^{2s}) f_n^2< +\infty\}, \\
  \textrm{$\{p_n\}_n$ is the dual basis of $\{e^{\mu_n(s-2L)}\}_n$ presented in Theorem \ref{tho-rus}},\\
  \textrm{the control vanishes on $(2L, T)$.}
 \end{gather}
\end{crllr}}

\subsubsection{Asymptotic calculation: {\color{black} can we achieve better estimates than  $\mathcal{O}(1)$   for the transformed operator $\mathcal{A}$?}}
We have seen from the above calculation that $\gamma \psi_{n}^{(1)}$ and $R_{1, n}(\gamma)$ \textcolor{black}{are} $\mathcal{O}(\gamma)$ and $\mathcal{O}(\gamma^2)$ \textcolor{black}{respectively}. Thanks to \eqref{ln-es}, at least for $L^{\infty}$ norm the $O(1)$ type estimates are sharp: there is no decay with respect to $n$. As we have indicated in Remark \ref{rmk-diag-es}, the diagonal matrix plays an important role in the estimation.  Hence it is natural to ask \textcolor{black}{whether} the normalized eigenfunctions of $\textcolor{black}{\mathcal{A}}$ are better than those of $\mathcal{S}_0$, as the diagonal coefficients of $J$ are 0. The normalized eigenfunctions are given by
\begin{equation}
    \textcolor{black}{\mathcal{A}}\bar \psi_{n}(\gamma)= \mu_{n}(\gamma) \bar \psi_{n}(\gamma), \; \langle \bar \psi_{n}(\gamma), \bar \chi_{n}^{\textcolor{black}{(0)}}\rangle=1.
\end{equation}

In this situation, everything we have defined for the calculation of eigenfunctions from \eqref{T-gamma} to \eqref{v-n-asy-sp} remain the same if we replace $J_0$ by $J$, except  for the part of controllability: as $(1,1)$ is replaced by $\exp\Big(\int_0^x \delta(y) dy\Big)(1,1)$.

Therefore,  from \eqref{varphi-1-exp} and using \eqref{defJ} we know that
\begin{equation}\label{varphi-1-exp-new}
    \bar \psi^{(1)}_n=\frac{3L}{4} \sum_{k\neq n} \frac{\langle J\psi_{n}^{\textcolor{black}{(0)}}, \psi_k^{\textcolor{black}{(0)}}  \rangle}{i \pi (k-n)} \psi_k^{\textcolor{black}{(0)}}= -\frac{\textcolor{black}{L}}{2\pi^2}\sum_{k\neq n} \frac{\big((-1)^{k-n}-1\big)}{n^2-k^2} \psi_k^{\textcolor{black}{(0)}}.
\end{equation}
{\color{black}It is easy to find that
\begin{equation}
    \lVert \bar \psi^{(1)}_n \lVert_{(L^2)^2}\xrightarrow[\lvert n\lvert\rightarrow +\infty]{} 0.
\end{equation}
Moreover, with the help of more precise estimates we are able to prove that
\begin{equation}
    \lVert \bar \psi^{(1)}_n \lVert_{\infty}\xrightarrow[\lvert n\lvert\rightarrow +\infty]{} 0.
\end{equation}
The interpolation implies that the same result holds for the $(L^p)^2$ norm. \textcolor{black}{It may be that for the transformed operator we would able to conclude to} some better asymptotic estimates rather than $O(1)$. On the other way around, as we know that the controllability is strongly related to those estimates, see \eqref{moment-bas}. The system is controllable only if $b_n$ is away from 0 in some sense. If for the transformed operator   $\Lambda \partial_x+  \delta(x) J$ we are able to get some better decay estimates, it would be interesting to use them to  investigate the controllability of the system
\begin{equation}
\begin{aligned}
&\partial_{t}w+\Lambda\partial_{x}w+\delta J w=u_{\gamma}(t)\begin{pmatrix}1\\1\end{pmatrix},
\label{w-con00}\\
& w_{1}(t,0)=-w_{2}(t,0),\\
& w_{2}(t,L)=-w_{1}(t,L),
\end{aligned}
\end{equation}
as this system is more complicated in some sense.  More precisely, as it is stated in  Remark \ref{rmk-diag-es} we used the diagonal term of $J_0$ to simplify the calculation, while in this case the diagonal term disappears and we may need to go for third order of $\gamma$ or beyond.
}

\subsection{Controllability of the target system}\label{sec-con-diss}
This time we directly consider the system with simplified coupling term.
As previously, it suffices to investigate the properties of the operator
\begin{equation}\label{til-T-gamma}
\widetilde{T}_{\gamma}:=  \widetilde{\mathcal{S}}_0=\Lambda \partial_x+ \delta(x){\color{black}J},
\end{equation}
\begin{equation}
    D(\widetilde{\mathcal{S}}_0):= \left\{(w_1, w_2)\in (H^1)^2: w_1(0)= -e^{2\mu L}w_2(0), \ w_1(L)= -w_2(L)\right\},
\end{equation}
{\color{black}and the controllability of
\begin{equation}\label{til-w-con}
\left\{\begin{aligned}
\partial_{t}w+\Lambda\partial_{x}w+\delta J w&=u_{\gamma}(t)\mathcal{I}_{\nu},
\\
 w_{1}(t,0)=-e^{2\mu L}w_{2}(t,0),& \ w_{2}(t,L)=-w_{1}(t,L).
\end{aligned}\right.
\end{equation} }
As in Section \ref{sec-con-or} we define  adjoint operators $\widetilde{\mathcal{S}}_0^*$, eigenvalues $\widetilde{\mu}_n$, and  normalized eigenfunctions $(\textcolor{black}{\widetilde{f}_{n}}, \textcolor{black}{\widetilde{\phi}_{n}})$.
As indicated at the beginning of Section \ref{s3} we are able to perform the same calculations for operator $\widetilde{\mathcal{S}}_0$ as in Section \ref{sec-def-eig}--\ref{sec-2.2}, we omit the explicit calculation  for the reader's convenience. More precisely, we have
\begin{lmm}\label{tild-lem-fond}
Let $\mu\neq0$. There exist $\bar r_{\mu}, c_{\mu}, C_{\mu}>0$ such that  for any $ \gamma\in B_{\bar r_{\mu}}$ and any $ n\in \mathbb{Z}$ we have
\begin{gather}\label{tilde-es-vn-final}
|\widetilde{\textcolor{black}{\mu}}_n-\widetilde{\mu}_n^{\textcolor{black}{(0)}}|< 1/(4L),\\
\{\widetilde{\phi}_{n}/\lVert \widetilde{\phi}_{n}\lVert_{(L^{2})^2}\}_n \textrm{ is a Riesz basis  \textcolor{black}{of $(L^{2})^2$}},\\
\{\widetilde{f}_{n}/\lVert \widetilde{f}_{n}\lVert_{(L^{2})^2}\}_n \textrm{ is a Riesz basis  \textcolor{black}{of $(L^{2})^2$}},\label{tild-varphi-es} \\
 \lVert \textcolor{black}{\widetilde{\phi}_{n}}-\widetilde{\phi}_{n}^{\textcolor{black}{(0)}} \lVert_{\infty},
      \lVert \textcolor{black}{\widetilde{f}_{n}}-\widetilde{f}_{n}^{\textcolor{black}{(0)}} \lVert_{\infty} \leqslant C_{\mu} \lvert \gamma \lvert,\\
     \lVert \textcolor{black}{\widetilde{\phi}_{n}}-\widetilde{\phi}_{n}^{\textcolor{black}{(0)}}- \gamma \widetilde{\phi}^{(1)}_n \lVert_{\infty} \leqslant C_{\mu} \lvert \gamma \lvert^2,\\
      \lVert \textcolor{black}{\widetilde{f}_{n}}-\widetilde{f}_{n}^{\textcolor{black}{(0)}}- \gamma \widetilde{f}^{(1)}_n \lVert_{\infty} \leqslant C_{\mu} \lvert \gamma \lvert^2,\\
       \lVert \widetilde{\phi}_{n}^{(1)}\lVert_{\infty},   \lVert \widetilde{f}_{n}^{(1)}\lVert_{\infty}\leq C_{\mu},\\
  c_{\mu} \leq  \lVert \widetilde{\phi}_{n}\lVert_{(L^{2})^2},   \lVert \widetilde{f}_{n}\lVert_{(L^{2})^2},    \lVert \widetilde{\phi}_{n}\lVert_{\infty},   \lVert \widetilde{f}_{n}\lVert_{\infty}\leq C_{\mu}.
\end{gather}
\end{lmm}
\begin{rmrk}
Due to the fact that $\lVert \widetilde{f}_n^{(0)}\lVert_{\infty}$, \eqref{eqref-f-L2} and the Riesz basis \eqref{eqref-f-n-riesz} strongly depend on $\mu$, the norm of $\widetilde{S}_n$  depends on the value of $\mu$ (see \eqref{es-sn-norm} for instance). This is one of the key points which result in the existence of $c_{\mu}, C_{\mu}$ and (especially) $\bar r_{\mu}$.
\end{rmrk}
Thanks to \eqref{40} and Lemma \ref{tild-lem-fond}, for $\forall n \in \mathbb{Z}$, we have
\begin{align}
 \widetilde{\textcolor{black}{\mu}}_n \langle \textcolor{black}{\widetilde{\phi}_{n}}, (1,1)\rangle
  = & \big(\widetilde{\phi}_{n, 1}(\gamma)(L)-\widetilde{\phi}_{n, 1}(\gamma)(0)-\widetilde{\phi}_{n, 2}(\gamma)(L)+\widetilde{\psi}_{n, 2}(\gamma)(0)\big)+ \mathcal{O}(\gamma)\\
  = & \big(\widetilde{\phi}_{n, 1}(L)-\widetilde{\phi}_{n, 1}(0)-\widetilde{\phi}_{n, 2}(L)+\widetilde{\phi}_{n, 2}(0)\big)+ \mathcal{O}(\gamma)\\
  = &2(-1)^n e^{-\mu L} - 1 - e^{-2\mu L}+\mathcal{O}(\gamma),\label{tild-es-bd-con}
\end{align}
which yields the controllability of the system with control term $u_{\gamma}(t)(1, 1)$.

 Moreover, for System \eqref{til-w-con} with control $u_{\gamma}(t) \mathcal{I}_{\nu}$   similar calculations as \eqref{es-cal-macal-I} lead to
\begin{align}
  \widetilde{\mu}_{n} \langle\widetilde{\phi}_{n}, \mathcal{I}\rangle
  = &[\widetilde{\phi}_{n, 1} \mathcal{I}_1]_0^L- [\widetilde{\phi}_{n, 2} \mathcal{I}_2]_0^L+ \mathcal{O}(\gamma)\\
  = &[\widetilde{\phi}_{n, 1} ]_0^L- [\widetilde{\phi}_{n, 2} ]_0^L+ \mathcal{O}(\gamma)\\
  = & \big(\widetilde{\phi}_{n, 1}(L)-\widetilde{\phi}_{n, 1}(0)-\widetilde{\phi}_{n, 2}(L)+\widetilde{\phi}_{n, 2}(0)\big)+ \mathcal{O}(\gamma)\\
  = &2(-1)^n e^{-\mu L} - 1 - e^{-2\mu L}+\mathcal{O}(\gamma),\label{tild-es-bd-con-I}
\end{align}
{\color{black}where we used the fact that $\mathcal{I}=(1, 1)+ \mathcal{O}(\gamma)$ and $\widetilde{\phi}_n=\widetilde{\phi}_n^{(0)}+\mathcal{O}(\gamma)$.}
Then, thanks to similar calculations on zeroth order of $\gamma$, we further get
\begin{align}
  \widetilde{\mu}_{n} \big\langle\widetilde{\phi}_{n}, f_0\big\rangle
 = & \big\langle \Lambda \partial_x \widetilde{\phi}_{n}+ \delta(x) J \widetilde{\phi}_{n}, f_0\big\rangle \\
 = & \big\langle \Lambda \partial_x \widetilde{\phi}_{n}, f_0\big\rangle + \mathcal{O}(\gamma)\\
 = & [\widetilde{\phi}_{n, 1}f_{0, 1}- \widetilde{\phi}_{n, 2}f_{0, 2}]_0^L+ \mathcal{O}(\gamma)\\
  = &\big(e^{-2\mu L}-1 \big) + \mathcal{O}(\gamma).\label{tild-es-bd-con-II}
  \end{align}
Therefore, combining \eqref{tild-es-bd-con-I} and \eqref{tild-es-bd-con-II} we get
\begin{equation}\label{tild-es-bd-con-III}
   \widetilde{\mu}_{n} \big\langle\widetilde{\phi}_{n}, \mathcal{I}_{\nu}\big\rangle=  \widetilde{\mu}_{n} \big\langle\widetilde{\phi}_{n}, \mathcal{I}+ \nu f_0\big\rangle= 2(-1)^n e^{-\mu L} - 1 - e^{-2\mu L}+ \nu \big(e^{-2\mu L}-1 \big)+ \mathcal{O}(\gamma).
\end{equation}
\begin{rmrk}
Similar estimates hold for $\textcolor{black}{\widetilde{f}_{n}}$.
\end{rmrk}
\begin{rmrk}
One can compare estimate \eqref{tild-es-bd-con-III} to \eqref{bound-es-for} where we calculated up to order 2, while in \eqref{tild-es-bd-con} only terms of order 1 are needed. That is because the zeroth order term in \eqref{bound-es-for} becomes 0 when $n$ is even, hence we need to  use the first order terms as dominating term in this case. However, as we have seen in \eqref{40} the zeroth order of \eqref{tild-es-bd-con} is already away from 0, even for $n=0$, thus we no longer need to estimate the second order.
\end{rmrk}
Hence we get the following lemma:
\begin{lmm}\label{til-lem-fond}
Let $\mu> 3/L$. {\color{black} Let $0<\lvert\nu\lvert <1$}. There exists $\widetilde{\gamma}_{\mu}$ such that for any $\gamma\in (-\widetilde{\gamma}_{\mu}, \widetilde{\gamma}_{\mu})$ and for any $n\in \mathbb{Z}$ we have
\begin{itemize}
    \item[(i)] $\{\textcolor{black}{\widetilde{\phi}_{n}}\}_n$ $ (\textrm{resp. $\{\textcolor{black}{\widetilde{f}_{n}}\}_n)$ is a Riesz basis  \textcolor{black}{of $L^{2}$}}$;
    \item[(ii)] $\lvert\langle \textcolor{black}{\widetilde{f}_{n}}, \textcolor{black}{\widetilde{\phi}_{n}} \rangle\lvert\in (1/2, 2)$;
    \item[(iii)]$|\widetilde{\textcolor{black}{\mu}}_n-\widetilde{\mu}_n^{\textcolor{black}{(0)}}|< 1/(4L)$;
     \item[(iv)] $
   0< 1/2 < |\widetilde{\textcolor{black}{\mu}}_n \langle \textcolor{black}{\widetilde{f}_{n}}, \mathcal{I}_{\nu}\rangle|, |\widetilde{\textcolor{black}{\mu}}_n \langle \textcolor{black}{\widetilde{\phi}_{n}}, \mathcal{I}_{\nu}\rangle|< 3/2.$
\end{itemize}
\end{lmm}
We note here that in this case $\gamma$ is allowed to be $0$.
Lemma \ref{til-lem-fond} and the classical moment theory lead to the controllability of System \eqref{target-control}.
\begin{thrm}\label{tmm-diss-con}
Let $\mu>3/L, T\geqslant 2 L$. If $\gamma\in (-\widetilde{\gamma}_{\mu}, \widetilde{\gamma}_{\mu})$,  then system \eqref{target-control} is $ D(\widetilde{\mathcal{A}}^s)$ controllable with $ D_p(\widetilde{\mathcal{A}}^{s-1}(0, 2L))$ controls.
\end{thrm}

\section{\textcolor{black}{Construction of candidates for the backstepping transformation}}\label{s4}
From now on, {\color{black}for ease of notations},  let $\mu>0$ be the desired decay rate, and $\gamma, \nu >0$ {\color{black} satisfying conditions indicated in Section \ref{subsubsection-expstab-target} and Section \ref{s3} such that the target system has desired stability, and that both the original systems \eqref{sys1virt} and the target system \eqref{target-control} are  controllable. }\textcolor{black}{In order to stabilize the original system exponentially, we will now follow the backstepping strategy, that is, we will try to find a feedback such that the resulting closed-loop system can be mapped invertibly to our target system \eqref{target}. We will proceed in several steps. The first step is the following: for a given feedback law $F$, we solve an infinite-dimensional equivalent of the first equation of $\eqref{FiniteDOpEq}$ to build a transformation $T$. Using characterizations of Riesz bases in Section \ref{Riesz basis property}, this allows us to give conditions on $F$ for $T$ to be invertible. Then in  Section \ref{TB-from-OpEq} we check if there are some feedbacks $F$ such that the above construction actually yields a solution to the real backstepping equations, by considering an infinite-dimensional equivalent (equation \eqref{OperatorEquality}) of \eqref{backstepping-eq}. The search for such a feedback adds enough constraints on the feedback law to determine the expression of an exponentially stabilizing feedback law. Additionally, our work on equation \eqref{OperatorEquality} will ensure that the transformation really maps trajectories of the closed-loop system \eqref{sys1virt} (if they exist) to trajectories of the target system \eqref{target}.  }

\subsection{Kernel equations}\label{Kernel equations}
\textcolor{black}{Suppose there exists a solution $\mathbf{Z}$ to \eqref{sys1virt},
and a solution $z$ to \eqref{target}.
We are looking} for an invertible transformation that maps $ \mathbf{Z}$ to $ z$, under the form of a Fredholm transformation:
\begin{equation}
\label{transfo}
\begin{pmatrix}
z_{1}(t,x)\\z_{2}(t,x)
 \end{pmatrix}=T( \mathbf{Z}(t,\cdot))(x):=\int_{0}^{L}K(x,y).\begin{pmatrix}
\mathbf{Z}_{1}(t,y)\\\mathbf{Z}_{2}(t,y)
 \end{pmatrix}dy
\end{equation}
where $K(x,y)$ is a $2\times2$ matrix.
Following Section \ref{FiniteD}, we would like the kernel to satisfy the following equation, which is analogous to the first equation of \eqref{FiniteDOpEq}:
\begin{equation}\label{kernelEquation}
\Lambda \partial_{x} K+\partial_{y}K\Lambda+\delta(x)JK-K\delta(y)J=-\exp\left(\int_{0}^{x}\delta(y) dy\right)\begin{pmatrix}F_{1}&F_{2}\\F_{1}& F_{2} \end{pmatrix} (y)
\end{equation}
where $F_{1},F_{2}$ are the coefficients of the feedback,
together with
the following boundary conditions:
\begin{equation}\label{KernBoundx}
\begin{aligned}
& k_{11}(0,y)=-e^{-2\mu L}k_{21}(0,y),\\
& k_{12}(0,y)=-e^{-2\mu L}k_{22}(0,y),\\
& k_{11}(L,y)=-k_{21}(L,y),\\
& k_{12}(L,y)=-k_{22}(L,y),
\end{aligned}
\end{equation}

and
\begin{equation}\label{KernBoundy}
\begin{aligned}
&k_{11}(x,L)=-k_{12}(x,L),\\
&k_{21}(x,L)=-k_{22}(x,L),\\
&k_{11}(x,0)=-e^{2\mu L}k_{12}(x,0),\\
&k_{21}(x,0)=-e^{2\mu L}k_{22}(x,0).
\end{aligned}
\end{equation}

To study solutions to \eqref{kernelEquation}, \eqref{KernBoundx}, and \eqref{KernBoundy}, let us derive equations for the family of functions
\begin{equation*}
g_n:=(Tf_n)_{n\in \Z}.
\end{equation*}

This allows us to express \eqref{kernelEquation} as functions of $(g_n)$. Indeed, \textcolor{black}{integrating against the $f_{n}$} and using the boundary conditions \eqref{KernBoundx} and \eqref{KernBoundy} \textcolor{black}{we get that, for any $n\in\mathbb{Z}$,}
$$\begin{array}{rcl}
   -\int_0^L\exp\left(\int_{0}^{x}\delta(y) dy\right)\begin{pmatrix}F_{1}&F_{2}\\F_{1}& F_{2}\end{pmatrix}(y)f_n(y)dy & = &  \int_0^L \left(\Lambda \partial_{x} K(x,y)f_n(y)+\partial_{y}K(x,y)\Lambda f_n(y) \right.\\
   & &\left.+\delta(x)JK(x,y)f_n(y)-K(x,y)\delta(y)J f_n(y)\right)dy \\
   &=& \Lambda \partial_{x} g_n  + \delta(x) J g_n - T\mathcal{A}f_n \\
   &=& \Lambda \partial_{x} g_n  - \mu_n g_n + \delta(x) J g_n,
\end{array}$$
hence the following {\color{black}equations} for $(g_n)$:
\begin{equation}\label{ProjKernelEquation}
\left\{\begin{array}{l}\Lambda \partial_{x} g_n - \mu_n g_n + \delta(x) J g_n = - \textcolor{black}{\langle f_n,F\rangle} \mathcal{I}_\nu, \\
g_{n,1}(0)+e^{-2\mu L} g_{n,2}(0)=0, \quad g_{n,1}(L)+g_{n,2}(L)=0. \end{array}\right.
\end{equation}

Now, we study solutions to \eqref{ProjKernelEquation} for any given $F$.

By property of biorthogonal families, we have the following decomposition:
\begin{equation}\label{gndecomp}
g_n=\sum_{p \in \Z} \langle g_n, \widetilde{\phi}_p \rangle \widetilde{f}_p \ \textrm{in} \ L^2, \quad \forall n \in \Z.
\end{equation}
Taking the scalar product of equation \eqref{ProjKernelEquation} with the $\widetilde{\phi}_p$, we get
\begin{equation}\label{coeffequation}
    \left\langle \widetilde{\mathcal{A}}g_n , \widetilde{\phi}_p \right\rangle -\mu_n \langle g_n, \widetilde{\phi}_p\rangle = - \textcolor{black}{\langle f_n,F\rangle} \left\langle\mathcal{I}_\nu, \widetilde{\phi}_p \right\rangle, \quad \forall n \in \Z, \ \forall p \in \Z.
\end{equation}
As $\mathcal{A}$ is anti-hermitian, {\color{black} the equation} \eqref{coeffequation} becomes
\begin{equation*}
    \left\langle g_n , \widetilde{\mathcal{A}}^\ast \widetilde{\phi}_p \right\rangle -\mu_n \langle g_n, \widetilde{\phi}_p\rangle = - \textcolor{black}{\langle f_n,F\rangle} \left\langle\mathcal{I}_\nu, \widetilde{\phi}_p \right\rangle, \quad \forall n \in \Z, \ \forall p \in \Z,
\end{equation*}
hence finally
\begin{equation}\label{coeffexpression}
    \langle g_n, \widetilde{\phi}_p\rangle = - \frac{\textcolor{black}{\langle f_n,F\rangle} \left\langle\mathcal{I}_\nu, \widetilde{\phi}_p \right\rangle}{\widetilde{\mu}_p - \mu_n }, \quad \forall n \in \Z, \ \forall p \in \Z.
\end{equation}

\subsection{Riesz basis property}\label{Riesz basis property}

Now we study the invertibility of \textcolor{black}{this operator $T$ thus defined.}
The first obvious observation \textcolor{black}{from \eqref{coeffexpression}} is that $T$ is injective if and only if
\begin{equation}\label{NonZeroFeedbackCoeff}
 \langle f_n, F\rangle\neq 0, \quad \forall n \in \Z.
 \end{equation}
 Next, considering {\color{black} the formula} \eqref{coeffexpression}, we turn to the family of functions given by
\begin{equation}\label{def-kn}
    k_n:=\sum_{p\in \Z} \frac1{\widetilde{\mu}_p - \mu_n}\widetilde{f}_p, \quad \forall n \in \Z.
\end{equation}

Then we have the following property for the $(k_n)$:
\begin{lmm}
$(k_n)$ is a Riesz basis of $L^2(0,L)^2$.
\end{lmm}
\begin{proof}
First, let us perform the computations of \eqref{coeffequation} with the $(f_n)$:
\begin{equation}\label{fncoeffequation}
\left\langle\mathcal{A}f_n , \widetilde{\phi}_p \right\rangle -\mu_n \langle f_n, \widetilde{\phi}_p\rangle = 0, \quad \forall n \in \Z.
\end{equation}
This time, due to the dissipative boundary conditions in $0$ for the $\widetilde{\phi}_p$, boundary terms appear in the integration by parts (which amounts to taking the adjoint $\widetilde{\mathcal{A}}^\ast$):
\begin{equation*}
    \begin{array}{rcl}
    \mu_n \langle f_n, \widetilde{\phi}_p\rangle &=& \left\langle\mathcal{A}f_n , \widetilde{\phi}_p \right\rangle \\
    &=& {\color{black}\left\langle f_n, \mathcal{A}^\ast \widetilde{\phi}_p\right\rangle + f_{n,2}(0)\overline{\widetilde{\phi}_{p,2}(0)}-f_{n,1}(0)\overline{\widetilde{\phi}_{p,1}(0)} } \\
    &=&\left\langle f_n, \widetilde{\mathcal{A}}^\ast \widetilde{\phi}_p\right\rangle + f_{n,2}(0)\overline{\widetilde{\phi}_{p,2}(0)}-f_{n,1}(0)\overline{\widetilde{\phi}_{p,1}(0)}  \\
    &=& \widetilde{\mu_p} \langle f_n, \widetilde{\phi}_p \rangle - f_{n,1}(0)\overline{\widetilde{\phi}_{p,1}(0)}\left( 1-e^{-2\mu L} \right),
    \end{array}
\end{equation*}
where the last equality is obtained using boundary conditions \textcolor{black}{given by \eqref{cond-0} and \eqref{target}}. From this we get the following decomposition for $f_n$:
\begin{equation}\label{fn in tildefp}
f_n=\sum_{p\in \Z} \frac{f_{n,1}(0)\overline{\widetilde{\phi}_{p,1}(0)}\left( 1-e^{-2\mu L} \right)}{\widetilde{\mu_p}-\mu_n}\widetilde{f}_p=f_{n,1}(0)\widetilde{\tau}_\mathcal{A} k_n, \quad \forall n \in \Z.
\end{equation}
where $\widetilde{\tau}_\mathcal{A}$ is the operator defined by
\begin{equation}\label{tautilde}
\widetilde{\tau}_\mathcal{A} \widetilde{f}_p:= \overline{\widetilde{\phi}_{p,1}(0)}(1-e^{-2\mu L}) \widetilde{f}_p, \quad \forall p \in \Z,
\end{equation}
{\color{black}which also explains the definition of $k_n$ in \eqref{def-kn}.}

Let us now give a crucial lemma.
\begin{lmm}\label{LmmBoundaryBehaviour}
\textcolor{black}{There exists $\gamma^{*}>0$ such that, for any $\gamma\in(0,\gamma^{*})$, t}here exist $m,M>0$ such that
\begin{equation}\label{boundaryphi}
\begin{aligned}
m \leq \left|\widetilde{\phi}_{p,1}(0)\right|\leq M, \quad \forall p \in \Z,\\
m \leq \left|f_{p,1}(0)\right|\leq M, \quad \forall p \in \Z.
\end{aligned}
\end{equation}
\end{lmm}
\begin{proof}
When $\gamma=0$, $f_{p,1}(0)=1$. Referring to asymptotic estimates given in Section \ref{s3}, more precisely \eqref{def-an} and estimate \eqref{varphi-0}, there exists
$\gamma^{*}>0$ such that this remains true for any $\gamma\in(0,\gamma^{*})$. Similarly, one can show looking at \eqref{target}--\eqref{Atilde} that $\widetilde{\phi}_{p,1}(0)\neq0$ and is uniformly away from 0 (independent of $p$), thus, using estimate \eqref{tild-varphi-es}, the result holds.
\end{proof}
We can assume that we have chosen $\gamma<\gamma^\ast$ without losing the controllability of \eqref{sys1virt} and \eqref{target-control}. Then it is clear thanks to Lemma \ref{LmmBoundaryBehaviour} that
$\widetilde{\tau}_\mathcal{A}$ is an isomorphism of $(L^2)^2$. Moreover, thanks to the same lemma \textcolor{black}{and the definition of a Riesz basis},
the family of functions defined by
\begin{equation*}
    \left(\frac{f_n}{f_{n,1}(0)}  \right)
\end{equation*}
forms a Riesz basis of $(L^2)^2$, so that the family
\begin{equation*}
    (k_n)=\left( \widetilde{\tau}_\mathcal{A}^{-1} \frac{f_n}{f_{n,1}(0)} \right)
\end{equation*}
also forms a Riesz basis of  $(L^2)^2$.
\end{proof}

We now recall a result about Riesz basis (see for instance \cite[Chapter 4]{Christensen})
\begin{prpstn}\label{riesz1}
A family of vector $(f_{k})_{k\in\mathbb{Z}}$ of a Hilbert space $H$ is a Riesz basis if and only if it is complete
and there exists positive constants $c$ and $C$ such that, for any scalar sequence $(a_{k})$ with finite support,
\begin{equation}
c\sum |a_{k}|^{2}\leq \lVert\sum a_{k}f_{k}\rVert_{H}^{2}\leq C\sum |a_{k}|^{2}.
\end{equation}
\end{prpstn}

We can now prove the following Riesz basis property for the $(g_n)$:
\begin{prpstn}\label{Riesz-prop}
The family
\begin{equation*}\left(\frac1{\textcolor{black}{\langle f_n,F\rangle}}g_n\right)_{n\in \Z}\end{equation*} is a Riesz basis of $D(\widetilde{\mathcal{A}})$.
\end{prpstn}
\begin{proof}
We first prove that the above family is a Riesz sequence: let $I$ be a finite subset of $\Z$, and $(a_n)\in \CC^I$. Then
\begin{equation}
    \begin{aligned}
        \left\|\sum_{n \in I} \frac{a_n}{\textcolor{black}{\langle f_n,F\rangle}} g_n\right\|_{D(\widetilde{\mathcal{A}})}^2 &=  \left\|\sum_{n \in I} \frac{a_n}{\textcolor{black}{\langle f_n,F\rangle}} \sum_{p\in \Z}  \frac{\textcolor{black}{\langle f_n,F\rangle} \left\langle\mathcal{I}_\nu, \widetilde{\phi}_p \right\rangle}{\widetilde{\mu}_p - \mu_n } \widetilde{f}_p\right\|_{D(\widetilde{\mathcal{A}})}^2 \\
        &= \left\|\sum_{p\in \Z}  \left\langle\mathcal{I}_\nu, \widetilde{\phi}_p \right\rangle \left(\sum_{n \in I} \frac{a_n}{\widetilde{\mu}_p - \mu_n }\right) \widetilde{f}_p\right\|_{D(\widetilde{\mathcal{A}})}^2 \\
        &= \left\|\sum_{p\in \Z}  \left\langle\mathcal{I}_\nu, \widetilde{\phi}_p \right\rangle \left(\sum_{n \in I} \frac{a_n}{\widetilde{\mu}_p - \mu_n }\right) \widetilde{f}_p\right\|^2\\
        &+ \left\|\sum_{p\in \Z}  \left\langle\mathcal{I}_\nu, \widetilde{\phi}_p \right\rangle \left(\sum_{n \in I} \frac{a_n}{\widetilde{\mu}_p - \mu_n }\right)\widetilde{\mu}_p \widetilde{f}_p\right\|^2,
    \end{aligned}
\end{equation}
as $(\widetilde{f}_p)$ is {\color{black}a Riesz basis for $(L^2)^2$}, and from Proposition \ref{riesz1}, there exist $C_1, C_2>0$ such that
\begin{equation}\label{Rieszineqforfptilde}
\begin{aligned}
    C_1\sum_{p\in \Z}  (1+|\widetilde{\mu}_p|^2)& \left|\left\langle\mathcal{I}_\nu, \widetilde{\phi}_p \right\rangle\right|^2\left|\sum_{n\in I} \frac{a_n}{\widetilde{\mu}_p - \mu_n } \right|^2 \\
    &\leq \left\|\sum_{n \in I}\frac{a_n}{\textcolor{black}{\langle f_n,F\rangle}} g_n\right\|_{D(\widetilde{\mathcal{A}})}^2 \leq C_2\sum_{p\in \Z}  (1+|\widetilde{\mu}_p|^2) \left|\left\langle\mathcal{I}_\nu, \widetilde{\phi}_p \right\rangle\right|^2\left|\sum_{n\in I} \frac{a_n}{\widetilde{\mu}_p- \mu_n } \right|^2.\end{aligned}
\end{equation}
Now, similar estimates to the ones used to obtain \eqref{Iboundary} (though even simpler in this case, as only first order of $\gamma$ is required), also thanks to the controllability of our system,
lead to the existence of constants $m, M>0$ such that
\begin{equation}
    m \leq (1+|\widetilde{\mu}_p|^2) \left|\left\langle\mathcal{I}_\nu, \widetilde{\phi}_p \right\rangle\right|^2  \leq M.
\end{equation}
Moreover,
\begin{equation}
    \sum_{p\in \Z} \left|\sum_{n\in I}\frac{a_n}{\widetilde{\mu}_p - \mu_n }\right|^2 = \left\| \sum_{n \in I} a_n k_n \right\|^2,
\end{equation}
so that, \textcolor{black}{from Proposition \ref{riesz1} and as $(k_{n})$ is a Riesz basis }, there exist constants $C_1, C_2>0$ such that
\begin{equation}\label{Rieszineqforkn}
C_1\sum_{n \in I} |a_n|^2 \leq \sum_{p\in \Z} \left|\sum_{n\in I}\frac{a_n}{\widetilde{\mu}_p - \mu_n }\right|^2  \leq C_2 \sum_{n \in I} |a_n|^2,
\end{equation}
and we get, putting \eqref{Rieszineqforfptilde} and \eqref{Rieszineqforkn} together, and for some constants $C_1, C_2>0$,
 \begin{equation}\label{Rieszineqforgn}
 C_1\sum_{n \in I} |a_n|^2 \leq\left\|\sum_{n \in I} a_n \frac{g_n}{\textcolor{black}{\langle f_n,F\rangle}}\right\|_{D(\widetilde{\mathcal{A}})}^2\leq C_2\sum_{n \in I} |a_n|^2.
 \end{equation}

Now let us prove that the above family is complete. Let $\alpha\in D(\widetilde{\mathcal{A}})$ be such that
\begin{equation}\label{orthTogn}
\left\langle \alpha, \textcolor{black}{\frac{g_n}{\langle F,f_{n}\rangle}} \right\rangle_{D(\widetilde{\mathcal{A}})} =0, \quad \forall n \in \Z.
\end{equation}
Then \eqref{orthTogn} \textcolor{black}{and the definition of $(k_{n})$ given by \eqref{def-kn}} imply that
\begin{equation}\begin{aligned}
    0&=\sum_{p\in \Z} (1+|\mu_p|^2) \langle\alpha, \widetilde{f}_p \rangle \frac{\left\langle\mathcal{I}_\nu, \widetilde{\phi}_p \right\rangle}{\widetilde{\mu}_p - \mu_n } \\
    &=\left\langle \sum_{p\in\Z}(1+|\mu_p|^2)\langle\alpha, \widetilde{f}_p\rangle\left\langle\mathcal{I}_\nu, \widetilde{\phi}_p \right\rangle \widetilde{\phi}_p, k_n \right\rangle, \quad \forall n \in \Z.
    \end{aligned}
\end{equation}
Thanks to \eqref{NonZeroFeedbackCoeff}, this implies
\begin{equation}
    \left\langle \sum_{p\in\Z}(1+|\mu_p|^2)\langle\alpha, \widetilde{f}_p\rangle\left\langle\mathcal{I}_\nu, \widetilde{\phi}_p \right\rangle \widetilde{\phi}_p, k_n \right\rangle=0, \quad \forall n \in \Z,
\end{equation}
which implies, thanks to the completeness of the Riesz basis $(k_n)$, that
\begin{equation}
    \sum_{p\in\Z}(1+|\mu_p|^2)\langle\alpha, \widetilde{f}_p\rangle\left\langle\mathcal{I}_\nu, \widetilde{\phi}_p \right\rangle \widetilde{\phi}_p=0.
\end{equation}
Hence, thanks to the controllability of the system,
\begin{equation*}
    \alpha=0.
\end{equation*}
This proves the completeness of the family, and, together with \eqref{Rieszineqforgn}, the proposition.
\end{proof}
\begin{crllr}\label{Tisomorphism}
If there exist constants $c,C>0$ such that
\begin{equation}\label{Fgrowth}
c(1+|n|) \leq |\langle F, f_n\rangle| \leq C(1+|n|), \quad \forall n \in \Z,
\end{equation}
then the expression, for $\alpha \in D(A)$,
\begin{equation}\label{DefinitionOfT}
T\alpha:=\sum_{n\in \Z} \langle \alpha, f_n \rangle g_n=-\sum_{n\in \Z} \langle \alpha, f_n \rangle \textcolor{black}{\langle f_n,F\rangle} \sum_{p\in \Z} \frac{\left\langle\mathcal{I}_\nu, \widetilde{\phi}_p \right\rangle}{\widetilde{\mu}_p - \mu_n } \widetilde{f}_p
\end{equation}
defines an isomorphism of $D(\mathcal{A}) \rightarrow D(\widetilde{\mathcal{A}})$. {\color{black}Moreover, if $F$ is real-valued, then it maps real-valued functions to real-valued functions.}
\end{crllr}
{\color{black}\begin{proof}
The first part of the corollary follows immediately from the Riesz basis property of Proposition \ref{Riesz-prop}. To see that $T$ maps real-valued functions to real-valued functions, notice first that $\mathcal{I}_\nu$ is real-valued, as $\mathcal{I}$ and $f_0$ are real valued (see Proposition \ref{eigenf-prop}). Then, if $\alpha \in D(\mathcal{A})$ is real-valued, using Corollary \ref{cor-real-valued}, Propositions \ref{eigenf-prop} and \ref{eigenf-prop-tilde}, we have
\begin{equation}
    \overline{T\alpha}=-\sum_{n\in \Z} \langle \alpha, f_{-n} \rangle \langle f_{-n},F\rangle\sum_{p\in \Z} \frac{\left\langle\mathcal{I}_\nu, \widetilde{\phi}_{-p} \right\rangle}{\widetilde{\mu}_{-p} - \mu_{-n} } \widetilde{f}_{-p}=T\alpha.
\end{equation}
\end{proof}}

\subsection{Finding a suitable candidate: the operator equality}\label{TB-from-OpEq}

Now, for any $F \in \dist$ satisfying \eqref{Fgrowth}, we have used the heuristic kernel equation \eqref{kernelEquation}, \eqref{KernBoundx}, \eqref{KernBoundy} to build an isomorphism. The next step is to find a feedback law $F$ such that the corresponding isomorphism $T$ {\color{black}is indeed a backstepping transformation, which really maps the solutions of \eqref{sys1virt} to solutions of \eqref{target}. This translates into the following operator equality, which is the infinite-dimensional equivalent of $F$-equivalence, i.e. \eqref{backstepping-eq}}:
\begin{equation}\label{OperatorEquality}
T(-\mathcal{A}+\mathcal{I}_\nu F)=-\widetilde{\mathcal{A}}T
\end{equation}
on the domain
\begin{equation}\label{Domain}
D_F=\left\{\alpha \in D(\mathcal{A}), \quad -\mathcal{A}\alpha + \langle  \alpha, F\rangle \mathcal{I}_\nu \in D(\mathcal{A})\right\}.
\end{equation}

Let $F\in \dist$, satisfying \eqref{Fgrowth}. To prove \eqref{OperatorEquality}, it suffices to prove
\begin{equation}\label{WeakOpEq}
 \langle T(-\mathcal{A}\alpha+\langle  \alpha, F \rangle \mathcal{I}_\nu), \widetilde{\phi}_m\rangle=\langle-\widetilde{\mathcal{A}}T\alpha, \widetilde{\phi}_m\rangle , \quad \forall \alpha \in D_F, \quad \forall m \in \Z.
\end{equation}
Let $m\in \Z$, we consider the left-hand side of \eqref{WeakOpEq}, which can be rewritten as
\begin{equation}
    \langle -\mathcal{A}\alpha+\langle  \alpha, F \rangle \mathcal{I}_\nu, T^\ast \widetilde{\phi}_m\rangle_{D(\mathcal{A})\times D(\mathcal{A})^\prime}
    \label{Tdual}.
\end{equation}
Now, to evaluate the linear form $T^\ast \widetilde{\phi}_m$ on the function $\mathcal{A}\alpha+\langle \alpha, F \rangle \mathcal{I}_\nu$, we approximate the function by its truncated expansion in the basis $(\textcolor{black}{f}_n)_{\textcolor{black}{n\in\mathbb{Z}}}$:
\begin{equation}\begin{aligned}
    \langle T(-\mathcal{A}\alpha^{(N)}+\langle \alpha, F \rangle \mathcal{I}_\nu^{(N)}),  \widetilde{\phi}_m\rangle
&=\langle -\mathcal{A}\alpha^{(N)}+\langle \alpha, F \rangle \mathcal{I}_\nu^{(N)}, T^\ast \widetilde{\phi}_m\rangle_{D(\mathcal{A})\times D(\mathcal{A})^\prime} \\
&\xrightarrow[N \to \infty]{} \langle -\mathcal{A}\alpha+\langle \alpha,F \rangle \mathcal{I}_\nu, T^\ast \widetilde{\phi}_m\rangle_{D(\mathcal{A})\times D(\mathcal{A})^\prime}\end{aligned}
    \label{TruncApprox}
\end{equation}
where, for $N>0$,
\begin{equation}\label{PartialSums}
\begin{aligned}
    \alpha^{(N)}&:=\sum_{n=-N}^N \langle \alpha, f_n \rangle f_n, \\
    \mathcal{I}_\nu^{(N)}&:=\sum_{n=-N }^N \langle \mathcal{I}_\nu, f_n \rangle f_n.
\end{aligned}
\end{equation}
We can then make the following computations for $\alpha\in D_F$ and $N>0$, using \eqref{ProjKernelEquation} and \eqref{DefinitionOfT}:
\begin{equation}\label{TroncEq}
    \begin{aligned}
       T(-\mathcal{A}\alpha^{(N)}+\langle \alpha, F \rangle \mathcal{I}_\nu^{(N)})&= \langle\alpha, F \rangle T\mathcal{I}_\nu^{(N)} +\sum_
       {n=-N }^N - \mu_n \langle \alpha, f_n \rangle g_n  \\
       &= \langle \alpha, F \rangle T\mathcal{I}_\nu^{(N)} +\sum_{n=-N }^N  \langle \alpha, f_n \rangle (-\widetilde{\mathcal{A}}g_n - \textcolor{black}{\langle f_n,F\rangle} \mathcal{I}_\nu)   \\
       &= -\widetilde{\mathcal{A}} T\alpha^{(N)}-\langle\alpha^{(N)}, F \rangle \mathcal{I}_\nu + \langle \alpha, F\rangle T\mathcal{I}_\nu^{(N)}.
    \end{aligned}
\end{equation}
Now, as
\begin{equation}
    \label{EasyLimit}
   \langle \widetilde{\mathcal{A}} T\alpha^{(N)}, \widetilde{\phi}_m \rangle \xrightarrow[N\to\infty]{} \langle \widetilde{\mathcal{A}} T\alpha, \widetilde{\phi}_m \rangle, \quad \forall m \in \Z, \quad \forall \alpha \in D_F,
\end{equation}
to get \eqref{WeakOpEq} {\color{black}(that is, $F$-equivalence)} \textcolor{black}{from \eqref{TroncEq}}, it suffices to have
\begin{equation}\label{OpEqLimits}
\begin{aligned}
    \langle \alpha^{(N)}, F \rangle & \xrightarrow[N\to\infty]{} \langle \alpha, F \rangle, \quad \forall \alpha \in D_F, \\
   \langle T\mathcal{I}_\nu^{(N)}, \widetilde{\phi}_m\rangle &\xrightarrow[N\to\infty]{} \langle \mathcal{I}_\nu, \widetilde{\phi}_m\rangle, \quad \forall m \in \Z.
\end{aligned}
\end{equation}
\begin{rmrk}
The second limit corresponds to a weak version of the $TB=B$ equation in \eqref{FiniteDOpEq}. In itself, it is only consistent that a condition of this type should appear, as we have started our construction in Section \ref{Kernel equations} by considering the infinite-dimensional analog of the first equation of \eqref{FiniteDOpEq}. But this is a modified backstepping equation, obtained by adding the $TB=B$ condition to the equation \eqref{backstepping-eq} to make the backstepping problem linear and well-posed. Accordingly, the nonlinearity in \eqref{backstepping-eq} translates as a nonlocal term in the infinite-dimensional analog, whereas the kernel equation \eqref{kernelEquation} we get considering the modified equation in \eqref{FiniteDOpEq} is a simple linear PDE. One can also consider that the nonlocal term is simplified by applying a formal $TB=B$ condition:
\begin{equation}\label{FormalTB=B}
    \int_0^L K(x,y) \mathcal{I}_\nu(y)dy = \mathcal{I}_\nu(x), \quad \forall x \in (0,L).
\end{equation}
However the equation above is actually purely formal, and the ``right way'' to formulate it is the weak form in \eqref{OpEqLimits}, which is not surprising, as, according to Corollary \ref{Tisomorphism}, $T\mathcal{I}_\nu$  is not defined and thus \eqref{FormalTB=B} has no real mathematical meaning. This is already the case for transport equations (\cite{ZhangRapidStab,ZhangFiniteStab}).
\end{rmrk}

\section{Backstepping transformation and feedback law}
\label{s5}
To find an actual backstepping transformation, \textcolor{black}{we want to construct a feedback $F$ that satisfies \eqref{OpEqLimits}. We} will first study the second limit of \eqref{OpEqLimits}, which will determine the value of $F$.
Then we will check that $F$ thus defined satisfies the first limit of \eqref{OpEqLimits}, so that the corresponding isomorphism $T$ is indeed a backstepping transformation.
\subsection{Construction of the feedback law}
\label{deuxeq}
Keeping the notations of \eqref{PartialSums}, 
first note that
\begin{equation}
    \begin{aligned}
        T\mathcal{I}_\nu^{(N)}&={\color{black} \sum_{n=-N}^N \langle \mathcal{I}_\nu,f_n\rangle g_n }\\
        &=-\sum_{n=-N}^N \langle \mathcal{I}_\nu,f_n\rangle \textcolor{black}{\langle f_n,F\rangle} \sum_{p\in \Z} \frac{\langle \mathcal{I}_\nu, \widetilde{\phi}_p \rangle}{\widetilde{\mu}_p-\mu_n} \widetilde{f}_p \\
        &= - \sum_{p\in \Z} \langle \mathcal{I}_\nu, \widetilde{\phi}_p \rangle  \left(\sum_{n=-N}^N \frac{\langle \mathcal{I}_\nu,f_n\rangle \textcolor{black}{\langle f_n,F\rangle}}{\widetilde{\mu_p}-\mu_n}\right) \widetilde{f}_p.    \end{aligned}
\end{equation}
Hence, for $m\in  \Z$, using \eqref{tautilde}, \textcolor{black}{and recalling that $\langle \widetilde{f}_p,\widetilde{\phi}_m\rangle=0$, for any $p\neq m$ from the definition of a biorthogonal family,}

\begin{equation}
    \begin{aligned}
      \langle T\mathcal{I}_\nu^{(N)}, \widetilde{\phi}_m\rangle&= -  \left\langle \sum_{p\in \Z} \langle \mathcal{I}_\nu, \widetilde{\phi}_p \rangle  \left(\sum_{n=-N}^N \frac{\langle \mathcal{I}_\nu,f_n\rangle \textcolor{black}{\langle f_n,F\rangle}}{\widetilde{\mu}_p-\mu_n}\right) \widetilde{f}_p, \widetilde{\phi}_m\right\rangle \\
      &=- \langle \mathcal{I}_\nu, \widetilde{\phi}_m \rangle\left\langle \sum_{p\in \Z}  \left(\sum_{n=-N}^N \frac{\langle \mathcal{I}_\nu,f_n\rangle \textcolor{black}{\langle f_n,F\rangle}}{\widetilde{\mu}_p-\mu_n}\right) \widetilde{f}_p, \widetilde{\phi}_m\right\rangle \\
      &=- \langle \mathcal{I}_\nu, \widetilde{\phi}_m \rangle\left\langle  \sum_{n=-N}^N \langle \mathcal{I}_\nu,f_n\rangle \textcolor{black}{\langle f_n,F\rangle}\sum_{p\in \Z} \frac{1}{\widetilde{\mu}_p-\mu_n}\widetilde{f}_p, \widetilde{\phi}_m\right\rangle \\
      &= - \langle \mathcal{I}_\nu, \widetilde{\phi}_m \rangle\left\langle  \sum_{n=-N}^N \langle \mathcal{I}_\nu,f_n\rangle \textcolor{black}{\langle f_n,F\rangle}\widetilde{\tau}_\mathcal{A}^{-1} \frac{f_n}{f_{n,1}(0)}, \widetilde{\phi}_m\right\rangle \\
      &= -\langle \mathcal{I}_\nu, \widetilde{\phi}_m \rangle\left\langle \widetilde{\tau}_\mathcal{A}^{-1} \sum_{n=-N}^N \langle \mathcal{I}_\nu,f_n\rangle \textcolor{black}{\langle f_n,F\rangle}\frac{f_n}{f_{n,1}(0)}, \widetilde{\phi}_m\right\rangle .
    \end{aligned}
\end{equation}
 Moreover, by definition \eqref{tautilde} of $\widetilde{\tau}_\mathcal{A}, (\widetilde{f}_n), (\widetilde{\phi}_n)$, we have:
 \begin{equation}
     \left(\widetilde{\tau}_\mathcal{A}^{-1}\right)^\ast \widetilde{\phi}_m= \frac{1}{\widetilde{\phi}_{m,1}(0)(1-e^{-2\mu L})} \widetilde{\phi}_m.
 \end{equation}
Hence
\begin{equation}
    \begin{aligned}
      \langle T\mathcal{I}_\nu^{(N)}, \widetilde{\phi}_m\rangle&= -\frac{\langle \mathcal{I}_\nu, \widetilde{\phi}_m \rangle}{\overline{\widetilde{\phi}_{m,1}(0)}(1-e^{-2\mu L})}\left\langle \sum_{n=-N}^N \langle \mathcal{I}_\nu,f_n\rangle \textcolor{black}{\langle f_n,F\rangle} \frac{f_n}{f_{n,1}(0)}, \widetilde{\phi}_m\right\rangle \\
      &=  -\frac{\langle \mathcal{I}_\nu, \widetilde{\phi}_m \rangle}{\overline{\widetilde{\phi}_{m,1}(0)}(1-e^{-2\mu L})}\sum_{n=-N}^N \langle \mathcal{I}_\nu,f_n\rangle \frac{\textcolor{black}{\langle f_n,F\rangle}}{f_{n,1}(0)}\left\langle  f_n, \widetilde{\phi}_m\right\rangle .
    \end{aligned}
\end{equation}
Now let us set
\begin{equation}
    \label{Fcoeff}
    \langle f_n, F \rangle:=-2\tanh(\mu L) \frac{\textcolor{black}{(f_{n,1}(0))^{2}}}{\langle \mathcal{I}_\nu, f_n \rangle}, \quad \forall n \in \Z ,
\end{equation}
so that
\begin{equation}\label{PartialSumforTB}
    \langle T\mathcal{I}_\nu^{(N)}, \widetilde{\phi}_m\rangle = \frac{\textcolor{black}{2}\langle \mathcal{I}_\nu, \widetilde{\phi}_m \rangle}{\overline{\widetilde{\phi}_{m,1}(0)}(1+e^{-    2\mu L})}\sum_{n=-N }^N f_{n,1}(0)\left\langle  f_n, \widetilde{\phi}_m\right\rangle .
\end{equation}
The sum in \eqref{PartialSumforTB} is analog to the Dirichlet sum which appears in \cite{ZhangRapidStab} \textcolor{black}{and that we recall here: for $f\in C^{1}_{pw}([0,L])$, then
\begin{equation}
\sum\limits_{n=-N}^{N}\langle f ,e_{p}\rangle\xrightarrow[N\rightarrow +\infty]{}\frac{f(0)+f(L)}{2}
\end{equation}
 where $(e_{p})_{p\in \mathbb{Z}}=(e^{\frac{2i\pi p x}{L}})_{p\in \mathbb{Z}}$. Note that this can be extended to functions $f\in(C^{1})^{2}$ such that $f_{1}(0)=-f_{2}(0)$ and $f_{1}(L)=-f_{2}(L)$ and the basis $(E_{p})_{p\in \mathbb{Z}}=\left(\left(e^{\frac{i\pi p x}{L}},-e^{\frac{-i\pi p x}{L}}\right)\right)_{p\in\mathbb{Z}}$. Indeed, using the gluing defined below in \eqref{attachementmap}, from each of these functions one can define $\underline{f}\in C^{1}_{pw}([0,2L])$ and $(\underline{E_{p}})_{p\in \mathbb{Z}}=(e^{\frac{2i\pi p x}{2L}})_{p\in \mathbb{Z}}$, thus
 \begin{equation}\label{dirichlet22}
\sum\limits_{n=-N}^{N}\langle f ,E_{p}\rangle\xrightarrow[N\rightarrow +\infty]{}\frac{\underline{f}(0)+\underline{f}(2L)}{2}=\frac{f_{1}(0)-f_{2}(0)}{2}.
\end{equation}
}
As it turns out, \textcolor{black}{the sum in \eqref{PartialSumforTB}} has the same remarkable property of converging towards the mean of the left and right values of $\widetilde{\phi}_m$. This comes from a powerful equiconvergence theorem proved in \cite{Komornik1984} \textcolor{black}{for a Schr\"{o}dinger operator (see \cite{Komornik1986} for its generalization on operators of order $n\geq 2$)}:
 \begin{thrm}
 \label{KomornikSchrodinger}
 Let $(u_{p})\in L^{2}(0,L)^{\mathbb{Z}}$ be a complete orthonormal system of eigenfunctions associated to the eigenvalues $\lambda_{p}$ of the Schr\"{o}dinger operator $-\partial_{xx} +q$, where $q$ is a locally integrable function.
 Let us denote for any $f\in L^{2}(0,L)$, $\mu>0$, and $x\in(0,L)$,
 \begin{equation}
 \sigma_{\mu}(f,x)=\sum\limits_{|Re(\sqrt{\lambda_{p}})|<\mu} \langle f, u_{p}\rangle u_{p}(x).
 \end{equation}
 Similarly, let $(\hat{u}_{p})\in L^{2}(0,L)^{\mathbb{Z}}$ be a complete orthonormal system of eigenfunctions associated to the eigenvalues $\hat{\lambda}_{p}$ of the Schr\"{o}dinger operator $-\partial_{xx} +\hat{q}$, where $\hat{q}$ is a locally integrable function and denote
 \begin{equation}
\hat{\sigma}_{\mu}(f,x)=\sum\limits_{|Re(\sqrt{\hat{\lambda}_{p}})|<\mu} \langle f, \hat{u}_{p}\rangle \hat{u}_{p}(x).
 \end{equation}
 Then, given any compact interval $K\subset(0,L)$, for any $f\in L^{2}(0,L)$, the following holds
 \begin{equation}
    \lim\limits_{\mu\rightarrow +\infty}\sup\limits_{x\in K}|\sigma_{\mu}(f,x)-\hat{\sigma}_{\mu}(f,x)|=0.
 \end{equation}
 \end{thrm}
The proof of this theorem can be adapted to our operator, leading to the following proposition.
\begin{prpstn}\label{PropEquiconvergence}
Let us denote
\begin{equation}
    \sigma_{\mu}(f,x)=\sum\limits_{|Im(\mu_{p})|<\mu} \langle f, f_{p}\rangle f_{p}(x)
\end{equation}
and
\begin{equation}
    p_{\mu}(f,x)=\sum\limits_{|Im(\mu_{p}^{\textcolor{black}{(0)}})|<\mu} \langle f, E_{p}\rangle E_{p}(x)
\end{equation} where
$(E_{p})_{p\in\mathbb{Z}}=\left(\left(e^{\frac{i\pi p x}{L}},-e^{\frac{-i\pi p x}{L}}\right)\right)_{p\in\mathbb{Z}}$. One has for any compact $K\subset\textcolor{black}{[}0,L)$
\begin{equation}
 \lim\limits_{\mu\rightarrow +\infty} \sup\limits_{x\in K}|\sigma_{\mu}(f,x)-p_{\mu}(f,x)|=0.
\end{equation}
\label{EquiConv}
\end{prpstn}
A way to adapt the proof of \cite[Theorem 2]{Komornik1984} in order to get Proposition \ref{PropEquiconvergence} is given in  Appendix \ref{appendixB}.

Now let us recall that for $m\in \Z$, $\widetilde{\phi}_m$ is the solution of a linear ODE given by the operator $\widetilde{\mathcal{A}}^\ast$. Thus $\underline{\widetilde{\phi}_m}\in C^1_{pw}(0,2L)$ (recall the definition of $\underline{f}$ given in \eqref{attachementmap}) and, applying Proposition \ref{PropEquiconvergence} and the Dirichlet convergence theorem \textcolor{black}{given by \eqref{dirichlet22}}, we get:
\begin{equation}\label{281}
    \sum_{n=-N}^N f_{n,1}(0)\left\langle  f_n, \widetilde{\phi}_m\right\rangle = \overline{ \sum_{n=-N}^N f_{n,1}(0)\left\langle\widetilde{\phi}_m, f_n\right\rangle}  \xrightarrow[N\to \infty]{} 
    \textcolor{black}{\frac{\overline{\widetilde{\phi}_{m,1}(0)-\widetilde{\phi}_{m,2}(0)}}{2}}
    \end{equation}
Now,
using the boundary conditions \eqref{AtildeStar} for $\widetilde{\phi}_m$, we have
\begin{equation}
\begin{aligned}
    \frac{\overline{\widetilde{\phi}_{m,1}(0)-\widetilde{\phi}_{m,2}(0)}}{2}
    &=\frac{\overline{\widetilde{\phi}_{m,1}(0)(1+e^{-2\mu L})}}{2}
\end{aligned}
\end{equation}
so that finally, if $F$ is defined by \eqref{Fcoeff}, we get from \eqref{PartialSumforTB}:
\begin{equation}
    \label{TB=B}
\langle T \mathcal{I}_\nu^{(N)}, \widetilde{\phi}_m \rangle \xrightarrow[N \to \infty]{}  \langle \mathcal{I}_\nu, \widetilde{\phi}_m \rangle, \quad \forall m \in \Z,
\end{equation}
which is the second limit of \eqref{OpEqLimits}.

\subsection{Regularity of the feedback law}\label{section-feedback-reg}
Let us now study the regularity of the feedback law defined by \eqref{Fcoeff}. First note that, thanks to \eqref{Igrowth} that was deduced from the moments condition that gave us the controllability of the system, and thanks to
Lemma \ref{LmmBoundaryBehaviour},
\begin{equation}
    \label{Fgrowth-mu}
    \textcolor{black}{c(1+|\mu_n|)}\leq |\langle f_n, F \rangle | \leq C \textcolor{black}{(1+|\mu_n|)}, \quad n\in \Z,
\end{equation}
for some constants $c,C>0$, which is \textcolor{black}{exactly} to
\eqref{Fgrowth}
for some other constants, thanks to \eqref{mu_n asymp}.
Then we have the following regularity result, analogous to \cite[Lemma 2.1]{ZhangRapidStab}:
\begin{lmm}\label{lemreg}
$F\in \dist$ defined by \eqref{Fcoeff} defines a linear form on $D(\mathcal{A}^2)$  which is continuous for $\norm_{D(\mathcal{A}^2)}$ but not for $\norm_{D(\mathcal{A})}$.
\end{lmm}

\begin{crllr}
If $F$ is defined by \eqref{Fcoeff}, then \textcolor{black}{the domain $D_{F}$ given by} \eqref{Domain} defines a dense domain of $D(\mathcal{A})$.
\end{crllr}
\begin{proof}
It is clear from \eqref{Domain} that:
\begin{equation}
    \mathcal{K}_F:=\{\alpha \in D(\mathcal{A}^2), \quad \langle \alpha, F \rangle =0 \} \subset D_F.
\end{equation}
Now, $\mathcal{K}_F$ is the kernel of a non-continuous linear form, so it is dense in $D(\mathcal{A}^2)$ for $\norm_{D(\mathcal{A})}$, which is in turn dense in ${D(\mathcal{A})}$ for $\norm_{D(\mathcal{A})}$, hence the density of $D_F$ in $D(\mathcal{A})$.
\end{proof}

Now, to obtain the second limit of \eqref{OpEqLimits}, we need a more precise knowledge of the regularity of $F$.

Recall that \eqref{I-asympt} had given more information on the growth of the coefficients of $\mathcal{I}_\nu$.
Now, taking the inverse, \textcolor{black}{and using \eqref{Igrowth}} we get the following property for the coefficients of the feedback law $F$:
\begin{equation}
    \label{F-asympt}
    \left(\frac1{\mu_n} \left(\langle f_n, F \rangle+\frac{\tanh(\mu L)}{\tau_n^{\mathcal{I}}} f_{n,1}(0)\mu_n\right)\right)\in \ell^2(\Z),
    \end{equation}
    where
    \begin{equation}
        \tau_n^{\mathcal{I}}:=\left(e^{\int_0^L \delta} \frac{f_{n,1}(L)}{f_{n,1}(0)}-1\right), \quad \forall n \in \Z.
    \end{equation}
  Note that we have, thanks to Lemma \textcolor{black}{\ref{LmmBoundaryBehaviour}} and \eqref{Iboundary},
  \begin{equation}
      \label{tau-I-growth}
      c \leq |\tau_n^\mathcal{I}| \leq  C, \forall n\in \Z ,
  \end{equation}
  for some constants $c,C>0$, so that the operator defined by
  \begin{equation}
      \label{tau-I}
      \tau^{\mathcal{I}} \alpha := \sum_{n\in \Z } \tau_n^{\mathcal{I}} \langle \alpha, f_n \rangle f_n, \quad \forall \alpha\in \textcolor{black}{(L^2)^2}
  \end{equation}
  is an isomorphism of $D(\mathcal{A}^s), \forall s\geq 0$.
  \begin{rmrk}\label{rmk-commu}
  Because it is a diagonal operator {\color{black}on $(f_n)$}, $ \tau^{\mathcal{I}}$ commutes with $\mathcal{A}$.
  \end{rmrk}
    \textcolor{black}{We recall now the definition of the spaces $X^{s}$ for $s\geq0$ 
  \begin{equation}
    X^s:=\{f\in (L^2_{(0)})^2, \quad (\tau^\mathcal{I})^{-1}(\Lambda \partial_x f + \delta(x) J f) \in (H^{s-1})^2\}.
\end{equation}
    }
    Let us then note:
    \begin{equation}
        \label{sing-part-coeff}
        h_n:=-\frac{\tanh(\mu L)}{\tau_n^\mathcal{I}} f_{n,1}(0)\mu_n, \quad n \in \Z ,
    \end{equation}
    and $h\in \dist$ the linear form defined by:
    \begin{equation}
        \label{sing-part}
        \langle f_n, h \rangle=h_n, \quad \forall n \in \Z .
    \end{equation}
    Then we have the following proposition.
    \begin{prpstn}\label{FeedbackReg}
The linear form $h$ defines the following linear form on $X^2\cap D(\mathcal{A})$,
continuous for $\norm_{X^2}$:
\begin{equation}\label{sing-part-val}
\langle \alpha, h \rangle=-\tanh(\mu L)\frac{ \left(\mathcal{A}(\tau^{\mathcal{I}})^{-1} \alpha\right)_1(0)-\left(\mathcal{A}(\tau^{\mathcal{I}})^{-1} \alpha\right)_2(0)}2
\end{equation}
Moreover, $\widetilde{F}:=F-h$ is continuous for $\|\cdot\|_{D(\mathcal{A})}$, so that $F$ is actually defined on $X^2\cap D(\mathcal{A})$, and is continuous for $\|\cdot\|_{X^2}$, but not for $\|\cdot\|_{D(\mathcal{A})}$.
\end{prpstn}
\textcolor{black}{\begin{rmrk}
{\color{black}With these notations,} $h$ is the ``singular part" of $F$, i.e. the part that limits the regularity of $F$, the rest $F-h$ being continuous for $\|\cdot\|_{D(\mathcal{A})}$. In the following we will therefore study specifically this singular part.
\end{rmrk}}
\begin{proof}[Proof of Proposition \ref{FeedbackReg}]
The continuity of $\widetilde{F}$ on $D(\mathcal{A})$ follows directly from \eqref{F-asympt}.
On the other hand, let $\alpha \in X^3\cap D(\mathcal{A})$. Then $\mathcal{A}(\tau^{\mathcal{I}})^{-1} \alpha\in (H^2)^2$, and thus satisfies the conditions of the Dirichlet convergence theorem, so that, using again Proposition \ref{PropEquiconvergence}, we get:
\begin{equation}\label{h-for-smooth}
\begin{aligned}
    \langle \alpha, h \rangle &= -\sum_{n\in \Z} \langle \alpha, f_n \rangle \frac{\tanh(\mu L)}{\tau_n^\mathcal{I}} f_{n,1}(0)\mu_n \\
    &= -\tanh(\mu L) \lim_{N\to \infty} p_{N}\left(\mathcal{A}(\tau^{\mathcal{I}})^{-1} \alpha, 0\right) \\
    &=  -\tanh(\mu L)\frac{ \left(\mathcal{A}(\tau^{\mathcal{I}})^{-1} \alpha\right)_1(0)-\left(\mathcal{A}(\tau^{\mathcal{I}})^{-1} \alpha\right)_2(0)}2.
    \end{aligned}
\end{equation}
It is clear from this last expression that $h$ is continuous for $\norm_{X^2}$, by the continuous injection of $(H^1)^2 \rightarrow (L^\infty)^2$.
Now let us show that $X^3\cap D(\mathcal{A})$ is dense in $X^2\cap D(\mathcal{A})$ for the $X^2$ norm. Let $\alpha\in X^2 \cap D(\mathcal{A})$, and $\varepsilon >0$. Then, by density of $(H^2)^2$ in $(H^1)^2$ there exists $a_\varepsilon \in (H^2)^2$ such that
\begin{equation}\label{a-eps-es}
    \|a_\varepsilon-\mathcal{A}(\tau^{\mathcal{I}})^{-1} \alpha \|_{(H^1)^2} \leq \varepsilon.
\end{equation}
Now,
\begin{equation}
    \langle\mathcal{A}(\tau^{\mathcal{I}})^{-1} \alpha, f_0\rangle =0
\end{equation}
so that
\begin{equation}\label{smallzeromode}
    |\langle a_\varepsilon, f_0 \rangle| \leq \varepsilon,
\end{equation}
and, setting
\begin{equation}
    \widetilde{a}_\varepsilon:= a_{\varepsilon}- \langle a_\varepsilon, f_0 \rangle f_0,
\end{equation}
we get, thanks to \eqref{smallzeromode} and the smoothness of $f_0$ as the solution of {\color{black} a linear ordinary differential equation ($\delta$ is smooth), that}
\begin{equation}
\begin{array}{c}
    \widetilde{a}_\varepsilon \in (H^2)^2, \\
    \left\|\widetilde{a}_\varepsilon-\mathcal{A}(\tau^{\mathcal{I}})^{-1} \alpha \right\|_{(H^1)^2} \leq C\varepsilon,
    \end{array}
\end{equation}
{\color{black}where} the constant $C>0$ \textcolor{black}{ only depends on $\lVert f_{0}\rVert_{(H^{1})^{2}}$}.
Let us now set:
\begin{equation}
    \alpha_\varepsilon := \langle \alpha, f_0 \rangle f_0 + \sum_{n \in \Z^\ast} \tau_n^{\mathcal{I}}\frac{\langle \widetilde{a}_\varepsilon, f_n \rangle}{\mu_n} f_n \in D(\mathcal{A}).
    \end{equation}
Then
\begin{equation}\label{a-eps-and-alpha-eps}
    \mathcal{A}(\tau^{\mathcal{I}})^{-1}  \alpha_\varepsilon =  \tilde{a}_\varepsilon,
\end{equation}
so that $\alpha_\varepsilon\in X^3$. Moreover, by \eqref{a-eps-es} \textcolor{black}{and \eqref{tau-I-growth}},
\begin{equation}\label{a-eps-A}
    \| \mathcal{A}\alpha_\varepsilon - \mathcal{A} \alpha \| \leq C \varepsilon,
\end{equation}
for some constant $C>0$. As $|\mu_n| \geq \delta >0, \ \forall |n|\geq 1$, and by definition of
$\alpha_\varepsilon$,
$\langle\alpha_{\varepsilon}-\alpha,f_{0}\rangle=0$, we get
\begin{equation}\label{L2-part}
    \| \alpha_\varepsilon - \alpha \| \leq C^\prime \varepsilon,
\end{equation}
and \eqref{a-eps-es}, \eqref{a-eps-and-alpha-eps}, \textcolor{black}{\eqref{a-eps-A}}, and \eqref{L2-part} yield
 \begin{equation}
     \label{density-es}
     \| \alpha_\varepsilon - \alpha \|_{X^2} \leq C^{\prime\prime} \varepsilon,
 \end{equation}
This proves the density of $X^3 \cap D(\mathcal{A})$ in $X^2\cap D(\mathcal{A})$.
We can then continuously extend \eqref{h-for-smooth} to $X^2\cap D(\mathcal{A})$, which proves the rest of the proposition.
\end{proof}

Now that we have some knowledge on $F$ defined by \eqref{Fcoeff}, we get some more knowledge on the corresponding domain $D_F$:
\begin{lmm}
\textcolor{black}{With the above definitions, the following space} inclusion holds:
\begin{equation}\label{DF-and-X}
D_F \subset X^2.
\end{equation}
\end{lmm}
\begin{proof}Recall that
\begin{equation}
D_F=\left\{\alpha \in D(\mathcal{A}); \quad -\mathcal{A}\alpha + \langle  \alpha, F\rangle \mathcal{I}_\nu \in D(\mathcal{A})\right\}.
\end{equation}
In particular, {\color{black} as $\tau^{\mathcal{I}}$ is an isomorphism of $D(\mathcal{A}^s)$,}
\begin{equation}\label{total-reg}
        (\tau^{\mathcal{I}})^{-1}\left(-\mathcal{A}\alpha + \langle \alpha, F\rangle \mathcal{I}_\nu \right) \in D(\mathcal{A}).
\end{equation}
{\color{black}We will prove that the second component in the preceding expression is in $(H^1)^2$}.  By \eqref{tau-I} and \eqref{I-asympt},
\begin{equation}\label{in-DA}
     (\tau^{\mathcal{I}})^{-1}\mathcal{I}_\nu-\sum_{n\in \Z^\ast } \frac{f_{n,1}(0)}{\mu_n} f_n \in D(\mathcal{A}).
\end{equation}
Now, let
\begin{equation}
    \varphi(x):=\begin{pmatrix} \frac12-\frac{x}{2L} \\[1em] \frac12-\frac{x}{2L} \end{pmatrix} \in (H^1)^2.
\end{equation}
{\color{black}By integration by parts it is clear} that
\begin{equation}
    \varphi - \sum_{n\in \Z^\ast} \frac{f_{n,1}(0)}{\mu_n } f_n \in D(\mathcal{A}),
\end{equation}
which implies, together with \eqref{in-DA},
\begin{equation}\label{part2-reg}
     (\tau^{\mathcal{I}})^{-1}\mathcal{I}_\nu - \varphi=(\tau^{\mathcal{I}})^{-1}\mathcal{I}_\nu-\sum_{n\in \Z^\ast } \frac{f_{n,1}(0)}{\mu_n} f_n - \left( \varphi - \sum_{n\in \Z^\ast} \frac{f_{n,1}(0)}{\mu_n } f_n  \right)\in D(\mathcal{A})
\end{equation}
so that
\begin{equation}
    (\tau^{\mathcal{I}})^{-1} \mathcal{I}_\nu  \in (H^1)^2.
\end{equation}
{\color{black}Finally, thanks to Remark \ref{rmk-commu}},  by \eqref{total-reg} and \eqref{part2-reg},
\begin{equation}
    \mathcal{A}(\tau^{\mathcal{I}})^{-1}\alpha=  {\color{black}(\tau^{\mathcal{I}})^{-1} \mathcal{A}\alpha}\in (H^1)^2.
\end{equation}
\end{proof}

Now we can tackle the first limit of \eqref{OpEqLimits}. Let $\alpha \in D_F$. \textcolor{black}{As}
\begin{equation}
    \alpha^{(N)} \xrightarrow[N\to \infty]{D(\mathcal{A})} \alpha,
\end{equation}
by Proposition \ref{FeedbackReg}, we only need to study the ``singular part'' $h$ of the feedback $F$. By \eqref{PartialSums} \textcolor{black}{and \eqref{fn-f-n}}, we have:
\begin{equation}
    \label{partial-h}
    \begin{aligned}
        \langle \alpha^{(N)}, h \rangle &= -\sum_{n=-N }^N \langle \alpha, f_n \rangle \frac{\tanh(\mu L)}{\tau_n^\mathcal{I}} f_{n,1}(0)\mu_n \\
        &=-\frac{\tanh(\mu L)}2 \sum_{n=-N }^N  \left(\frac{\langle \alpha, f_n \rangle}{\tau_n^\mathcal{I}}\mu_n+\frac{\langle \alpha, f_{-n} \rangle}{\tau_{-n}^\mathcal{I}}\mu_{-n}\right)  f_{n,1}(0) \\
        &=- \frac{\tanh(\mu L)}2 \sum_{n=-N }^N \left\langle \mathcal{A} ((\tau^{\mathcal{I}})^{-1}\alpha- \sigma((\tau^{\mathcal{I}})^{-1}\alpha)), f_n \right\rangle  f_{n,1}(0),
    \end{aligned}
\end{equation}
where
\begin{equation}
   \sigma(f)(x):=\sum_{n\in \Z} \langle f, f_{-n} \rangle f_n, \quad \forall f \in (L^2)^2.
\end{equation}
\textcolor{black}{Note that we used the fact that for $f\in D(\mathcal{A})$, $\mathcal{A}\sigma(f)=-\sigma(\mathcal{A}f)$.}
Now, as $\alpha \in X^2 \cap D(\mathcal{A})$, we have
\begin{equation}
    \mathcal{A} (\tau^{\mathcal{I}})^{-1} \alpha = \Lambda \partial_x (\tau^{\mathcal{I}})^{-1} \alpha + \delta(x) J( \tau^{\mathcal{I}})^{-1} \alpha \in (H^1)^2
\end{equation}
so
\begin{equation}
    \Lambda \partial_x( \tau^{\mathcal{I}})^{-1} \alpha  \in (L^2)^2,
\end{equation}
hence
\begin{equation}(\tau^{\mathcal{I}})^{-1} \alpha \in (H^1)^2.
\end{equation}
Repeating the same argument, we get
\begin{equation}
   ( \tau^{\mathcal{I}})^{-1} \alpha \in (H^2)^2 \cap D(\mathcal{A}).
\end{equation}
We now use the following lemma:
\begin{lmm}
\label{sym-reg}
Let $f\in (H^2)^2\cap D(\mathcal{A})$. Then
\begin{equation}
    f-\sigma(f)\in D(\mathcal{A}^2).
\end{equation}
\end{lmm}
\begin{proof}
Using the regularity of $f$, we write, by double integration by parts:
\begin{equation}
\begin{aligned}
    \langle f, f_n \rangle &= \frac1{\mu_n} \langle \mathcal{A}f, f_n \rangle \\
    &= \frac{ \left(\mathcal{A}f_{1}(L)+\mathcal{A}f_{2}(L)\right)f_{n,1}(L)- \left(\mathcal{A}f_{1}(0)+\mathcal{A}f_{2}(0)\right)f_{n,1}(0)}{(\mu_n)^2}  \\
    &\quad \quad \quad  + \frac1{(\mu_n)^2}\langle \mathcal{A}^2f, f_n \rangle,\text{  }
    \end{aligned}
\end{equation}
Hence, using \eqref{imaginary-ev} and \eqref{fn-f-n}, we get,
\begin{equation}
    \langle f, f_n \rangle - \langle f, f_{-n} \rangle = \frac1{(\mu_n)^2}\langle \mathcal{A}^2f, f_n-f_{-n} \rangle.
\end{equation}
As $\mathcal{A}^2f \in (L^2)^2$, this proves the Lemma.
\end{proof}
From Lemma \ref{sym-reg}, we get
\begin{equation}
    \mathcal{A} ((\tau^{\mathcal{I}})^{-1}\alpha- \sigma((\tau^{\mathcal{I}})^{-1}\alpha)) \in D(\mathcal{A}).
\end{equation}
Hence, the last sum of \eqref{partial-h} converges absolutely. As $\alpha \in X^2\cap D(\mathcal{A})$, by Proposition \ref{FeedbackReg} and uniqueness of the limit, we get
 \begin{equation}\langle \alpha^{(N)}, F \rangle \xrightarrow[N\to \infty]{} \langle \alpha, F \rangle,\end{equation}
 which is the first limit of \eqref{OpEqLimits}.

 Thus, by the results of the two previous subsections and \eqref{OpEqLimits}, we have the following proposition:
 \begin{prpstn}\label{opeqprop}
 Let $F$ be defined by \eqref{Fcoeff}, and $T$ accordingly defined by \eqref{DefinitionOfT}. Then $T$ satisfies \eqref{OperatorEquality} on the domain $D_F$.
 \end{prpstn}

\subsection{Well-posedness and stability of the closed-loop system}\label{s6}

Now that we have constructed a pair $(T,F)$ that satisfies \eqref{OperatorEquality}, let us check that the closed-loop system \eqref{sys1virt} corresponding to the feedback $F$ is well-posed in some sense. The idea is to use the dynamics of the target system \eqref{target}, as \eqref{OperatorEquality} essentially means that $T$ exchanges the dynamics of the target system and the closed-loop system. We have seen in Section \ref{section-target} that \eqref{target} is well-posed, more specifically, \textcolor{black}{looking at Section \ref{subsubsection-expstab-target},} $\widetilde{\mathcal{A}}$ with domain $D(\widetilde{\mathcal{A}}^2)\subset D(\widetilde{\mathcal{A}})$ generates a contraction semigroup on $D(\widetilde{\mathcal{A}})$ for the norm defined by the Lyapunov function \eqref{defV} with $p=1$. We note that semigroup $\widetilde{S}(t), t\geq0$.

{\color{black} Following the intuition that $T$ exchanges the dynamics of the target system and the closed-loop system, let us now define the following semigroup:
\begin{equation}\label{closed-loop-semigroup}
    \begin{array}{rccl}
    S:&\RR^+ & \rightarrow & \mathcal{L}(D(\mathcal{A})) \\
    &t &\mapsto & T^{-1} \widetilde{S}(t) T.
    \end{array}
\end{equation}
The goal of this section is to prove that the closed-loop operator $-\mathcal{A}+\mathcal{I}_\nu F$ we have designed indeed generates the semigroup $S(t), t\geq 0$, which will imply that the closed-loop system \eqref{sys1virt} with feedback law defined by \eqref{Fcoeff} is well-posed.

Let us begin by giving a more precise characterization of its domain $D_F \subset D(\mathcal{A})$.} We start with the following lemma:
\begin{lmm}\label{lemm-A+BK-diag}
The operator $-\mathcal{A}+\mathcal{I}_\nu F$ admits a Riesz basis of eigenvectors in $D(\mathcal{A})$, given by
\begin{equation}
    h_p:=T^{-1} \frac{\widetilde{f}_p}{\widetilde{\mu}_p}, \quad \forall p \in \Z,
\end{equation}
with corresponding eigenvalues $({\color{black}-\widetilde{\mu}_p})_{p\in \Z}$.
\end{lmm}
\begin{proof}
It is clear that the normalized family $(\widetilde{f}_p/\widetilde{\mu}_p)_{p\in \Z}$ is a Riesz basis of $D(\widetilde{\mathcal{A}})$, so that, applying the isomorphism $T^{-1}$, $(h_p)_{p\in\Z}$ is a Riesz basis of $D(\mathcal{A})$.
Let us now show that
\begin{equation}
    h_{p}\in D_F, \quad \forall p\in \Z.
\end{equation}
Let us denote the following decomposition along the orthonormal basis $(f_n)_{n\in \Z}$:
\begin{equation}\label{hp-reg}
h_{p}=\sum\limits_{n\in\mathbb{Z}}a_{n,p} f_{n}, \quad (\mu_n a_{n,p})_{n\in \Z}\in \ell^2(\Z), \  \forall p \in \Z.
\end{equation}
From the definitions of $T$ and $h_{p}$,
\begin{equation}
\label{decompftildeg}
\widetilde{f}_{p}=\widetilde{\mu}_p Th_p=\sum\limits_{n\in\mathbb{Z}}\widetilde{\mu}_p a_{n,p} g_{n}, \quad \forall p \in \Z,
\end{equation}
which implies, from {\color{black}\eqref{gndecomp} and} \eqref{coeffexpression},
\begin{equation}
\widetilde{f}_{p}=-\sum\limits_{n\in\mathbb{Z}}\widetilde{\mu}_p a_{n,p} \langle f_{n},F\rangle \sum\limits_{k\in\mathbb{Z}}\frac{\langle \mathcal{I}_{\nu},\widetilde{\phi}_{k}\rangle}{\widetilde{\mu}_{k}-\mu_{n}}\widetilde{f}_{k}, \quad \forall p \in \Z.
\end{equation}
Now, as in Subsection \ref{deuxeq}, consider the truncations along the Riesz basis $(g_n/\langle f_n, F \rangle)_{n\in \Z}$,
\begin{equation}
    \widetilde{f}_p^{(N)}:= \sum_{n=-N}^N \widetilde{\mu}_p a_{n,p}\langle f_n, F \rangle  \frac{g_n}{\langle f_n, F \rangle}, \quad \forall N \in \NN^\ast, \ \forall p \in \Z,
\end{equation}
so that
\begin{equation}
\begin{split}
\langle\widetilde{f}_{p}^{(N)},\widetilde{\phi}_{m}\rangle=&-\widetilde{\mu}_p\left\langle\sum\limits_{n=-N}^N a_{n,p} \langle f_{n},F\rangle \sum\limits_{k\in\mathbb{Z}}\frac{\langle \mathcal{I}_{\nu},\widetilde{\phi}_{k}\rangle}{\widetilde{\mu}_{k}-\mu_{n}}\widetilde{f}_{k},\widetilde{\phi}_{m}\right\rangle \\
=&-\widetilde{\mu}_p\left\langle\sum\limits_{k\in\mathbb{Z}} \langle \mathcal{I}_{\nu},\widetilde{\phi}_{k}\rangle\left(\sum\limits_{n=-N}^N  \frac{a_{n,p} \langle f_{n},F\rangle}{\widetilde{\mu}_{k}-\mu_{n}}\right)\widetilde{f}_{k},\widetilde{\phi}_{m}\right\rangle \\
=& {\color{black}-\widetilde{\mu}_p\langle \mathcal{I}_{\nu},\widetilde{\phi}_{m}\rangle \left\langle \sum\limits_{n=-N}^N  \frac{a_{n,p} \langle f_{n},F\rangle}{\widetilde{\mu}_{m}-\mu_{n}}\widetilde{f}_{m},\widetilde{\phi}_{m}\right\rangle} \\
=& -\widetilde{\mu}_p\langle \mathcal{I}_{\nu},\widetilde{\phi}_{m}\rangle \left\langle\sum\limits_{k\in\mathbb{Z}} \left(\sum\limits_{n=-N}^N  \frac{a_{n,p} \langle f_{n},F\rangle}{\widetilde{\mu}_{k}-\mu_{n}}\right)\widetilde{f}_{k},\widetilde{\phi}_{m}\right\rangle \\
=&{\color{black} -\widetilde{\mu}_p\langle \mathcal{I}_{\nu},\widetilde{\phi}_{m}\rangle \left\langle \sum\limits_{n=-N}^N  a_{n,p} \langle f_{n},F\rangle  \sum\limits_{k\in\mathbb{Z}} \frac{\widetilde{f}_{k}}{\widetilde{\mu}_{k}-\mu_{n}},\widetilde{\phi}_{m}\right\rangle} \\
=&{\color{black} -\widetilde{\mu}_p\langle \mathcal{I}_{\nu},\widetilde{\phi}_{m}\rangle \left\langle \sum\limits_{n=-N}^N  a_{n,p} \langle f_{n},F\rangle  k_n,\widetilde{\phi}_{m}\right\rangle} \\
=&-\widetilde{\mu}_p\langle \mathcal{I}_{\nu},\widetilde{\phi}_{m}\rangle \left\langle\sum\limits_{n=-N}^N a_{n,p} \langle f_{n},F\rangle \widetilde{\tau}_{\mathcal{A}}^{-1}\frac{f_{n}}{f_{n,1}(0)},\widetilde{\phi}_{m}\right\rangle,\\
& \quad \forall p\in \Z, \ \forall m \in \Z, \ \forall N \in\NN^\ast,
\end{split}
\end{equation}
where we used the biorthogonality of the families $(\widetilde{f}_p)_{p\in \Z}$ and {\color{black}$(\widetilde{\phi}_p)_{p\in \Z}$}, and the relation \eqref{fn in tildefp}. Using again the aforementioned biorthogonality, and the convergences
\begin{equation}\begin{split}
    \widetilde{f}_p^{(N)} &\xrightarrow[N\to \infty]{D(\widetilde{\mathcal{A}})} \widetilde{f}_p, \\
    \sum\limits_{n=-N}^N a_{n,p} \langle f_{n},F\rangle \widetilde{\tau}_{\mathcal{A}}^{-1}\frac{f_{n}}{f_{n,1}(0)} &\xrightarrow[N\to\infty]{L^2} \sum\limits_{n\in \Z}a_{n,p} \langle f_{n},F\rangle \widetilde{\tau}_{\mathcal{A}}^{-1}\frac{f_{n}}{f_{n,1}(0)},
    \end{split}
\end{equation}
given by \eqref{Igrowth},  \eqref{Fcoeff} {\color{black}and \eqref{hp-reg}},
we get
\begin{equation}
\left\langle\frac{\widetilde{f}_{p}}{\langle \mathcal{I}_{\nu},\widetilde{\phi}_m\rangle},\widetilde{\phi}_{m}\right\rangle
=\left\langle\frac{\widetilde{f}_{p}}{\langle \mathcal{I}_{\nu},\widetilde{\phi}_p\rangle},\widetilde{\phi}_{m}\right\rangle
=- \left \langle\widetilde{\mu}_p \sum\limits_{n\in\mathbb{Z}}a_{n,p} \langle f_{n},F \rangle \widetilde{\tau}_\mathcal{A}^{-1}\frac{f_{n}}{f_{n,1}(0)},\widetilde{\phi}_{m}\right\rangle.
\end{equation}
\textcolor{black}{Note that the first equality holds as $\left\langle\widetilde{f}_{p}/\langle \mathcal{I}_{\nu},\widetilde{\phi}_m\rangle,\widetilde{\phi}_{m}\right\rangle=0$ whenever $p\neq m$. Therefore}, by property of the Riesz basis $(\widetilde{\phi}_m)_{m\in \Z}$, and by continuity and invertibility of $\widetilde{\tau}_\mathcal{A}$,
\begin{equation}    \widetilde{\tau}_{A} \frac{\widetilde{f}_{p}}{\langle \mathcal{I}_{\nu},\widetilde{\phi}_p\rangle}=-\widetilde{\mu}_p\sum\limits_{n\in\mathbb{Z}}a_{n,p} \langle f_{n},F \rangle \frac{f_{n}}{f_{n,1}(0)}.
\end{equation}
Using the expression of $\widetilde{\tau}_\mathcal{A}$ given in \eqref{tautilde}, we finally get
\begin{equation}
\overline{\widetilde{\phi}_{p,1}(0)}(1-e^{-2\mu L})\frac{\widetilde{f}_{p}}{\langle \mathcal{I}_{\nu},\widetilde{\phi}_p\rangle}= -\widetilde{\mu}_p\sum\limits_{n\in\mathbb{Z}}a_{n,p} \langle f_{n},F \rangle \frac{f_{n}}{f_{n,1}(0)},
\end{equation}
so that, by property of the orthonormal basis $(f_n)_{n\in \Z}$,
\begin{equation}\label{anp-expression}
a_{n,p}=-\frac{\overline{\widetilde{\phi}_{p,1}(0)}(1-e^{-2\mu L})}{\widetilde{\mu}_p \langle \mathcal{I}_{\nu},\widetilde{\phi}_{p}\rangle}\frac{f_{n,1}(0)}{\langle f_{n}, F\rangle}\langle\widetilde{f}_{p}, f_{n}\rangle.
\end{equation}

Let us now compute $\langle h_p, F \rangle$ with truncations of the $h_p$:
   \begin{equation*}
       h_{p}^{(N)}:=\sum\limits_{n=-N}^{N}a_{n,p} f_{n}.
   \end{equation*}
   \begin{equation}
   \label{expr0}
   \begin{split}
\langle h_{p}^{(N)}, F \rangle &= \sum\limits_{n=-N}^{N}a_{n,p}\langle f_{n}, F \rangle\\
&=-\frac{\overline{\widetilde{\phi}_{p,1}(0)}(1-e^{-2\mu L})}{\widetilde{\mu}_p \langle \mathcal{I}_{\nu},\widetilde{\phi}_{p}\rangle}\sum\limits_{n=-N}^{N}
f_{n,1}(0)\langle \widetilde{f}_{p}, f_{n}\rangle.
\end{split}
 \end{equation}
 Now, using Proposition \ref{PropEquiconvergence}, {\color{black}similar to \eqref{281}}, we have,
 \begin{equation}
\sum\limits_{n=-N}^{N}
f_{n,1}(0)\langle  \widetilde{f}_{p}, f_{n}\rangle \xrightarrow[N\to \infty]{} \frac{\widetilde{f}_{p,1}(0)-\widetilde{f}_{p,2}(0)}{2},
 \end{equation}
so that
 \begin{equation}\label{F-on-hp}
\langle h_p, F\rangle=-\frac{\overline{\widetilde{\phi}_{p,1}(0)}(1-e^{-2\mu L})}{\widetilde{\mu}_p\langle \mathcal{I}_{\nu},\widetilde{\phi}_{p}\rangle}\frac{\widetilde{f}_{p,1}(0)-\widetilde{f}_{p,2}(0)}{2}.
 \end{equation}

    Using \eqref{anp-expression} together with \eqref{F-on-hp}, we get
    \begin{equation}
    \begin{split}
       \left\langle-\mathcal{A}h_{p}+ \mathcal{I}_{\nu}\langle h_{p}, F\rangle, f_n \right\rangle=&
       \frac{\overline{\widetilde{\phi}_{p,1}(0)}(1-e^{-2\mu L})}{\widetilde{\mu}_p\langle \mathcal{I}_{\nu},\widetilde{\phi}_{p}\rangle}\frac{f_{n,1}(0)}{\langle f_{n}, F\rangle}\langle\widetilde{f}_{p}, f_{n}\rangle \mu_{n}\\
        & -\langle\mathcal{I}_{\nu},f_n\rangle\frac{\overline{\widetilde{\phi}_{p,1}(0)}(1-e^{-2\mu L})}{\widetilde{\mu}_p\langle \mathcal{I}_{\nu},\widetilde{\phi}_{p}\rangle}\frac{\widetilde{f}_{p,1}(0)-\widetilde{f}_{p,2}(0)}{2} \\
       =&\frac{\overline{\widetilde{\phi}_{p,1}(0)}(1-e^{-2\mu L})}{\widetilde{\mu}_p\langle \mathcal{I}_{\nu},\widetilde{\phi}_{p}\rangle}\left(-\frac{\langle \mathcal{I}_\nu, f_n\rangle}{2\tanh(\mu L)f_{n,1}(0)}\langle\widetilde{f_{p}}, f_{n}\rangle \mu_{n}\right.\\
        &-\left. \langle\mathcal{I}_{\nu},f_n\rangle\frac{\widetilde{f}_{p,1}(0)-\widetilde{f}_{p,2}(0)}{2}\right).
\end{split}
    \end{equation}

Note that, proceeding as for \eqref{fncoeffequation}, we have
\begin{equation}
\begin{split}
\widetilde{\mu}_{p}\langle\widetilde{f}_{p},f_{n}\rangle&=-\overline{\mu_{n}}\langle\widetilde{f}_{p},f_{n}\rangle+(\widetilde{f}_{p,2}\overline{f_{n,2}}(0)-\widetilde{f}_{p,1}(0)\overline{f_{n,1}}(0))\\
&=\mu_{n}\langle\widetilde{f}_{p},f_{n}\rangle-\widetilde{f}_{p,1}(0)f_{n,1}(0)(1-e^{2 \mu L})
\end{split}
\end{equation}
where we used the boundary conditions given by \eqref{domaineA} and \eqref{Atilde} together with \eqref{imaginary-ev}. Thus
\begin{equation}
\langle\widetilde{f}_{p},f_{n}\rangle=\frac{-\widetilde{f}_{p,1}(0)f_{n,1}(0)(1-e^{2 \mu L})}{\widetilde{\mu}_{p}-\mu_{n}}.
\end{equation}
Therefore, using \textcolor{black}{again the boundary conditions given by \eqref{domaineA} and \eqref{Atilde} and} also \eqref{Fcoeff}, we have
\begin{equation}
\begin{split}
 \left\langle -\mathcal{A}h_{p}+  \langle h_{p}, F\rangle \mathcal{I}_{\nu}, f_n \right\rangle = & \frac{\overline{\widetilde{\phi}_{p,1}(0)}(1-e^{2\mu L})}{\widetilde{\mu}_p\langle \mathcal{I}_{\nu},\widetilde{\phi}_{p}\rangle} \left(-\frac{\langle \mathcal{I}_\nu, f_n\rangle}{2\tanh(\mu L)}\frac{-\widetilde{f}_{p,1}(0)(1-e^{2 \mu L})}{\widetilde{\mu}_{p}-\mu_{n}} \mu_{n}\right. \\
        &-\left. \langle\mathcal{I}_{\nu},f_n\rangle\frac{\widetilde{f}_{p,1}(0)-\widetilde{f}_{p,2}(0)}{2}\right) \\
        =& -\frac{\overline{\widetilde{\phi}_{p,1}(0)}(1-e^{2\mu L})}{\widetilde{\mu}_p\langle \mathcal{I}_{\nu},\widetilde{\phi}_{p}\rangle} \langle \mathcal{I}_\nu, f_n\rangle \\
        &\left(\widetilde{f}_{p,1}(0)\frac{1+e^{2\mu L}}2 \frac{\mu_n}{\widetilde{\mu}_{p}-\mu_{n}}+ \widetilde{f}_{p,1}(0)\frac{1+e^{2\mu L}}2 \right) \\
        =& -\frac{\widetilde{f}_{p,1}(0)\overline{\widetilde{\phi}_{p,1}(0)}(1-e^{4\mu L})}{2\langle \mathcal{I}_{\nu},\widetilde{\phi}_{p}\rangle} \frac{\langle \mathcal{I}_\nu, f_n\rangle}{\widetilde{\mu}_p-\mu_n}.
 \end{split}
\end{equation}
This shows, thanks to \eqref{Igrowth} and \eqref{mu_n asymp}, that
\begin{equation}
    (-\mathcal{A}+\mathcal{I}_\nu F) h_p \in D(\mathcal{A}).
\end{equation}

Then we can apply \eqref{OperatorEquality} to the $(h_p)_{p\in \Z}$, thanks to Proposition \ref{opeqprop}:
\begin{equation}
\begin{split}
    T(-\mathcal{A}+\mathcal{I}_\nu F)h_p & = -\widetilde{\mathcal{A}}Th_p \\
    & = -\widetilde{\mathcal{A}} \textcolor{black}{\frac{\widetilde{f}_p}{\widetilde{\mu}_{p}}} \\
    & = -{\color{black}\widetilde{\mu}_p}\textcolor{black}{\frac{\widetilde{f}_p}{\widetilde{\mu}_{p}}},
    \quad \forall p\in \Z,
    \end{split}
\end{equation}
so that
\begin{equation}
    \label{A+BK-eigenvalue-eqn}
    (-\mathcal{A}+\mathcal{I}_\nu F)h_p=-\widetilde{\mu}_p h_p, \quad \forall p \in \Z.
\end{equation}
\end{proof}

We can now prove the following proposition which gives a precise characterization of the elements of $D_F$:
\begin{prpstn}\label{prop-domain-equality}
The domain $D_F$ satisfies the following equality:
\begin{equation}
D_F = T^{-1}D(\widetilde{\mathcal{A}}^2).
\end{equation}
\end{prpstn}
\begin{proof}
Given Lemma \ref{lemm-A+BK-diag}, one clearly has the following characterization of $D_F$:
\begin{equation}\label{DF-in-coeffs}
    D_F=\left\{ f \in D(\mathcal{A}){\color{black}:} \ f = \sum_{n\in \Z} f_p h_p, \ (\widetilde{\mu}_p f_p)_{p\in \Z} \in \ell^2(\Z)\right\}.
\end{equation}
Then, let $\alpha \in D(\widetilde{\mathcal{A}}^2)$, with the following decomposition:
\begin{equation}
    \alpha:=\sum_{p\in \Z} \alpha_p \frac{\widetilde{f}_p}{\widetilde{\mu}_p}, \quad (\widetilde{\mu}_p \alpha_p)_{p\in \Z} \in \ell^2(\Z).
\end{equation}
Then
\begin{equation}
    T^{-1}\alpha:=\sum_{p\in \Z} \alpha_p T^{-1}\frac{\widetilde{f}_p}{\widetilde{\mu}_p}=\sum_{p\in \Z} \alpha_p h_p.
\end{equation}
Now, by construction of $T$, $T^{-1} \alpha \in D(\mathcal{A})$, and as $(\widetilde{\mu}_p \alpha_p)_{p\in \Z} \in \ell^2(\Z)$, it follows from \eqref{DF-in-coeffs} that
\begin{equation*}
    T^{-1}\alpha \in D_F.
\end{equation*}
Hence
\begin{equation}
    T^{-1}D(\widetilde{\mathcal{A}}^2) \subset D_F.
\end{equation}
The converse inclusion is a consequence of the operator equality \eqref{OperatorEquality}. Indeed, let $\alpha \in D_F$, then
\begin{equation*}
    T\alpha \in D(\widetilde{\mathcal{A}}), \quad \widetilde{\mathcal{A}} T \alpha=-T(-\mathcal{A}+\mathcal{I}_\nu F)\alpha \in D(\widetilde{\mathcal{A}}),
\end{equation*}
so that
\begin{equation*}
     T\alpha \in D(\widetilde{\mathcal{A}}^2),
\end{equation*}
hence,
\begin{equation}
     D_F\subset T^{-1}D(\widetilde{\mathcal{A}}^2).
\end{equation}
\end{proof}
Then we have the following result:
\begin{prpstn}
The mapping $S(t)$
defines an exponentially stable $C^0$-semigroup on $D(\mathcal{A})$, and its infinitesimal generator is the unbounded operator $(-\mathcal{A}+\mathcal{I}_\nu F, D_F)$. {\color{black}Moreover, it is real-valued on real-valued functions.}
\end{prpstn}

\begin{proof}
By continuity and invertibility of $T$, \eqref{closed-loop-semigroup} clearly defines a $C^0$-semigroup, and the domain $D_S$ of its infinitesimal generator is clearly $T^{-1}(D(\widetilde{\mathcal{A}}^2))$. Now, Proposition \ref{prop-domain-equality} implies that
\begin{equation*}
    D_S=D_F.
\end{equation*}
Let $\alpha\in D_S$. Then $T\alpha\in D(\widetilde{\mathcal{A}}^2)$ so that, by definition of $S$, the definition of the infinitesimal generator $\widetilde{\mathcal{A}}$ of $\widetilde{S}(t)$, and \eqref{OperatorEquality},
\begin{equation}\label{target-inf-gen}
   \frac{\widetilde{S}(t)T\alpha - T\alpha}t \xrightarrow[t\to 0^+]{D(\widetilde{\mathcal{{A}}})} -\widetilde{\mathcal{A}} T \alpha= T(-\mathcal{A}+\mathcal{I}_\nu F)\alpha.
\end{equation}
Hence, applying the isomorphism $T^{-1}$ to both sides of \eqref{target-inf-gen}, we get
\begin{equation}
    \frac{S(t)\alpha - \alpha}t \xrightarrow[t\to 0^+]{D(\mathcal{{A}})}(-\mathcal{A}+\mathcal{I}_\nu F)\alpha,
\end{equation}
which proves the second part of the proposition.

Let $\alpha\in D(\mathcal{A})$. Then, using the equivalence of $\norm_{D(\widetilde{\mathcal{A}})}$ and the $(H^1)^2$ norm, and estimate \eqref{target-exp-stab-Hp} with $p=1$, we can write
\begin{equation}
\begin{aligned}
    \|S(t)\alpha\|_{D(\mathcal{A})} &= \|T^{-1} \widetilde{S}(t) T \alpha \|_{D(\mathcal{A})} \\
    &\leq \vertiii{T^{-1}} \|\widetilde{S}(t) T \alpha \|_{D(\widetilde{\mathcal{A}})} \\
    &\leq C\vertiii{T^{-1}}  e^{-\frac{3\mu}{4} s} \| T \alpha \|_{D(\widetilde{\mathcal{A}})} \\
    &\leq C\vertiii{T^{-1}} \vertiii{T} e^{-\frac{3\mu}{4} s} \|\alpha \|_{D(\mathcal{A})}, \quad \forall t\geq 0.
    \end{aligned}
\end{equation}
Hence, $S(t), t\geq 0$ is an exponentially stable semi-group.

{\color{black}Finally, it is clear that $F$ defined by \eqref{Fcoeff} is real-valued, as $\mathcal{I}_\nu$ is real-valued and by Proposition \ref{eigenf-prop}. Moreover, it is clear that the semigroup $\widetilde{S}(t)$ is real-valued on real-valued functions, so that by Corollary \ref{Tisomorphism}, for $\alpha\in D(\mathcal{A})$ a real-valued function, $S(t)\alpha $ is real-valued for $t\geq0$.}
 \end{proof}
This ends the proof of Proposition \ref{Prop-feedback-for-virtual-system}.

\section{Further comments}
Finally, let us make some comments and point out some prospects.
\begin{itemize}
    \item \textbf{Choice of target system.} Let us remark here that this is probably the first time that a backstepping approach is applied using a target system \textcolor{black}{where the damping occurs} at the boundary. Indeed, usual applications of the backstepping method seek to transform the control system into a similar system with an additional \textcolor{black}{internal} damping term. For example, the classical example of the heat equation shows how the stability of the heat equation
 \begin{equation}\label{heat-ex}
    \left\{\begin{aligned}
    w_t-w_{xx} & =0 , \\
    w(0)=0 &, \quad
    w(1)=U(t), \\
    \end{aligned}
    \right.
    \end{equation}
    can be enhanced by finding an invertible mapping to the \textcolor{black}{following system}
 \begin{equation}\label{target-heat-ex}
    \left\{\begin{aligned}
    u_t-u_{xx}& =-\lambda u, \\
    u(0) =0 &,  \quad u(1)=0. \\
    \end{aligned}
    \right.
    \end{equation}
    In the case of the water tank, it appears that our computations are simpler and more natural when the target system has boundary damping {\color{black}rather than internal damping. In fact, calculations in the case of internal damping do not seem to lead to an expression for the feedback law. Whether this is purely technical, or profoundly linked to the conservation of mass, or the presence of coupling terms, remains an open question.} At any rate, one of the reasons for which the choice of such a target system is nontrivial is that $\Lambda$ and $J$ do not commute, so that there is no ``simple'' transformation between our target system and a system with internal damping such as
    \begin{equation}
\label{target-example}
\left\{\begin{aligned}
 &\partial_{t}z+\Lambda \partial_x z + \delta(x)J z + \mu z=0,\\
&z_{1}(t,0)=-z_{2}(t,0), \\
&z_{2}(t,L)=-z_{1}(t,L).
 \end{aligned}\right.
\end{equation}
    {\color{black}This is in contrast with the case of a scalar transport equation (see \cite{ZhangRapidStab}), where adding an internal damping and adding a boundary damping in the target system are equivalent.} This means also that ensuring the stability or our target system where the dissipation occurs at the boundaries is not trivial. Here it is guaranteed by the existence of a Lyapunov function of the form \eqref{defV2} (also called basic Lyapunov function, see \cite{BastinCoron2011} and \cite[Section 2.3]{BastinCoron1D} for more details).

    \item \textbf{Relation between $\gamma$ and $\mu$.}  {\color{black}Thanks to Theorem \ref{th1}, it is clear that the higher $\mu$ is, the smaller $\gamma$ should be. However, is it possible to get a quantitative description, or even a sharp estimate, about this relation?

    Clearly, in order to obtain such a result, one needs to deal with three important issues, controllability of the original system,  stability of the target system, and controllability of the target system, which are described by Lemma \ref{lem-fond}, Proposition \ref{expstab-target}, and Lemma \ref{til-lem-fond}  respectively.  The middle one is rather clear, while the other two are more involved, mainly due to the Riesz basis estimates of the perturbed operators (of both the original and target systems).  However, we believe that it is possible to get an explicit value.

    In fact, this would be related to the  stabilization  around uniform steady-states.  As it is  known that the linearized system around  uniform steady-states is neither controllable nor stabilizable, following the idea of the return method, one may ask if it is  possible to stabilize the (nonlinear) system around  uniform steady-state as a limit of the non-uniform steady-states case. This requires further study of the stabilization of the nonlinear system around those states, and to get some uniform estimates.  Furthermore, it would also be interesting to study the limit $\mu\to +\infty$, in order to gain insight on the so called finite-time stabilization problem.}

    \item \textbf{Regularity of the feedback law.} It {\color{black}can be seen from} Section \ref{section-feedback-reg} that the feedback law is not bounded on the state space. In a sense, this is to be expected when dealing with methods that rely on spectral properties of {\color{black} hyperbolic} systems. Indeed, in backstepping, or in Riccati methods, the spectrum is essentially shifted to the left by the addition of the feedback control term, and general results on pole-shifting methods in infinite dimension (\cite{Sun, Rebarber}) seem to indicate that bounded feedback laws can only achieve limited pole-shifting. Indeed, the difference between the initial spectrum $(\rho_k)$ and the shifted spectrum $(\lambda_k)$ should satisfy
    \begin{equation}
        \sum_{k\in \Z}\left|\frac{\rho_k-\lambda_k}{B_k}\right|^2 <\infty
    \end{equation}
    where the $(B_k)$ are the coefficients of the control vector in a normalized basis of eigenvectors of the system.
    One can see that unless the control operator is very singular, the term
    $\rho_k-\lambda_k$ cannot be bounded from below, and thus exponential stabilization cannot be achieved.

    \item \textbf{Well-posedness of the closed-loop system.} In Section \ref{s6} we prove that the closed-loop operator $-\mathcal{A}+\mathcal{I}_\nu F$ is the infinitesimal generator of a contraction semi-group. Intuitively, this seems to be a direct consequence of the equivalence we have built between the closed-loop system and the target system. However, in practice it is not as straightforward: we have to tread carefully regarding the domains of the operators, as is illustrated in Section \ref{TB-from-OpEq}.

    The study of the domain $D_F$ of the closed-loop operator in particular constitutes a key point in our analysis. Indeed, Lemma \ref{lemm-A+BK-diag} essentially states that $T$ ``exchanges'' the eigenvectors of $-\mathcal{A}+\mathcal{I}_\nu F$ with those of $\widetilde{\mathcal{A}}$, which seems obvious when considering the operator equality \eqref{OperatorEquality}--\eqref{Domain}, but actually requires additional work because we know so little \textit{a priori} about the domain $D_F$ on which this equality holds. This in turn gives additional information in the form of the nice characterization of Proposition \ref{prop-domain-equality}, and in the end the well-posedness of the closed-loop system really follows naturally from the properties of the target system.

    This is notably different from previous works on the stabilization of PDEs using the Gramian (or Riccati) method (\cite{Komornik97, Urquiza, Vest}), where the wellposedness of the closed-loop system also has to be established once the feedback law is defined. In the work of Vilmos Komornik (\cite{Komornik97}), the first notion of well-posedness that can be found is that the closed-loop operator is a \textit{dense restriction} of the infinitesimal generator of a semi-group. This has since then been improved, in \cite{Urquiza} using optimal control theory and results on algebraic Riccati equations, and in \cite{Vest}, by some duality method together with what is called ``formal conjugation'', where the author derives some sort of operator equality from a formal Riccati equation.

\end{itemize}

\appendix
\section{Proof of Proposition \ref{eigenf-prop}}\label{appendixA}
\begin{proof}
Note that if we conjugate the eigenvalue equation, using the fact that the $\mu_n$ are all imaginary, we get, for $n\in \Z$:
\begin{equation}\label{conjugate}
    \begin{aligned}
       \partial_x \overline{f_{n,1}}+\frac{\delta}3 \overline{f_{n,2}}&= -\mu_n \overline{f_{n,1}}, \\
         -\partial_x\overline{f_{n,2}}-\frac{\delta}3 \overline{f_{n,1}}&= -\mu_n \overline{f_{n,2}}, \\
    \end{aligned}
\end{equation}
which proves \eqref{imaginary-ev} and the first equality of \eqref{fn-f-n}.
Moreover \eqref{conjugate} can be written
\begin{equation}\label{conjugate2}
       \mathcal{A} \begin{pmatrix}-\overline{f_{n,2}}\\ -\overline{f_{n,1}}\end{pmatrix}= \mu_n \begin{pmatrix}-\overline{f_{n,2}}\\ -\overline{f_{n,1}}\end{pmatrix}.
\end{equation}
Now, as
\begin{equation}
    \left\|\begin{pmatrix}-\overline{f_{n,2}}\\ -\overline{f_{n,1}}\end{pmatrix}\right\|= \|f_n\|
\end{equation}
and, according to \eqref{conjugate2}, these two functions of $L^2(0,L)^2$ $f_n$ and $(-\overline{f_{n,2}},\overline{f_{n,1}})^T$ are solutions to the same eigenvalue problem. Hence
\begin{equation}\label{plusoumoins}
    -\overline{f_{n,2}} = f_{n,1} 
    \textrm{ or } 
      \overline{f_{n,2}} = f_{n,1} .
\end{equation}
Now let us recall that, by \eqref{varphi-es} and \eqref{R-es}, the $f_n$ are $L^\infty$-close to the $E_n$, which satisfy the first relation of \eqref{plusoumoins}. So for a small enough $\gamma$, we have
\begin{equation}
       \overline{f_{n,1}} = -f_{n,2},
       \label{moins}
\end{equation}
which proves \eqref{fn-f-n}.
\end{proof}
\section{Proof of Proposition \ref{EquiConv}}\label{appendixB}
\begin{proof}[Proof of Proposition \ref{EquiConv}]
As \cite[Theorem 2]{Komornik1984} deals with scalar second-order equation, we first define a map $R$ on $L^2(0,L)^2$, gluing the two components to form a function of $L^2(0,L)$
and apply the first order operator twice
to recover a second order scalar equation.
For $f\in H^{1}((0,L);\mathbb{C}^{2})$ satisfying $f^{1}(L)=-f^{2}(L)$ we define $R(f):=\underline f$ on $(0,2L)$ by
\begin{equation}
\underline f= \textcolor{black}{\mathds{1}_{[0,L]}} f_{1} - \textcolor{black}{\mathds{1}_{[L, 2L]}} f_{2}(2L - \cdot)
\label{attachementmap}
\end{equation}
This is a natural mapping, given the boundary condition $f_{1}(L)=-f_{2}(L)$, and
defines an isomorphism between $H^{1}(0,2L)$ and $\left\{f \in H^{1}((0,L);\mathbb{C}^{2})| f_{1}(L)=-f_{2}(L)\right\}$.
We extend this definition to $L^{2}$ functions by density, and the resulting map $R$ is, up to a constant, an isometry from $L^2(0,L)^2 \rightarrow L^2(0,2L)$ for their usual scalar products.

Now, notice that the $(f_p)$ are also the eigenfunctions of the operator $-\mathcal{A}^2$, for which the following expression can be derived:
\begin{equation}\label{Asquared}
\begin{aligned}
-\left(\Lambda\partial_{x}+\delta(x) J\right)^{2}&= -\partial_{x}^{2}-\delta(x)^2J^2 - \Lambda \partial_x (\delta(x) J) - \delta(x) J \Lambda \partial_x\\
&=-\partial_{x}^{2}-\delta(x)^2J^2 - \delta^\prime (x)\Lambda J- \delta(x) \Lambda J \partial_x -\delta(x) J \Lambda \partial_x\\
&=-\partial_{x}^{2}-\delta(x)^2J^2 - \delta^\prime (x)\Lambda J,
\end{aligned}
\end{equation}
the last equality being obtained thanks to the relation
\begin{equation}\label{anticommutation}
\begin{aligned}
\Lambda J + J \Lambda&= \begin{pmatrix}
1 & 0 \\ 0 & -1
\end{pmatrix}\begin{pmatrix}
0 & 1/3 \\ -1/3 & 0
\end{pmatrix} + \begin{pmatrix}
0 & 1/3 \\ -1/3 & 0
\end{pmatrix} \begin{pmatrix}
1 & 0 \\ 0 & -1
\end{pmatrix} \\
&= \begin{pmatrix}
0 & 1/3 \\ 1/3 & 0
\end{pmatrix} + \begin{pmatrix}
0 & -1/3 \\ -1/3 & 0
\end{pmatrix} \\
&= 0.
\end{aligned}
\end{equation}
Hence,
\begin{equation}\label{Asquared2}
-\mathcal{A}^2=-\partial_{x}^{2} + \frac13 \begin{pmatrix}
\delta(x)^2/3 & -\delta^\prime (x) \\
-\delta^\prime(x) & \delta(x)^2/3
\end{pmatrix},
\end{equation}
and
\begin{equation*}-\mathcal{A}^2 R^{-1} \underline{f_k}= -\mu_k^2 R^{-1} \underline{f_k}\end{equation*}
i.e.
\begin{equation}\label{glued1}-R\mathcal{A}^2 R^{-1} \underline{f_k}= -\mu_k^2\underline{f_k}.
\end{equation}
Now, it is clear from \eqref{attachementmap} that
\begin{equation}\label{diffcompatibility}
R \Lambda\partial_x \alpha = \partial_x ( R \alpha), \quad \forall \alpha \in H^1, \ \alpha_1(L)=-\alpha_2(L).
\end{equation}
Also, for $a\in L^2(0,L)$, we define
\begin{equation}
    \underline{a}:=R\begin{pmatrix} a \\ -a \end{pmatrix}.
\end{equation}
We have
\begin{equation}\label{multiplygluing}
    R(af)=\underline{a}\underline{f}, \quad \forall f \in L^2(0,L)^2.
\end{equation}
Finally,
\begin{equation}\label{swapgluing}
    R\left( \begin{pmatrix}
    0 & 1 \\ 1 & 0
    \end{pmatrix} f \right)=\textcolor{black}{\mathds{1}}_{[0,L]} f_{2} - \textcolor{black}{\mathds{1}}_{[L, 2L]} f_{1}(2L - \cdot)=-\underline{f}(2L- \cdot).
\end{equation}
 From \eqref{diffcompatibility},\eqref{multiplygluing}, and \eqref{swapgluing}, \eqref{glued1} becomes
\begin{equation}
    \label{glued2}
    -\partial^2_x \underline{f_k}+ \frac{\underline{\delta}^2}9 \underline{f_k} + \frac{\underline{(\delta^\prime)}}3  \underline{f_k}(2L-\cdot) = -\mu_k^2  \underline{f_k}
\end{equation}
so that, using \eqref{swapgluing},
\begin{equation}\label{involutionisconjugation}
\underline{f_k}(2L-x)=\overline{\underline{f_k}}(x),\quad \forall x \in (0, 2L),
\end{equation}
and we finally get
\begin{equation}
    \label{glued3}
    -\partial^2_x \underline{f_k}+ \frac{\underline{\delta}^2}9  \underline{f_k} + \frac{\underline{(\delta^\prime)}}3 \overline{\underline{f_k}}= -\mu_k^2  \underline{f_k}.
\end{equation}

Thus $( \underline{f_{p}})_{p\in\mathbb{Z}}$ is a family of eigenvectors of the operator $\mathcal{L}$ defined by
\begin{equation}
\label{defmathcalL1}
\mathcal{L}u=-\partial_{x}^{2}u+\frac{1}{9}\underline\delta^{2} u+\frac{1}{3}\textcolor{black}{\underline{(\delta^\prime)}}\overline{u},
\end{equation}
with eigenvalues $-\lambda_{p}^{2}$ which, from \eqref{imaginary-ev}, are real and nonnegative.
Besides, $( \underline{f_{p}})_{p\in\mathbb{Z}}$ is still an orthonormal basis of $L^{2}(0,2L)$. Observe that $(\underline E_{p})_{p\in\mathbb{Z}}=(e^{i\pi px/L})_{p\in\mathbb{Z}}$
and is an orthonormal basis of $L^{2}(0,2L)$ and a family of eigenvectors of the operator
$\mathcal{L}_{0}=-\partial_{x}^{2}$ with associated eigenvalues $(\pi^{2} p^{2}/L^{2})_{p\in\mathbb{Z}}$. Note that $\mathcal{L}$ can be written as
\begin{equation}
\label{Lsimple}
\mathcal{L}u=-\partial_{x}^{2}u-q u-q_{1}u(2L-\cdot)
\end{equation}
where $q$ and $q_{1}$ are both $L^{1}$ (and in fact $C^{\infty}$) functions on $(0,2L)$.
We now extend periodically the functions $\underline{f_{p}}$, $\underline{E_{p}}$, $q$ and $q_{1}$ on $(-L,3L)$ as follows :
\begin{equation}
\begin{aligned}
\underline{f_{p}}(\cdot)=\underline{f_{p}}(\cdot+2L),\text{ on }(-L,0),\\
\underline{f_{p}}(\cdot)=\underline{f_{p}}(\cdot-2L),\text{ on }(2L,3L),
\end{aligned}
\end{equation}
and similarly for $\underline{E_{p}}$, $q$ and $q_{1}$. As $\underline{f_{p}}(2L)=\underline{f_{p}}(0)$ from \eqref{attachementmap} and \eqref{domaineA}, the functions $\underline{f_{p}}$ thus constructed are continuous on $[-2L,2L]$. Besides,
\begin{equation}
\lVert q\rVert_{L^{1}(-L,3L)}=2\lVert q\rVert_{L^{1}(0,2L)},\text{  }\lVert q_{1}\rVert_{L^{1}(-L,3L)}=2\lVert q_{1}\rVert_{L^{1}(0,2L)}.
\end{equation}

Then, let a compact interval $K\subset[0,L)$ and consider the restriction of $f_{p}$ to $K$, one can easily see from \eqref{attachementmap} that its gluing map corresponds to the restriction of $ \underline{f_{p}}$ on $\{x\in [0,2L)|x\in K \text{ or }2L-x\in K\}$ which is
a compact set symmetrical with respect to $L$ and with two connected components. This means that, in order to end the proof of Proposition \ref{KomornikSchrodinger}, it suffices to show the same type of result as \cite[Theorem 2]{Komornik1984} applied to $( \underline{f_{p}})_{p\in\mathbb{Z}}$ and $(\underline E_{p})_{p\in\mathbb{Z}}$,
but on compacts of \textcolor{black}{$(-L,3L)$}, symmetrical with respect to $L$ with two connected components only.
Observe that \cite[Theorem 2]{Komornik1984} is only given for compact interval but is also true for compacts with finite number of connected components. Now, the two only differences between our case \textcolor{black}{ans} \cite[Theorem 2]{Komornik1984} are that $\mathcal{L}$ has a non-local term which is the third term $q_{1}u(2L-\cdot)$ in the right-hand side of \eqref{Lsimple}, and that the $\underline{f_{p}}$ are continuous but their derivatives are not always continuous and have \textcolor{black}{discontinuities} on $\mathcal{D}:=\{x=0,x=2L\}$.
Observe, however, that in the proof of \cite[Theorem 2]{Komornik1984}, the fact that $ \underline{f_{p}}$ are eigenvectors of $\mathcal{L}$ is only used through the Titchmarsh formula \cite{Komornik1984}, and
note that a generalized Titchmarsh formula still holds for this operator and we have, using \eqref{defmathcalL1}--\eqref{Lsimple} and two integration by parts, for $p\in \Z$, and $x\in\textcolor{black}{(-L,3L)}$, $t\in\textcolor{black}{(0,\min(|3L-x|,|x+L|))}$
\begin{equation}\label{titchmarsh}
\begin{aligned}
 \underline{f_{p}}(x+t)&+ \underline{f_{p}}(x-t)=2f_{p}(x)\cos(\sqrt{-\mu_{p}^{2}}t)\\
&+\int_{x-t}^{x+t}q(\xi) \underline{f_{p}}(\xi)\frac{\sin(\sqrt{-\mu_{p}^{2}}(t-|x-\xi|))}{\sqrt{-\mu_{p}^{2}}}d\xi\\
&+\int_{x-t}^{x+t}q_{1}(\xi) \underline{f_{p}}(2L-\xi)\frac{\sin(\sqrt{-\mu_{p}^{2}}(t-|x-\xi|))}{\sqrt{-\mu_{p}^{2}}}d\xi\\
&+\sum\limits_{x_{1}\in(x-t,x+t)\cap\mathcal{D}}\frac{\sin(\sqrt{-\mu_{p}^{2}}(t-|x-x_{1}|))}{\sqrt{-\mu_{p}^{2}}}(\underline{f_{p}}'(x_{1}^{+})-\underline{f_{p}}'(x_{1}^{-})).
\end{aligned}
\end{equation}
As expected, the two differences with \cite[Theorem 2]{Komornik1984} are now translated in the appearance of the third and fourth terms that do not appear in the case of \cite[Theorem 2]{Komornik1984}.
Now observe that in the proof of \cite[Theorem 2]{Komornik1984}, the second term of the right-hand side is each time bounded using the $L^{\infty}$ norm of $ \underline{f_{p}}$
and $Q=\lVert q\rVert_{L^{1}(0,L)}$.
Thus, to adapt the proof of \cite[Theorem 2]{Komornik1984}, all we have to do is to provide the same type of bounds at each step for the third and fourth terms.
As for any compact $K\subset\textcolor{black}{(-L,3L)}$ symmetrical with respect to $L$,
$\lVert  \underline{f_{p}} \rVert_{L^{\infty}(K)}=\lVert  \underline{f_{p}}(2L-\cdot) \rVert_{L^{\infty}(K)}$, the
same bounds hold in our case for the third term
replacing $Q$ by $\lVert q\rVert_{L^{1}(-L,3L)}+\lVert q_{1}\rVert_{L^{1}(-L,3L)}$.\footnote{For the sake of rigor let us note that when $K=[a,b]$ is symmetrical with respect to $L$, $K_{R}=[a-R,b+R]$ is also symmetrical with respect to $L$.}
To deal with the fourth term, we need to study the jump \textcolor{black}{discontinuities} $\underline{f_{p}}'(x_{1}^{+})-\underline{f_{p}}'(x_{1}^{-})$.
From the definition of $\underline{f_{p}}$, one has
\begin{equation}
\begin{aligned}
&\partial_{x}\underline{f_{p}}-\frac{\underline{\delta}}{3}\underline{f_{p}}(2L-x)+\mu_{p}\underline{f_{p}}=0\text{ for }\textcolor{black}{x\in(0,L)\cup(2L,3L)},\\
&\partial_{x}\underline{f_{p}}+\frac{\underline{\delta}}{3}\underline{f_{p}}(2L-x)+\mu_{p}\underline{f_{p}}=0\text{ for }\textcolor{black}{x\in(-L,0)\cup(L,2L)}.
\end{aligned}
\end{equation}
Thus for $x_{1}\in\mathcal{D}$, one has
\begin{equation}
\underline{f_{p}}'(x_{1}^{+})-\underline{f_{p}}'(x_{1}^{-})=2\underline{\delta}(x_{1})\underline{f_{p}}(x_{1})
\end{equation}
which implies, for $x_{1}\in K$, that
\begin{equation}
\left|\underline{f_{p}}'(x_{1}^{+})-\underline{f_{p}}'(x_{1}^{-})\right|\leq C\lVert \underline{f_{p}}\rVert_{L^{\infty}(K)},
\end{equation}
where $C$ is a constant independent of $p$. With this in mind, and noting that for any compact $K$, $K\cap\mathcal{D}$ has a finite cardinal $\mathcal{N}_{K}\textcolor{black}{\leq 2}$ depending only on $K$, we can bound the fourth term of \eqref{titchmarsh} as in \cite[Theorem 2]{Komornik1984}. More precisely we have :
\begin{enumerate}
    \item For Lemma 1 of \cite{Komornik1984}, with $R>0$ such that $x+2R\in K $, integrating the fourth term on $(0,R)$ gives
\begin{equation}\label{boundtitch}
\begin{aligned}
&\left|\int_{0}^{R}\sum\limits_{x_{1}\in(x-t,x+t)\cap\mathcal{D}}\frac{\sin(\sqrt{-\mu_{p}^{2}}(t-|x-x_{1}|))}{\sqrt{-\mu_{p}^{2}}}(\underline{f_{p}}'(x_{1}^{+})-\underline{f_{p}}'(x_{1}^{-}))dt\right|\\
&\leq A C\lVert \underline{f_{p}}\rVert_{L^{\infty}(K)}\int_{0}^{R}\sum\limits_{x_{1}\in(x-t,x+t)\cap\mathcal{D}}|t-|x-x_{1}||dt\\
&\leq A C\lVert \underline{f_{p}}\rVert_{L^{\infty}(K)}\mathcal{N}_{K}R^{2},
\end{aligned}
\end{equation}
where we used in the last line that $x_{1}\in(x-t,x+t)$ and where $A=1$ but corresponds to the constant $A$ in \cite{Komornik1984}. Thus, this bound
is similar to the bound obtained in \cite{Komornik1984} for the second term of \eqref{titchmarsh} replacing $Q$ with $Q=C\mathcal{N}_{K}$. The proof of the first part a) of Theorem 1 in \cite{Komornik1984} follows directly.
\item For the part b) of Theorem 1 in \cite{Komornik1984}, one has, with $x\in K$ and $K_{R}$ is the compact extension of $K$ given by $\{x\in \overline{\mathcal{B}_{y}(R)}| y\in K\}$, \textcolor{black}{where $\mathcal{B}_{y}(R)$ is the ball centered in $y$ of radius $R$,}
\begin{equation}\label{boundtitch2}
\begin{aligned}
&\left|\int_{0}^{R} \cos(\mu t)\sum\limits_{x_{1}\in(x-t,x+t)\cap\mathcal{D}}\frac{\sin(\sqrt{-\mu_{p}^{2}}(t-|x-x_{1}|))}{\sqrt{-\mu_{p}^{2}}}(f_{p}'(x_{1}^{+})-f_{p}'(x_{1}^{-}))dt\right|\\
&\leq C\lVert \underline{f_{p}}\rVert_{L^{\infty}(K_{R})}\int_{0}^{R} \sum\limits_{x_{1}\in(x-t,x+t)\cap\mathcal{D}}\left|\frac{\sin(\sqrt{-\mu_{p}^{2}}(t-|x-x_{1}|))}{\sqrt{-\mu_{p}^{2}}}\right|dt
\end{aligned}
\end{equation}
Now, for $t\in[0,R]$ and $x_{1}\in(x-t,x+t)$, we have $(t-|x-x_{1}|)\in[0,R]$, which implies that
\begin{equation}
\left|\frac{\sin(\sqrt{-\mu_{p}^{2}}(t-|x-x_{1}|))}{\sqrt{-\mu_{p}^{2}}}\right|\leq 2\frac{R+1}{1+|\sqrt{-\mu_{p}^{2}}|}.
\end{equation}
Thus, \eqref{boundtitch2} becomes
\begin{equation}\label{boundtitch3}
\begin{aligned}
&\left|\int_{0}^{R} \cos(\mu t)\sum\limits_{x_{1}\in(x-t,x+t)\cap\mathcal{D}}\frac{\sin(\sqrt{-\mu_{p}^{2}}(t-|x-x_{1}|))}{\sqrt{-\mu_{p}^{2}}}(f_{p}'(x_{1}^{+})-f_{p}'(x_{1}^{-}))dt\right|\\
&\leq C\lVert \underline{f_{p}}\rVert_{L^{\infty}(K_{R})}2\frac{R+1}{1+|\sqrt{-\mu_{p}^{2}}|}\mathcal{N}_{K}R
\end{aligned}
\end{equation}
and we have once again a bound similar to the bound obtained in \cite{Komornik1984} for the second term of \eqref{titchmarsh} with $Q=C\mathcal{N}_{K}$.
\item Finally, one can do the same with the proof of Theorem 2 by noting that Lemma 3 and 4 still holds identically and that
\begin{equation}
\begin{aligned}
&\int_{0}^{R} \frac{\sin(\mu t)}{\pi t}\sum\limits_{x_{1}\in(x-t,x+t)\cap\mathcal{D}}\frac{\sin(\sqrt{-\mu_{p}^{2}}(t-|x-x_{1}|))}{\sqrt{-\mu_{p}^{2}}}(f_{p}'(x_{1}^{+})-f_{p}'(x_{1}^{-}))dt\\
&=\sum\limits_{\{|x-x_{1}|\leq R\}\cap\mathcal{D}} \int_{|x-x_{1}|}^{R} \frac{\sin(\mu t)}{\pi t}\frac{\sin(\sqrt{-\mu_{p}^{2}}(t-|x-x_{1}|))}{\sqrt{-\mu_{p}^{2}}}(f_{p}'(x_{1}^{+})-f_{p}'(x_{1}^{-}))dt.
\end{aligned}
\end{equation}
\end{enumerate}
Thus, we can overall apply the results of \cite{Komornik1984} by replacing $Q$ with $Q=\textcolor{black}{2}(\lVert q\rVert_{L^{1}(0,2L)}+\lVert q_{1}\rVert_{L^{1}(0,2L)})+C\mathcal{N}_{K}$.
Thus, we still have from \cite[Theorem 2]{Komornik1984}\textcolor{black}{
\begin{equation}
 \lim\limits_{\mu\rightarrow +\infty} \sup\limits_{x\in K}\left|\sum\limits_{|Im(\mu_{p})|<\mu} \langle \underline{f}, \underline{f_{p}}\rangle \underline{f_{p}}(x)-\sum\limits_{|Im(\mu_{p}^{(0)})|<\mu} \langle \underline{f}, \underline{E_{p}}\rangle \underline{E_{p}}(x)\right|=0.
\end{equation}
Choosing now the compact $K=[-L/2,5L/2]\in(-L,3L)$ and symmetrical with respect to $L$, one can see that for any $x\in[0,L]$, $x\in K$ and $2L-x\in K$. Thus}
\begin{equation}
 \lim\limits_{\mu\rightarrow +\infty} \sup\limits_{x\in K}|\sigma_{\mu}(f,x)-p_{\mu}(f,x)|=0.
\end{equation}
This ends the proof of Lemma \ref{EquiConv}
\end{proof}

\section{
Expression of the feedback law in the physical coordinates
}
\label{appendixC}
In this appendix, we show that \eqref{F} implies \eqref{Fcoeff} after the change of variables described in Subsection \ref{transfosubsection}.
Let $(h_{n},v_{n})_{n\in\mathbb{Z}}$ be the Riesz basis of $(L^{2})^{2}$ of eigenvectors associated to the problem \eqref{sv-lin}--\eqref{bc-sv}.
\textcolor{black}{We first construct an orthonormal basis $(f_{n})$ of $(L^{2})^{2}$ of eigenvectors of $\mathcal{A}$ deduced from $(h_{n},v_{n})_{n\in\mathbb{Z}}$ by applying the change of variables described in Subsection \ref{transfosubsection}. Looking at \eqref{diffeo0}--\eqref{diffeo} this family $(f_{n})$ is given by}
\begin{equation}\label{app-def-fn}
    f_{n}=\exp\left(\int_{0}^{z}\delta(x)dx\right)S(r(\cdot))\begin{pmatrix}h_{n}\\v_{n}\end{pmatrix}(r(\cdot)),
\end{equation}
where
\begin{equation}\label{def-S-var}
\begin{aligned}
    S=\begin{pmatrix}
        \sqrt{\frac{1}{H^\gamma}}&1\\
        -\sqrt{\frac{1}{H^\gamma}}&1
       \end{pmatrix}\\
       \end{aligned}
\end{equation}
and $r$ is the bijection from $[0,L]$ to $[0,L]$ defined by
\begin{equation}
r^{-1} :x\rightarrow  \frac{L}{L_{\gamma}}\frac{2}{\gamma}\left(\sqrt{1+\frac{\gamma L}{2}}\right)-\left(\sqrt{1-(\gamma x- \frac{ L}{2}})\right).
\end{equation}
As the transformations
given in \eqref{diffeo0}--\eqref{scal}
define a diffeomorphism, 
$(f_{n})$ is indeed a Riesz basis of $(L^{2})^{2}$ of eigenvectors of $\mathcal{A}$ given by \eqref{A}.

Thus, from Section \ref{s3}--\ref{s5}, the exponential stability holds provided that 
$F\in\dist$
satisfies \eqref{Fcoeff}. We are now going to show that \eqref{Fcoeff} indeed holds, using the hypothesis \eqref{F} on $(h_{n},v_{n})$ and \eqref{control-form}.

Let  $F\in \dist$ be a feedback law. From \eqref{control-form} and \eqref{diffeo0}--\eqref{sys01}, one has
\begin{equation}\label{uapp}
\begin{aligned}
    u_{\gamma}(t)&=\left(\langle \zeta(t,\cdot), F\rangle+\zeta_0\langle f_{0}, F\rangle\right),
    \end{aligned}
\end{equation}
and
\begin{equation}
    \begin{aligned}
\langle \mathbb{\zeta}(t,\cdot), F\rangle=\left\langle \exp\left(\int_{0}^{z}\delta(x)dx\right)S(r(\cdot))\begin{pmatrix}h\\v\end{pmatrix}(L_{\gamma}t/L,r(\cdot)), F \right\rangle.
    \end{aligned}
\end{equation}
Thus there exists $F_{1}\in \mathcal{D}_\gamma^\prime$ such that
       \begin{equation}\label{egal}
          \langle \mathbb{\zeta}(L/L_{\gamma}t,\cdot),F\rangle= \left\langle \begin{pmatrix}h\\v\end{pmatrix}(t,\cdot), F_1\right\rangle,
       \end{equation}
       and therefore from \eqref{app-def-fn}
\begin{equation}
\langle f_{n}, F\rangle=\left\langle \begin{pmatrix}h_{n}\\v_{n}\end{pmatrix}, F_{1}\right\rangle.
\end{equation}
Now if $F_{1}$ satisfies \eqref{F}, one has using \eqref{def-S-var} and \eqref{app-def-fn}
\begin{equation}
    \langle f_{n},F\rangle=-2\tanh(\mu L) \frac{(f_{n,1}(0))^2}{\int_{0}^{L} f_{n,1}(x)+f_{n,2}(x)dx}, \quad \forall n \in \Z ,
\end{equation}
which is exactly \eqref{Fcoeff}, noting that $f_{n,1}$ is real from \eqref{fn-f-n}.
Finally, as we restrict ourselves to solutions of the system \eqref{sys1virt}, \eqref{cond-0} with $\zeta_0(0)=0$ (see \eqref{integrator-starts-at-0}), and from \eqref{egal},
the control under the form \eqref{F0-new}--\eqref{F0-new2} corresponds exactly to \eqref{uapp}.

Finally, to get the final exponential decay rate for \eqref{sv-lin}--\eqref{bc-sv} with feedback $F_1^\gamma$, let us recall that we operated a scaling in time in \eqref{scal}, so the decay rate is
\begin{equation}
    \frac34 \mu \frac{L}{L^\gamma} \xrightarrow[\gamma\to 0]{} \frac{3}{8} \mu.
\end{equation}
In particular, for $\gamma>0$ small enough,
\begin{equation}
     \frac34 \mu \frac{L}{L^\gamma} \geq \frac{\mu}4,
\end{equation}
which gives us the decay rate of Theorem \ref{th1}.
\begin{rmrk}
Note that the conditions \eqref{F} and \eqref{Fcoeff} remain the same when the basis under consideration is renormalized.
\end{rmrk}
\section*{Acknowledgements}
The authors would like to thank the ANR project Finite4SoS (No.ANR 15-CE23-0007). Amaury Hayat was financially supported by IPEF,
Christophe Zhang was financially supported by the French Corps des Mines, and the Chair for Applied Analysis (Alexander von Humboldt professorship) of the Friedrich Alexander Universit\"{a}t Erlangen-N\"{u}rnberg.
\bibliographystyle{plain}
\bibliography{Bib.bib}
\end{document}